\documentclass[letterpaper, 11pt,  reqno]{amsart}

\usepackage{amsmath,amssymb,amscd,amsthm,amsxtra, esint}
\usepackage{tikz} 
\usepackage{color}
\usepackage{hyperref}

\usepackage{cases}

\usepackage{todonotes}

\usepackage[left=32mm, right=32mm, 
bottom=27mm]{geometry}

\usepackage{tikz}

\usepackage{tikz-cd}

\usepackage{marginnote}
\usepackage{scalerel} 

\usetikzlibrary{shapes.misc}
\usetikzlibrary{shapes.symbols}
\usetikzlibrary{decorations}
\usetikzlibrary{decorations.markings}

\tikzset{
    dot/.style={circle,fill=black,draw=black,inner sep=0pt,minimum size=0.5mm},
    >=stealth,
    }

\tikzset{
    dot2/.style={circle,fill=black,draw=black,inner sep=0pt,minimum size=0.2mm},
    >=stealth,
    }

\tikzset{
    ddot/.style={circle,fill=white,draw=black,inner sep=0pt,minimum size=0.8mm},
    >=stealth,
    }

\tikzset{decision/.style={ 
        draw,
        diamond,
        aspect=1.5
    }}

\tikzset{dia2/.style
={diamond,fill=white,draw=black,inner sep=0pt,minimum size=1mm},
    >=stealth,
    }

\tikzset{dia/.style
={star,fill=black,draw=black,inner sep=0pt,minimum size=1mm},
    >=stealth,
    }

\tikzset{dia/.style
={diamond,fill=black,draw=black,inner sep=0pt,minimum size=1.3mm},
    >=stealth,
    }

\makeatletter
\def\DeclareSymbol#1#2#3{\xsavebox{#1}{\tikz[baseline=#2,scale=0.15]{#3}}}
\def\<#1>{\xusebox{#1}}
\makeatother

\newcommand{\pl}{\mathbin{\scaleobj{0.7}{\tikz \draw (0,0) node[shape=circle,draw,inner sep=0pt,minimum size=8.5pt] {\scriptsize $<$};}}}

\usetikzlibrary{decorations.pathmorphing, patterns,shapes}

\newcommand{\llpvariable}[2]{
    \tikz[baseline=0.01em]{\draw[line width=#1pt] (0, 0) -- (0, 0.6em); \draw[fill=#2, line width=#1pt] (0.1pt, 0.6em) circle (0.3ex);}
}

\newcommand{\llp}{
  \ensuremath{\llpvariable{0.2}{white}}
}

\newcommand{\llpg}{
  \ensuremath{\llpvariable{0.8}{black}}
}

\usetikzlibrary{mindmap,backgrounds,trees,arrows,decorations.pathmorphing,decorations.pathreplacing,positioning,shapes.geometric,shapes.misc,decorations.markings,decorations.fractals,calc,patterns}

\tikzset{>=stealth',
         cvertex/.style={circle,draw=black,inner sep=1pt,outer sep=3pt},
         vertex/.style={circle,fill=black,inner sep=1pt,outer sep=3pt},
         star/.style={circle,fill=yellow,inner sep=0.75pt,outer sep=0.75pt},
         tvertex/.style={inner sep=1pt,font=\scriptsize},
         gap/.style={inner sep=0.5pt,fill=white}}
\tikzstyle{mybox} = [draw=black, fill=blue!10, very thick,
    rectangle, rounded corners, inner sep=10pt, inner ysep=20pt]
\tikzstyle{boxtitle} =[fill=blue!50, text=white,rectangle,rounded corners]

\tikzstyle{decision} = [diamond, draw, fill=blue!20,
    text width=4.5em, text badly centered, node distance=2.5cm, inner sep=0pt]
\tikzstyle{block} = [rectangle, draw, fill=blue!20,
    text width=2.7cm, text centered, rounded corners, minimum height=3em]

\tikzstyle{line} = [draw, very thick, color=black!50, -latex']

\tikzstyle{cloud} = [draw, ellipse,fill=red!40, 
node distance=2.5cm,
    minimum height=1em]

\tikzstyle{cloud2} = [draw, ellipse,fill=red!30, text=white,text width=10em, node distance=2.5cm, text centered, minimum height=4em]

\tikzstyle{cloud3} = [draw, ellipse, fill=cyan!30, 
node distance=2.5cm,
    minimum height=3em, 
    text width=1.7cm]

\tikzstyle{cloud4} = [draw, ellipse,fill=orange!70, node distance=2.5cm,
   text width=2.1cm,  
           minimum height=3em]

\tikzstyle{cloud5} = [draw, ellipse,fill=red!20, node distance=2.5cm,
    text centered, 
    text width=3.0cm, 
    minimum height=3em]

\tikzstyle{cloud6} = [draw, ellipse,fill=red!20, node distance=2.5cm,
    text width=1.5cm, 
    text centered, 
    minimum height=3em]

\tikzset{
    position/.style args={#1:#2 from #3}{
        at=(#3.#1), anchor=#1+180, shift=(#1:#2)
    }
}

\DeclareRobustCommand{\gobblefour}[4]{}

\usepackage{tikz}
\usepackage{xcolor}
\usetikzlibrary{calc,intersections,through,backgrounds}
\usetikzlibrary{decorations.pathreplacing,decorations.pathmorphing,decorations.markings,decorations.shapes}
\usetikzlibrary{shapes}
\usetikzlibrary{external}

\makeatletter
\renewcommand{\subjclassname}{%
  \textup{1991} Mathematics Subject Classification}
\@xp\let\csname subjclassname@1991\endcsname \subjclassname
\@namedef{subjclassname@2000}{%
  \textup{2000} Mathematics Subject Classification}
\@namedef{subjclassname@2010}{%
  \textup{2010} Mathematics Subject Classification}
  \@namedef{subjclassname@2020}{%
  \textup{2020} Mathematics Subject Classification}
\makeatother

\allowdisplaybreaks[2]

\usepackage{color}

\definecolor{gr}{rgb}   {0.,   0.69,   0.23 }
\definecolor{bl}{rgb}   {0.,   0.5,   1. }
\definecolor{mg}{rgb}   {0.85,  0.,    0.85}
\definecolor{yl}{rgb}   {0.8,  0.7,   0.}
\definecolor{or}{rgb}  {0.7,0.2,0.2}

\newcommand{\Cc}{\mathcal{C}}
\newcommand{\Bb}{\mathcal{B}}

\newcommand{\CC}{\mathcal{C}}
\newcommand{\CE}{\mathcal{E}}

\usepackage{dsfont}
\newcommand{\IDC}{\mathds{1}}

\usepackage{tikz}
\usepackage{xcolor}
\usetikzlibrary{calc,intersections,through,backgrounds}
\usetikzlibrary{decorations.pathreplacing,decorations.pathmorphing,decorations.markings,decorations.shapes}
\usetikzlibrary{shapes}
\usetikzlibrary{external}

\newtheorem{theorem}{Theorem} [section]

\newtheorem{lemma}[theorem]{Lemma}
\newtheorem{proposition}[theorem]{Proposition}
\newtheorem{remark}[theorem]{Remark}

\newtheorem{definition}[theorem]{Definition}
\newtheorem{corollary}[theorem]{Corollary}

\DeclareMathOperator*{\supp}{supp}

\newcommand{\Z}{\mathbb{Z}}
\newcommand{\R}{\mathbb{R}}
\newcommand{\C}{\mathbb{C}}

\newcommand{\T}{\mathbb{T}}
\renewcommand{\S}{\mathbb{S}}

\newcommand{\deff}{\stackrel{\textup{def}}{=}}

\renewcommand{\S}{\mathcal{S}}

\let\Re=\undefined\DeclareMathOperator*{\Re}{Re}
\let\Im=\undefined\DeclareMathOperator*{\Im}{Im}

\let\P= \undefined
\newcommand{\P}{\mathbf{P}}

\newcommand{\Q}{\mathbf{Q}}
\renewcommand{\AA}{\mathbf{A}}

\newcommand{\E}{\mathbb{E}}

\newcommand{\K}{\mathcal{K}}

\newcommand{\al}{\alpha}

\newcommand{\dl}{\delta}

\newcommand{\too}{\longrightarrow}

\newcommand{\Dl}{\Delta}

\newcommand{\eps}{\varepsilon}
\newcommand{\kk}{\kappa}

\newcommand{\ld}{\lambda}

\newcommand{\dt}{\partial_t}

\newcommand{\les}{\lesssim}

\newcommand{\jb}[1]
{\langle #1 \rangle}

\newcommand{\gm}{\mathbf{\mathrm{g}}}

\newcommand{\Prob}{\mathbb{P}}

\newcommand{\M}{\mathcal{M}}

\newcommand{\N}{\mathbb{N}}

\newcommand{\ZZ}{\mathcal{Z}}

\renewcommand{\H}{\mathcal{H}}
\newcommand{\D}{\mathcal{D}}

\newcommand{\drm}{\textbf{d}}

\makeatletter
\def\DeclareSymbol#1#2#3{\expandafter\gdef\csname MH@symb@#1\endcsname{\tikz[baseline=#2, scale=.18]{#3}}}
\def\<#1>{\ensuremath{\mathchoice{\tikzsetnextfilename{macros#1}{\color{black}\csname MH@symb@#1\endcsname}}{\tikzsetnextfilename{macros#1}{\color{black}\csname MH@symb@#1\endcsname}}{\tikzsetnextfilename{macros#1}\scalebox{.7}{\color{black}\csname MH@symb@#1\endcsname}}
{\tikzsetnextfilename{macros#1}\scalebox{.5}{\color{black}\csname MH@symb@#1\endcsname}}}} 
\makeatother 
\DeclareSymbol{1}{0}{\path (-.4,0) -- (.4,0); \draw[line width=0.2mm] (0,0) -- (0,1.2) 
node[circle, fill, draw, solid, inner sep=0pt, minimum size=1.8pt] {};}

\newcommand{\PA}{\mathsf{P}}
\newcommand{\PI}{\mathsf{\Pi}}
\newcommand{\DC}{\mathsf{C}}

\newcommand{\dg}{\mathbf{d}}

\numberwithin{equation}{section}
\numberwithin{theorem}{section}

\begin{document}
\baselineskip = 15pt

\title{Anderson stochastic quantization equation}

\author[H.~Eulry, A.~Mouzard, and T.~Robert]
{Hugo Eulry, Antoine Mouzard, and Tristan Robert}
%
%
%
%
%
%
%
%
%

%

\begin{abstract} We study the parabolic defocusing stochastic quantization equation with both mutliplicative spatial white noise and an independant space-time white noise forcing, on compact surfaces, with polynomial nonlinearity. After renormalizing the nonlinearity, we construct the random Gibbs measure as an absolutely continuous measure with respect to the law of the Anderson Gaussian Free Field for fixed realization of the spatial white noise. Then, when the initial data is distributed according to the Gibbs measure, we prove almost sure global well-posedness for the dynamics and invariance of the Gibbs measure.
\end{abstract}

\date{\today}


\maketitle
%


\tableofcontents

\baselineskip = 14pt

\section{Introduction}
\subsection{A singular SPDE in a singular random environment}
Let $\M$ be a closed Riemannian surface, $2$-dimensional, compact and boundaryless. We consider the parabolic defocusing stochastic quantization equation
\begin{align}{\label{AndersonSQE}}
\begin{cases}
    \partial_t u + \H u +F'(u) = \sqrt{2}\zeta,\\
    u(0)=u_0,
    \end{cases}
\end{align}
where
\begin{align}\label{polynom}
F(X)=\sum_{k=0}^{2m}a_k X^k
\end{align} 
is a polynomial of even degree $2m$, $\zeta$ is a space-time white noise on $\R\times\M$, and $\H$ denotes the Anderson Hamiltonian operator. The latter is a rough and random perturbation of the Laplace-Beltrami operator: 
\begin{align}\label{H}
\H\approx -\Delta +\xi,
\end{align}
where $\xi$ is a spatial white noise on $\M$, independent of $\zeta$.

This study finds its place among the many other ones dealing with invariant measures for singular stochastic PDEs. If we were to get rid of the spatial white noise and consider $-\Dl$ instead of $\H$, then \eqref{AndersonSQE} corresponds to the usual stochastic quantization equation of the $P(\Phi_2)$ model, for which the analogue of the Gibbs measure \eqref{AndersonGibbs} was constructed as part of the development of the Euclidean Quantum Field Theory in the 60s and 70s (see e.g. \cite{Simon} and references therein), while the corresponding dynamics was first studied in \cite{DPD}; see also \cite{GH,MW,TW}, among others. Other dynamics leaving this measure invariant were also studied: wave equations \cite{GKO,GKOT,ORT,OT1,OTWZ}, Schrödinger equations \cite{BL,Bourgain96,DNY1,OT,Tz},  and complex Ginzburg-Landau equations \cite{DBDF,RZ,Trenberth}, to name but a few.

Without the stochastic forcing term $\zeta$ in \eqref{AndersonSQE} and with deterministic initial data, the equation becomes the Parabolic Anderson Model with polynomial nonlinearity, amenable to both the paracontrolled calculus approach \cite{GIP} or the approach using regularity structures \cite{Hairer}. Recent progress on the understanding of the Anderson Hamiltonian \cite{AllezChouk,GIP,GUZ,Labbe,HL,Mouzard22} opened the door to the study of dynamics in singular environment with regular enough deterministic data, such as the Anderson heat equation \cite{GIP,HL2}, or Anderson wave and Schr\"odinger equations \cite{DLTV,DM,DW,GUZ,MZ,TV1,TV2,Ugurcan}. In this paper, we go one step further and study a singular stochastic PDE in a singular random environment, namely the Anderson heat equation with polynomial nonlinearity and both rough deterministic or random initial data, and singular stochastic source term \eqref{AndersonSQE}. A similar approach was recently considered by Barashkov, De Vecchi, and Zachhuber \cite{BDVZ}, who studied the problem of invariance of the Gibbs measure for the cubic non-linear Anderson wave equation.

\subsection{Main results and sketch of the proof}

 Because $\xi$ in \eqref{H} is almost surely of regularity\footnote{$\Cc^\sigma$ denotes the H\"older space of regularity $\sigma$; see Definition~\ref{def} below.} $\Cc^{-1-\kappa}$ for any $\kappa>0$, the very definition of $\H$ is tricky as one needs to go through a renormalization procedure to make sense of it as an unbounded $L^2(\M)$ operator. Its precise definition is given in Section \ref{SEC:Anderson}, see for instance \cite{AllezChouk,GUZ,Labbe,Mouzard22} and references therein for the whole story, the most important point being that $\H$ can be defined as an unbounded positive self-adjoint operator with dense domain and compact resolvent. In this work, we then construct a Gibbs type measure for \eqref{AndersonSQE} and prove local and almost sure global well-posedness and invariance of this measure.

\smallskip

The noiseless version of \eqref{AndersonSQE} has indeed a gradient flow structure 
$$
\partial_t u + \nabla_u \CE(u)=0,
$$
with energy 
$$
\CE(u):=\int_\M \Big[\frac{1}{2}u\H u + F(u)\Big]\, d x.
$$

\noindent Following the finite dimensional mantra, this yields a formal Gibbs measure for \eqref{AndersonSQE} under the form 
$$
d\rho(u)=\ZZ^{-1}e^{-\CE(u)}d u,
$$
or more precisely 
\begin{align}{\label{AndersonGibbs}}
    d\rho(u)=\ZZ^{-1}e^{-\int_\M \big[\frac{1}{2}u\H u + F(u)\big] dx}d u.
\end{align}
As it is, \eqref{AndersonGibbs} is only a formal expression and several problems arise when we try to make sense of it. First of all, $d u$ supposedly refers to an infinite dimensional Lebesgue measure, which has no proper definition. The usual trick is to consider instead the quadratic part of the density together with $d u$ so that it defines a Gaussian measure 
$$
d \mu^\H(u):=\tilde{\ZZ}^{-1}e^{-\frac{1}{2}\int_\M u\H u\, dx}d u,
$$
see Subsection \ref{AndersonGreenGFF} for the details on $\mu^\H$. It is defined as the centered Gaussian measure on $L^2(\M)$ with covariance operator $\H^{-1}$. However, since $\H^{-1}$ is not trace class on $L^2(\M)$, $\mu^\H$ will rather be supported on strictly larger spaces, namely $H^{-\varepsilon}(\M)$ for any $\varepsilon>0$. Another equivalent definition of $\mu^\H$ would be to match the law of the Gaussian Free Field (GFF) associated to $\H$, that is the random series
\begin{align}{\label{DefGFF}}
    u^\omega(x):=\sum_{n\geq0}\frac{\gamma_n(\omega)}{\sqrt{\lambda_n}}\varphi_n(x)
\end{align}

\noindent where $\gamma_n$ are i.i.d. standard Gaussian random variables, $\lambda_n$ the eigenvalues for $\H$ and $\varphi_n$ the corresponding $L^2(\M)$-orthonormal basis of eigenfunctions. This random series can be shown to almost surely converge in $H^{-\varepsilon}(\M)$ for positive $\varepsilon$ but not in $L^2(\M)$. This poses another difficulty in that typical elements $u$ in the support of $\mu^\H$ would then be distributions, making the non-linear term $F(u)$ ill-defined as $u\notin L^2(\M)$ $\mu^\H$-almost surely. This can be handled by performing a Wick renormalization of the non-linear term $F(u)$ that will be described in Subsection \ref{SubRenormWick}. In short, if 
\begin{align}\label{PN}
\P_N:=\psi(N^{-2}\Delta)
\end{align}
is a Schwartz multiplier, then 
$$
\E_{\mu^\H}\left[|\P_N u|^2(x)\right]=\sum_{n\geq0}\frac{1}{\lambda_n}|\P_N\varphi_n|^2(x)\rightarrow +\infty
$$
as $N$ goes to $+\infty$. This is due to the fact that the Green's function of $\H$ on the diagonal has a logarithmic divergence, see \eqref{GN1} for a precise statement. Thus, Wick reordering replaces monomials $X^k$ by Hermite polynomials $H_k(X, \sigma^2)$ with variance $\sigma^2$. Writing 
\begin{align}\label{sigmaN}
\sigma_N^2(x):=\E_{\mu^\H}\left[|\P_N u|^2(x)\right],
\end{align} 
we then define the corresponding renormalized monomial
$$
(\P_N u)^{\diamond k}:= H_k(\P_N u, \sigma_N^2)
$$
and extend the definition to polynomials as in \eqref{polynom} by linearity:
\begin{align}\label{polynomWick}
F_N^\diamond(\P_N u)=\sum_{k=0}^{2m}a_k(\P_Nu)^{\diamond k}=\sum_{k=0}^{2m}a_kH_k(\P_Nu,\sigma_N^2).
\end{align} 
The key point is that we get convergence of the renormalized non-linearity as $N\to+\infty$, that is
$$
\int_\M F_N^\diamond(\P_N u)\,d x\to \int_\M F^\diamond(u)\,d x
$$
in $L^p(d \mu^\H)$ for any $p\ge1$; see Proposition~\ref{RenormNonLinear} below. Definition and convergence properties of Hermite polynomials will be explained in Subsection \ref{SubRenormWick}. From now on, $\sigma_N$ is fixed to be as above and all the renormalized powers will be constructed with respect to $\sigma_N$. After this regularization and renormalization, we are left with an approximate Gibbs measure 
\begin{align}{\label{TruncatedGibbsMeasure}}
    d\rho_N(u):=\ZZ_N^{-1}e^{-\int_\M F_N^\diamond(\P_Nu)\,d x}\,d \mu^\H(u).
\end{align}
Recall that we work with a defocusing non-linearity, thus $\rho_N$ is well-defined for any fixed $N$. Our first result is the following, where we construct the Gibbs measure \eqref{AndersonGibbs} as the limit of the sequence $(\rho_N)_N$.

\begin{theorem}{\label{ConvergenceGibbsMeasure}}
    The sequence $(\rho_N)_N$ converges in total variation to a measure $\rho$ that is absolutely continuous with respect to $\mu^\H$.
\end{theorem}

\noindent Detailed proof will be given in Subsection \ref{SubConvergenceGibbsMeasure} but the core idea, once convergence of the renormalized polynomials has been established, is to use Boué-Dupuis' formula as introduced in \cite{BG1} and \cite{BG2} to get uniform boundedness of the density.


\smallskip 

Now that a candidate Gibbs measure has been constructed, we have to build the corresponding dynamics and check that the measure stays invariant under the flow. A natural way to do so is to truncate the equation 
\begin{align}{\label{TruncEquation}}
    \partial_t u_N + \H u_N + \P_N f^\diamond_N(\P_Nu_N)=\sqrt{2}\zeta
\end{align}

\noindent where $f_n^\diamond= (F_N^\diamond)'$. Local well-posedness can be proved using Da Prato Debussche trick, from there, it is only a matter of proving that solutions exist globally for some well chosen subset of initial data and that we can pass $N$ to the limit. We first show local-well posedness and convergence of the solution to \eqref{TruncEquation} for deterministic initial data.
\begin{theorem}\label{THM:LWP}
Let $0<\epsilon<\frac1{2m-1}$, $\eps<\sigma<1$, and $q>1$ be such that $\frac{\sigma+\eps}{2}q'<\frac1{2m-1}$. Then $\Prob$-almost surely, for any $u_0\in\Cc^{-\eps}(\M)$ and $T_0>0$, there exists $T\in (0;T_0\wedge 1]$ such that for any $N\in\N^*$,  \eqref{TruncEquation} admits a solution $u_N\in C([0;T];\Cc^{-\eps}(\M))$, unique in the affine space $\llp+X_T^{-\eps,\sigma}$. Moreover $u_N$ converges almost surely to some $u\in \llp+X_T^{-\eps,\sigma}$ which is the unique solution to\footnote{The renormalized nonlinearity $f^\diamond$ is well-defined on the class $\llp+X_T^{-\eps,\sigma}$ as the limit of $\P_Nf_N^\diamond(\P_N u)$ thanks to the algebraic rule in Lemma~\ref{BinomialHermite} and the convergence of $(\P_N\llp)^{\diamond k}$.}
\begin{align}\label{LimitEquation}
\begin{cases}
\dt u +\H u +f^\diamond(u)=\sqrt{2}\zeta,\\
u(0)=u_0,
\end{cases}
\end{align}
in the class $\llp+X_T^{-\eps,\sigma}$.
\end{theorem}
The space $X_T^{-\eps,\sigma}\subset C([0;T];\Cc^{-\eps}(\M))$ is defined in \eqref{space} below, and $\llp$ is the stationnary solution to the linearized stochastic equation
\begin{align*}
(\dt  +\H )\llp = \sqrt{2}\zeta.
\end{align*}
Thus Theorem~\ref{THM:LWP} provides a solution for any data in $\Cc^{-\eps}(\M)$, which is the natural space where $\llp(t)$ lives.

\smallskip

Our next result deals with globalization of the solution in the case where the initial data is random and distributed by the Gibbs measure $\rho$ constructed in Theorem~\ref{ConvergenceGibbsMeasure}, and also proves that the law of the solution is invariant and given by $\rho$.

\begin{theorem}\label{THM:GWP}
Let $\eps$ and $\sigma$ be as in Theorem~\ref{THM:LWP}. If $u_0$ is distributed according to $\rho$, then almost surely, for any $T>0$ and $N\in\N^*$, the truncated equation \eqref{TruncEquation} admits a unique solution in $\llp+X_T^{-\eps,\sigma}$. Moreover $u_N$ converges in $C([0;T];\Cc^{-\eps}(\M))$ to the unique solution $u\in \llp+X_T^{-\eps,\sigma}$ to \eqref{LimitEquation}, and the law of $u(t)$ does not depend on $t$ and is given by $\rho$.
\end{theorem}

\noindent As mentionned above, $\zeta$ denotes a space-time white noise on $\R\times\M$ that is independent of $\H$. In other words, $\zeta$ is a centered Gaussian process indexed by $L^2(\R\times\M)$ functions and with covariance given by 
$$
\E\left[\zeta(\varphi)\zeta(\psi)\right]=\langle\varphi,\psi\rangle_{L^2(\R\times\M)},
$$
which is why $\zeta$ is said to have \textit{delta} correlations since formaly the pointwise correlations are $\E\big[\zeta(t,x)\zeta(s,y)\big]=\delta_0(x-y)\delta_0(t-s)$. Note that given any $L^2(\M)$ orthonormal basis $(\varphi_n)_n$, we have the decomposition
$$
\zeta(t)dt=\sum_{n\geq0}d B_n(t)\varphi_n
$$
where $B_n(t):=\langle\zeta,\IDC_{[0,t]}\varphi_n\rangle$ are i.i.d. standard Brownian motion on $\R$. Since $\H$ itself has an $L^2(\M)$-orthonormal basis of eigenfunctions, the aforementioned expansion of $\zeta$ can be done using these eigenfunctions.

\begin{remark}
\rm
$\H$ being a random operator for which we fix a realisation $\H^{\omega'}$, then $\zeta$ would write
$$
\zeta^{\omega}(t)dt =\sum_{n\geq0}d B_n^{\omega,\omega'}(t)\varphi_n^{\omega'}
$$
for a realisation $\omega$. Note that $\zeta$ itself does not depend on $\omega'$ while it coefficients $B_n$ do depend on both $\omega$ and $\omega'$. In the following, we fix a realisation $\omega'$ so that $\H^{\omega'}$ is defined once and for all. In other words, all the probabilistic considerations have to be taken conditionally to $\H$. In particular, while we strongly believe that the double layer of randomness would be worth studying, it is out of the scope of this paper and we provide an approach that only relies on properties derived from the construction of $\H$. Section \ref{SEC:Anderson} will recall all the material needed here.
\end{remark}
The main difficulty in the solution theory of the renormalized equation \eqref{TruncEquation} is then the operator $\H$ itself. On the one hand, the formal expression \eqref{H} can be made rigorous, namely $\H$ can be written (through conjugation with the so-called $\Gamma$ map, see \eqref{Gamma} below) as a perturbation of $-\Dl$. On the other hand, the perturbative terms have very rough coefficients, so that one has to carefully check that the classical analytic toolbox for the Laplace-Beltrami operator used in the context of more standard singular stochastic PDEs can be adapted to our setting despite those rough coefficients. In particular, we only dispose of a Schauder estimate with limited range of exponents (see \eqref{SchauderH} below), and we have to use the precise mapping properties of the perturbative terms in $\H$ to show that the Green's function for $\H$ has the same singularities as the one for $-\Dl$, which is crucial in building the Wick powers of the stochastic convolution to any degree; see Corollary~\ref{COR:GHN} and Proposition~\ref{RenormWick}.
 
 We conclude this introduction with some further remarks.


\begin{remark}
\rm
The solution constructed in Theorem~\ref{THM:LWP} for arbitrary initial data in $\Cc^{-\eps}(\M)$ is only local in time, while using the Gibbs measure we could globalize the solution only for $\rho$-almost every initial data in Theorem~\ref{THM:GWP}. The question of deterministic global well-posedness and ergodicity of the dynamics is an interesting open question, and will be investigated in a follow-up work.
\end{remark}
 
\begin{remark}\label{Rk:commute}
\rm
We made sense of the formal equation \eqref{AndersonSQE} through a limiting procedure involving the equation with regularized nonlinearity \eqref{TruncEquation}. In the field of singular stochastic PDEs, it is more common to instead regularize the noise and replace $\zeta$ by some smooth approximation $\P_N\zeta$ and study the convergence of the corresponding solutions $u_N$ to the renormalized equation
\begin{align*}
\dt u_N + \H u_N + f_N^\diamond(u_N)=\sqrt{2}\P_N\zeta
\end{align*}
where now the renormalized nonlinearity $f_N^\diamond$ is not truncated by the operator $\P_N$ contrary to \eqref{TruncEquation}.
However, \eqref{TruncEquation} is better-suited for our purpose since (i) we want \textit{in fine} to prove the invariance of the Gibbs measure, which we only know how to do by using a finite-dimensional approximation of the dynamics which has itself a finite-dimensional approximation of the Gibbs measure as invariant measure; and (ii) since $\H$ and $-\Dl$ do not commute, and since the regularizing operator $\P_N$ is a Schwartz multiplier based on $-\Dl$, it is not as straightforward to study the stochastic convolution $\widetilde{\llp_N}(t):=\int_{-\infty}^t e^{-(t-s)\H}\P_N\zeta(ds)$ and its Wick powers to get the analogue of Proposition~\ref{RenormWick}.
\end{remark}
  
\begin{remark}
\rm
In \eqref{AndersonSQE}, $\H$ is a Schr\"odinger operator with rough random electric potential $\xi$. The same arguments as those in the present paper can also deal with \eqref{AndersonSQE} where \eqref{H} is replaced by $(i\nabla + A)^2$, where now $A$ is a rough and random magnetic potential. The construction of the corresponding magnetic Anderson operator when the magnetic field $B= \nabla \wedge A$ is a spatial white noise, namely $\nabla\wedge A=\xi$, has been carried out in \cite{MM}. In this case, \eqref{AndersonSQE} can be interpreted as the stochastic quantization equation for the marginal of the $P(\Phi_2)$-Yang-Mills-Higgs measure. Indeed, the abelian Yang-Mill-Higgs measure is the Gibbs type measure formally written as
\begin{align*}
``d\rho = \ZZ^{-\CE(A,\Phi)}dAd\Phi",
\end{align*}
where the abelian $P(\Phi_2)$-Yang-Mills-Higgs action is given by
\begin{align*}
\CE(A,\Phi)=\frac12\int_\M\Big[B^2+|(i\nabla +A)\Phi|^2+F(\Phi)\Big]dx;
\end{align*}
see e.g. \cite{Shen}.
Upon fixing a global gauge to have a one-to-one correspondence between $A$ and $B$, we can then rewrite the measure similarly as in \eqref{AndersonGibbs} by
\begin{align*}
``d\rho = \ZZ^{-1}e^{-\int_\M\big[\frac12\Phi \H_B\Phi+F(\Phi)\big]dx}d\Phi d\mu(B),
\end{align*}
where $\H_B = (i\nabla + A)^2$ and
\begin{align*}
d\mu(B)=\ZZ^{-1}e^{-\frac12\int_\M B^2dx}dB
\end{align*}
is the white noise measure. Given a realization of $B$, the measure $d\rho_B= \ZZ^{-1}e^{-\int_\M\big[\frac12\Phi \H_B\Phi+F(\Phi)\big]dx}d\Phi$ is then the analogue of the measure \eqref{AndersonGibbs}, with the white noise electric potential $\xi$ replaced by the white noise magnetic field $B$, and one can study its invariance under a parabolic singular stochastic dynamics similar to \eqref{AndersonSQE}.
\end{remark}
\begin{remark}
\rm
Since for fixed time, the stochastic convolution for the Anderson stochastic heat equation and that of the Anderson stochastic damped wave equation
\begin{align}\label{SAdW}
\dt^2u + \H u +\dt u= \sqrt{2}\zeta
\end{align}
share the same covariance function (see e.g. \cite{ORW}), given by the Green's function of $\H$, then Proposition~\ref{RenormWick} also holds for the Wick powers of the solution to \eqref{SAdW}. Together with the same nonlinear analysis in $L^2$-based Sobolev spaces\footnote{We can simply use that $\D^{-\eps}=H^{-\eps}$ and $\D^{1-\eps}=H^{1-\eps}$, and that the propagator $\frac{\sin(t\sqrt{\H})}{\sqrt{\H}}$ maps $\D^{-\eps}$ to $\D^{1-\eps}$ from the functional calculus.} as e.g. in \cite{ORT}, this shows that the analogue of Theorem~\ref{THM:GWP} also holds true for the stochastic damped Anderson nonlinear wave equation with any polynomial nonlinearity
\begin{align*}
\dt^2u+\H u +\dt u + f^\diamond(u)=\sqrt{2}\zeta.
\end{align*}
\end{remark}

\subsection*{Plan of the paper}The paper is organized as follows. In Section \ref{SEC:stochastic} we construct the renormalized powers of the stochastic objects in play. We also use Boué-Dupuis formula to construct the Gibbs measure as the limit in total variation of its approximations (Theorem~\ref{ConvergenceGibbsMeasure}). In Section \ref{SEC:dynamics}, we first prove local well-posedness for deterministic initial data in $\CC^{-\epsilon}(\M)$, that is Theorem \ref{THM:LWP}. Then we prove invariance of the measure as well as global well-posedness for random initial data distributed according to the Gibbs measure, that is Theorem \ref{THM:GWP}. Section \ref{SEC:Anderson} is devoted to recalling the construction of the Anderson operator $\H$ and some of its properties, together with a study of its Green's function and of its Schartz multipliers. Finally, in Appendix \ref{SEC:Appendix}, we gather the analytic toolbox that we need in analyzing the Wick powers and in the nonlinear estimates for our well-posedness theory.
\section{Construction of the stochastic objects}\label{SEC:stochastic}

In this section we provide the probabilistic framework to make sense of the renormalization procedure involved in the definition of $\rho_N$ in \eqref{TruncatedGibbsMeasure} and also prove Theorem \ref{ConvergenceGibbsMeasure}.

\subsection{Algebraic setting and Wick powers}

We first of all recall the definition of Hermite polynomials and the corresponding algebraic rules, refering to \cite{Nualart} or \cite{DPT} for more insights on the whole story. We denote by $H_k$ the $k$-th Hermite polynomial defined by the generating series 
$$
e^{tx-\frac{1}{2}t^2}=\sum_{k\geq0}\frac{t^k}{k!}H_k(x)
$$
for $x,t\in\R$. From this expansion, we easily get the recursion formula 
$$
H_k'=kH_{k-1}
$$
for any positive integer $k$. Such polynomials are heavily related to Gaussian variables in the sense of the following lemma, see \cite[Lemma 1.1.1]{Nualart} for a proof.

\begin{lemma}{\label{RuleHermite}}
    Let $X,Y$ be real standard Gaussian random variables. Then 
    \begin{align*}
        \E\left[H_k(X)H_l(Y)\right]=k!\delta_{k,l}\E[XY]^k.
    \end{align*}
\end{lemma}

\noindent Note that we have an extra $k!$ factor here, this is due to our normalization of Hermite polynomials so that they have leading coefficient 1 as opposed to $\frac{1}{k!}$ in \cite{Nualart}. We also get the corresponding binomial formula. 
\begin{lemma}{\label{BinomialHermite}}
    For $x,y\in\R$ and $k\in\N^*$
    $$
    H_k(x+y)=\sum_{p=0}^k\binom{k}{p}x^{k-p}H_p(y).
    $$
\end{lemma}

\begin{proof}
    It is a direct consequence of the definition of $H_k$ by the generating series and the following expansion for $x,y,t\in\R$
    \begin{align*}
        \sum_{k\geq0}\frac{t^k}{k!}H_k(x+y)&=e^{(x+y)t-\frac{1}{2}t^2}\\
        &=e^{xt}e^{yt-\frac{1}{2}t^2}\\
        &=\sum_{n\geq0}\frac{t^n}{n!}x^n\sum_{p\geq0}\frac{t^p}{p!}H_p(y)\\
        &=\sum_{k=n+p}\frac{t^k}{k!}\frac{k!}{p!n!}x^{n}H_{p}(y)\\
        &=\sum_{k\geq0}\frac{t^k}{k!}\sum_{p=0}^k\binom{k}{p}x^{k-p}H_p(y).
    \end{align*}
\end{proof}

\noindent We then define the Hermite polynomials with variance $\sigma^2>0$ as 
$$
H_k(x,\sigma^2):=\sigma^kH_k(\sigma^{-1}x),
$$
which allows to transfer the algebraic rules \eqref{RuleHermite} and \eqref{BinomialHermite} to Gaussian random variables with non-unit variance. In the following, for a real centered Gaussian random variable $X$, we will use the shorthand notation called \textit{Wick power} of $X$
$$
X^{\diamond k}:=H_k(X,\sigma^2),
$$
where $\sigma^2=\E[X^2]$. We end this subsection with the  Wiener chaos estimate, see for example \cite[Theorem I.22]{Simon}.

\begin{lemma}{\label{WienerChaos}}
    Let $g=(g_n)$ be a sequence of i.i.d. real standard Gaussian random variables on some probability space $\Omega$ and $(P_j)$ be a family of polynomials in $g$ with degree at most $k\in\N$. Then for $p\geq 2$
    $$
    \Big\|\sum_{j\geq0}P_j(g)\Big\|_{L^p(\Omega)}\leq (p-1)^{\frac{k}{2}}\Big\|\sum_{j\geq0}P_j(g)\Big\|_{L^2(\Omega)}.
    $$
\end{lemma}

\subsection{Renormalization}{\label{SubRenormWick}}

As progressing through the construction of the Gibbs measure \eqref{AndersonGibbs}, we will encounter truncated stochastic objects for which we would like to take the limit in the truncation parameter. This will either be for the non-linearity in the energy
$$
\int_\M F^{\diamond}_N(\P_N u)\,d x,
$$
where $u$ is distributed according to the GFF $\mu^\H$ and $F_N^{\diamond}$ defined in \eqref{polynomWick}, or for the stochastic convolution involved in the construction of the renormalized dynamic
\begin{align}{\label{DefLoli}}
    \llp(t):=\sqrt{2}\int_{-\infty}^t e^{-(t-s)\H}\zeta(ds)
\end{align}
and its mollification
$$
\llp_N(t):=\sqrt{2}\P_N\int_{-\infty}^t e^{-(t-s)\H}\zeta(ds).
$$

\begin{proposition}{\label{RenormWick}}
    Fix $k\in\N$, for any $T>0$, $\varepsilon\in(0,1)$ and $p\geq q\geq1$ the sequence $(\llp^{\diamond k}_N)_N$ is a Cauchy sequence in $L^p(\Omega;L^q([0,T];\CC^{-\varepsilon}(\M)))$ that also converges almost surely in $L^q([0,T];\CC^{-\varepsilon}(\M))$. Moreover $t\mapsto\llp_N(t)$ and the limiting process $t\mapsto\llp(t)$ are almost surely in $C([0,T];\CC^{-\varepsilon}(\M))$ and we have the following tail estimates and deviation bounds:
    \begin{align}\label{EstimateTailLoliN}
        \mathbb{P}\left(\|\llp_{N}^{\diamond k}\|_{L^q_T\CC_x^{-\varepsilon}}> R\right)\leq C e^{-cR^\frac{2}{k}T^{-\frac{2}{qk}}},
    \end{align}
    and
    \begin{align}\label{EstimateDeviationLoliN}
        \mathbb{P}\left(\|\llp_{N}^{\diamond k}-\llp^{\diamond k}\|_{L^q_T\CC_x^{-\varepsilon}}> N^{-\kappa}R\right)\leq C e^{-cR^{\frac{2}{k}}T^{-\frac{2}{qk}}},
    \end{align}
    for some constant $c,C>0$, and exponent $\kappa>0$ that do not depend on $T$, $R$, $p$, or $N$.
\end{proposition}

\begin{proof}
    As a starter, we prove that $(\llp^{\diamond k}_N(t))_N$ is bounded in $L^p(\Omega;\CC^{-\varepsilon}(\M))$ uniformly in $N\in\N$ and $t\ge 0$. 
    Using Itô's isometry and expanding $\zeta$ in the eigenbasis $(\varphi_n)$ of $\H$, we get for $t\geq0$, $x,y\in\M$ and $N_1,N_2\in\N$
    \begin{align*}
       & \E\left[\llp_{N_1}(t,x)\llp_{N_2}(t,y)\right]\\
       &=2\E\left[(\P_{N_1}^x\otimes\P_{N_2}^y)\int_{-\infty}^t\int_{-\infty}^te^{-(t-s)\H}\zeta^x(ds)e^{-(t-\tau)\H}\zeta^y(d\tau)\right](x,y)\\
        &=2(\P_{N_1}^x\otimes\P_{N_2}^y)\sum_{n,m\geq0}\E\left[\int_{-\infty}^t\int_{-\infty}^te^{-(t-s)\lambda_n}\varphi_n^xd B_n(s)e^{-(t-\tau)\lambda_m}\varphi_m^yd B_m(\tau)\right](x,y)\\
        &=2(\P_{N_1}^x\otimes\P_{N_2}^y)\left(\sum_{n\geq0}\int_{-\infty}^te^{-2(t-s)\lambda_n}\varphi_n^x\varphi_n^y\,ds\right)(x,y)\\
        &=\sum_{n\geq0}\frac{(\P_{N_1}\varphi_n)(x)(\P_{N_2}\varphi_n)(y)}{\lambda_n}\\
        &=G^\H_{N_1,N_2}(x,y),
    \end{align*}
    where $G^\H_{N_1,N_2}(x,y)=(\P_{N_1}^x\otimes\P_{N_2}^y)G^\H(x,y)$ and $G^\H$ is the Green function of $\H$. Since the covariance above does not depend on time, we drop the $t$ dependency as it does not matter in the estimates. Combining the Sobolev inequality (Lemma~\ref{LEM:Besov}~(i) and~(ii)) with the Wiener chaos estimate \eqref{WienerChaos}, we get for $\sigma\in(0,1)$ small enough depending on $\varepsilon$ and $p$:
    \begin{align*}
        \E\left[\|\llp_{N}^{\diamond k}\|_{\CC^{-\varepsilon}}^p\right]&\lesssim\E\left[\|\llp_{N}^{\diamond k}\|_{W^{-\sigma,p}}^p\right]\\
        &\lesssim\int_\M\E\left[\left((1-\Delta)^{-\frac{\sigma}{2}}(\llp_{N}^{\diamond k})(x)\right)^p\right]\,d x\\
        &\lesssim(p-1)^{p\frac{k}{2}}\int_\M\E\left[\left((1-\Delta)^{-\frac{\sigma}{2}}(\llp_{N}^{\diamond k})(x)\right)^2\right]^{\frac{p}{2}}\,d x\\
        &=(p-1)^{p\frac{k}{2}}\int_\M\left(\int_{\M^2} G_{\sigma}(x,y)G_{\sigma}(x,z)\E\left[\llp_{N}^{\diamond k}(y)\llp_{N}^{\diamond k}(z)\right]\,dy\,dz\right)^{\frac{p}{2}}\,d x,
    \end{align*}
    where $G_\sigma$ is the Green function of $(1-\Delta)^{-\frac{\sigma}{2}}$.
    From there, we use the previous computation and the algebraic rule \eqref{RuleHermite} to get rid of the expectation so that 
    \begin{align*}
        \E\left[\|\llp_{N}^{\diamond k}\|_{\CC^{-\varepsilon}}^p\right]\lesssim(p-1)^{p\frac{k}{2}}\int_\M\left(\int_{\M^2} G_{\sigma}(x,y)G_{\sigma}(x,z)G_N^\H(y,z)^k\,dy\,dz\right)^{\frac{p}{2}}\,d x.
    \end{align*}
   Recall from \eqref{Gsigma} and \eqref{GHN1} that
    $$
    G_\sigma(x,y)\les \drm(x,y)^{\sigma-2}
    $$
    and 
    $$
    \left|G_N^\H(x,y)+\frac{1}{2\pi}\log\big(\drm(x,y)+N^{-1}\big)\right|\leq C.
    $$
    Thus once we replace inside the integral it yields
    \begin{align*}
        &\E\left[\|\llp_{N}^{\diamond k}\|_{\CC^{-\varepsilon}}^p\right]\\
        &\lesssim(p-1)^{p\frac{k}{2}}\int_\M\left(\int_{\M^2} \drm(x,y)^{\sigma-2}\drm(x,z)^{\sigma-2}(1+|\log(\drm(y,z)+N^{-1})|^k)\,dy\,dz\right)^{\frac{p}{2}}\,dx.
    \end{align*}
    Note that, since we integrate symmetric functions of $(y,z)\in\M^2$, we can restrict the integral to $\drm(x,y)\geq \drm(x,z)$ so that, using the triangle inequality on $\drm(y,z)$,
    \begin{align*}
        \E\left[\|\llp_{N}^{\diamond k}\|_{\CC^{-\varepsilon}}^p\right]&\lesssim(p-1)^{p\frac{k}{2}}\\
        &+(p-1)^{p\frac{k}{2}}\int_\M\left(\int_{\drm(x,y)\geq \drm(x,z)} \drm(x,y)^{\sigma-2}\drm(x,z)^{\sigma-2}|\log(\drm(x,y))|^k\,dy\,dz\right)^{\frac{p}{2}}\,d x.
    \end{align*}
    Since $|\log(\drm(x,y))|^k\lesssim_k \drm(x,y)^{-\frac{\sigma}{2}}$, we end up with 
    \begin{align*}
        \E\left[\|\llp_{N}^{\diamond k}\|_{\CC^{-\varepsilon}}^p\right]\lesssim(p-1)^{p\frac{k}{2}}+(p-1)^{p\frac{k}{2}}\int_\M\left(\int_{\drm(x,y)\geq \drm(x,z)} \drm(x,y)^{\frac{\sigma}{2}-2}\drm(x,z)^{\sigma-2}\,dy\,dz\right)^{\frac{p}{2}}\,d x,
    \end{align*}
    where the inner integral on the right-hand side is bounded uniformly in $x\in\M$ as $\M$ is 2-dimensional and compact. Using next Minkowski's integral inequality, we get
    $$
    \E\left[\|\llp^k\|_{L^q_T\CC^{-\epsilon}_x}^p\right]\leq \left\|\E\left[\|\llp^k\|_{\CC^{-\epsilon}_x}^p\right]^{\frac{1}{p}}\right\|_{L^q_T}^p\lesssim(p-1)^{p\frac{k}{2}}T^{\frac{p}{q}}.    $$
    This gives the uniform boundedness of $\llp^{\diamond k}_N$ in $L^p(\Omega;L^q([0,T];\CC^{-\varepsilon}(\M)))$ for any $q\ge1$. Chebychev's inequality and optimizing in $p$ then ensures that the tail estimate \eqref{EstimateTailLoliN} holds true:
    \begin{align*}
        \mathbb{P}\left(\|\llp_{N}^{\diamond k}\|_{L^q_T\CC^{-\epsilon}_x}>R\right)&\leq C(p-1)^{p\frac{k}{2}}R^{-p}T^{\frac{p}{q}}\\
        &\leq Ce^{-cR^\frac{2}{k}T^{-\frac{2}{qk}}}
    \end{align*}
    for some constants $c,C$ that do not depend on $T$, $R$, $p$, or $N$.

    Now to show that $\{\llp_N^{\diamond k}\}_N$ is Cauchy, let $N_1,N_2$ be two integers with $N_1\leq N_2$, then as before we get 
    \begin{align*}
        &\E\left[\|\llp_{N_1}^{\diamond k}-\llp_{N_2}^{\diamond k}\|_{\CC^{-\varepsilon}}^p\right]\\
        &\lesssim\E\left[\|\llp_{N_1}^{\diamond k}-\llp_{N_2}^{\diamond k}\|_{W^{-\sigma,p}}^p\right]\\
        &\lesssim\int_\M\E\left[\left((1-\Delta)^{\frac{-\sigma}{2}}(\llp_{N_1}^{\diamond k}-\llp_{N_2}^{\diamond k})(x)\right)^p\right]\,d x\\
        &\lesssim(p-1)^{p\frac{k}{2}}\int_\M\E\left[\left((1-\Delta)^{\frac{-\sigma}{2}}(\llp_{N_1}^{\diamond k}-\llp_{N_2}^{\diamond k})(x)\right)^2\right]^{\frac{p}{2}}\,d x\\
        &=(p-1)^{p\frac{k}{2}}\int_\M\left(\int_{\M^2} G_{\sigma}(x,y)G_{\sigma}(x,z)\E\left[(\llp_{N_1}^{\diamond k}-\llp_{N_2}^{\diamond k})(y)(\llp_{N_1}^{\diamond k}-\llp_{N_2}^{\diamond k})(z)\right]\,dy\,dz\right)^{\frac{p}{2}}\,d x
    \end{align*}
    where we used Lemma \ref{RuleHermite}. Itô's isometry on the expectation term yields
    \begin{align*}
        \E\left[(\llp_{N_1}^{\diamond k}-\llp_{N_2}^{\diamond k})(y)(\llp_{N_1}^{\diamond k}-\llp_{N_2}^{\diamond k})(z)\right]&=\E\left[\llp_{N_1}^{\diamond k}(y)\llp_{N_1}^{\diamond k}(z)\right]-\E\left[\llp_{N_1}^{\diamond k}(y)\llp_{N_2}^{\diamond k}(z)\right]\\
        &\qquad -\E\left[\llp_{N_2}^{\diamond k}(y)\llp_{N_1}^{\diamond k}(z)\right]+\E\left[\llp_{N_2}^{\diamond k}(y)\llp_{N_2}^{\diamond k}(z)\right]\\
        &=G_{N_1}^\H(y,z)^k-G_{N_1N_2}^\H(y,z)^k-G_{N_2N_1}^\H(y,z)^k+G_{N_2}^\H(y,z)^k.
    \end{align*}
    Using successively the mean value inequality, estimate \eqref{GHN2}, and again the comparison between logarithm and monomials, we get for any $0<\delta\ll 1$:
    \begin{align*}
        |G_{N_1}^\H(y,z)^k-G_{N_1N_2}^\H(y,z)^k|&\lesssim |G_{N_1}^\H(y,z)-G_{N_1N_2}^\H(y,z)|(1+\drm(y,z)^{-\frac{\sigma}{6}})\\
        &\lesssim ((|\log(\drm(y,z)+N_1^{-1})|\vee 1)\wedge (N_1^{\delta-1}\drm(y,z)^{-1}))(1+\drm(y,z)^{-\frac{\sigma}{6}}).
    \end{align*}
    In the case $\dg(x,y)\le N_1^{-1}$, we have by interpolation and the comparison between logarithm and monomials that
    \begin{align*}
        ((|\log(\drm(y,z)+N_1^{-1})|\vee 1)\wedge (N_1^{\delta-1}\drm(y,z)^{-1}))&\les \log N_1\wedge (N_1^{\delta-1}\drm(y,z)^{-1}))\\
        & \lesssim_{\delta,\theta} N_1^{(\frac\delta2-1)\theta}\dg(y,z)^{-\theta} 
    \end{align*}
    for any $\theta\in(0;1)$, uniformly in $N_1$, while in the case $\dg(x,y)\ge N_1^{-1}$ we get
    \begin{align*}
    ((|\log(\drm(y,z)+N_1^{-1})|\vee 1)\wedge (N_1^{\delta-1}\drm(y,z)^{-1}))&\les \log \dg(y,z)^{-1}\wedge (N_1^{\delta-1}\drm(y,z)^{-1}))\\
    &\lesssim_{\delta,\theta} N^{(\delta-1)\theta}\dg(y,z)^{-((1-\theta)\frac\nu2+\theta)}
    \end{align*}
for any $\theta\in(0;1)$ and $0<\nu\ll1$, uniformly in $N_1$. This finally gives
    \begin{align}{\label{GN1N2k}}
        |G_{N_1}^\H(y,z)^k-G_{N_1N_2}^\H(y,z)^k|&\lesssim N_1^{-2\kappa}\drm(y,z)^{-\frac{\sigma}{2}}
    \end{align}
    for some $\kappa>0$ and implicit constant that is uniform in $y,z$ and $N_1,N_2$. Similar computations as before yield that the whole inner integral on $\M^2$ is bounded uniformly in $x\in\M$ with
    $$
    \int_{\M^2} G_{\sigma}(x,y)G_{\sigma}(x,z)\E\left[(\llp_{N_1}^{\diamond k}-\llp_{N_2}^{\diamond k})(y)(\llp_{N_1}^{\diamond k}-\llp_{N_2}^{\diamond k})(z)\right]\,dy\,dz\lesssim N_1^{-2\kappa}
    $$
    so that, using Minkowski inequality, $(\llp^{\diamond k}_N)_N$ indeed is a Cauchy sequence in $L^p(\Omega;L^q([0,T];\CC^{-\varepsilon}(\M)))$, we denote by $\llp^{\diamond k}$ the limiting process.\\
    Since $L^p(\Omega)$ convergence ensures convergence in probability, we can use Chebychev's inequality together with passing to the limit $N_2\to+\infty$ and optimizing in $p$ so that 
    \begin{align*}
        \mathbb{P}\left(\|\llp_{N}^{\diamond k}-\llp^{\diamond k}\|_{L^q_T\CC_x^{-\varepsilon}}>N^{-\kappa}R\right)&\leq (p-1)^{p\frac{k}{2}}N^{\kappa p}R^{-p}\E\left[\|\llp_{N}^{\diamond k}-\llp^{\diamond k}\|_{L^q_T\CC_x^{-\varepsilon}}^p\right]\\
        &\leq C(p-1)^{p\frac{k}{2}}R^{-p}T^{\frac{p}{q}}\\
        &\leq Ce^{-cR^\frac{2}{k}T^{-\frac{2}{qk}}}
    \end{align*}
    for some constants $c,C>0$ and the same exponent $\kappa>0$ defined above. Note that $c$, $C$ and $\kappa$ do not depend on $N,R$ or $T$. This proves deviation estimate \eqref{EstimateDeviationLoliN} as well as almost sure convergence in $L^q([0,T],\CC^{-\varepsilon}(\M))$ using Borel-Cantelli lemma.

    Now for the continuity of $t\mapsto \llp(t)\in\CC^{-\varepsilon}(\M)$, we use Kolmogorov's continuity criterion. We have similarly as above for some large $p$ depending on $\eps$, using Sobolev inequality and the Wiener chaos estimate
    \begin{align*}
       & \E\big\|\llp(t+\delta)-\llp(t)\big\|_{\Cc^{-\eps}}^p \les \E\big\|\llp(t+\delta)-\llp(t)\big\|_{W^{-\frac\eps2,p}}^p\\
        &\les (p-1)^{\frac{p}2}\int_\M\Big(\int_{\M^2}G_{\frac\eps2}(x,y)G_{\frac\eps2}(x,z)\E\big[(\llp(t+\delta,y)-\llp(t,y))(\llp(t+\delta,z)-\llp(t,z))\big]dydz\Big)^{\frac{p}2}dx\\
         &= (p-1)^{\frac{p}2}\int_\M\Big(\int_{\M^2}G_{\frac\eps2}(x,y)G_{\frac\eps2}(x,z)\big[(1-e^{-\delta\H})\otimes 1\big] G^\H(y,z)dydz\Big)^{\frac{p}2}dx.
     \end{align*}
 Now, using Sobolev inequality,
 \begin{align*}
 \Big\|\int_\M G_{\frac\eps2}(x,y)\big[(1-e^{-\delta\H})\otimes 1\big] G^\H(y,z)dy\Big\|_{L^\infty_{x,z}}&= \Big\|\big[(1-e^{-\delta\H})\otimes 1\big] G^\H(x,z)\Big\|_{L^\infty_{z}W^{-\frac\eps2,\infty}_x}\\
 &\les \Big\|\big[(1-e^{-\delta\H})\otimes 1\big] G^\H(x,z)\Big\|_{L^\infty_{z}\D^{1-\frac\eps4}_x},
\end{align*}       
with $\D^{1-\frac\eps4}$ the Sobolev space associated to $\H$, see \eqref{SobolevH}. Using the mean value inequality and the bound $|1-e^{-\delta\lambda_n}|^2\lesssim (\delta\lambda_n)^{\frac{\eps}{8}}$, for any fixed $z\in\M$,
\begin{align*}
\Big\|\big[(1-e^{-\delta\H})\otimes 1\big] G^\H(\cdot,z)\Big\|_{\D^{1-\frac\eps4}_x}^2&=\sum_{n\ge 0}\frac{|1-e^{-\delta\lambda_n}|^2}{\lambda_n^{1+\frac\eps4}}\varphi_n(z)^2\\
&\les \delta^{\frac\eps8}\sum_{n\ge 0}\lambda_n^{-1-\frac\eps8}\varphi_n(z)^2\\
&\sim\delta^\frac\eps8 \|\delta_z\|_{\D^{-1-\frac\eps8}}^2\\
&\les \delta^{\frac\eps8},
\end{align*}
uniformly in $z\in\M$, where we used that $\|\delta_z\|_{\D^{-1-\frac\eps8}}$ is uniformly bounded in $z\in\M$, see the proof of \cite[Lemma 23]{BDM}. Since $G_{\frac\eps2}(x,z)\in L^\infty_x L^1_z$, and since $\M$ is compact, putting everything together we get
\begin{align*}
    \E\big\|\llp(t+\delta)-\llp(t)\big\|_{\Cc^{-\eps}}^p\les \delta^{p\frac\eps{32}},
\end{align*}
from which $\llp\in C(\R;\Cc^{-\eps}(\M))$ a.s. due to Kolmogorov continuity criterion\footnote{Kolmogorov regularity theorem rather ensures that there exists a modification of the process $\llp_N$ (resp. $\llp$) that is almost surely continuous. We identify $\llp_N$ (resp. $\llp$) with this continuous modification in our argument.}.
\end{proof}

In our globalisation argument in Section~\ref{SEC:global} below, we will also need to investigate the regularization of $\llp^{\diamond k}$ based on $\H$ instead of $\Dl$.
\begin{proposition}\label{PROP:llpM}
Fix $k\in\N$ and $\chi\in C^\infty_0(\R)$ supported in $(-1;1)$. For $M\in\N^*$, let $\chi_M=\chi(M^{-2}\H)$. Then for any $\eps\in(0,1)$, any $p\geq q\geq1$, and any $N\in\N^*$, the sequence $\{(\P_N\chi_M\llp)^{\diamond k}\}_M$ is a Cauchy sequence in $L^p(\Omega;L^q([0;T];\CC^{-\varepsilon}(\M)))$ that also converges almost surely in $L^q([0;T];\CC^{-\varepsilon}(\M))$ to $(\P_N\llp)^{\diamond k}$. Moreover, $(\P_N\chi_M \llp)^{\diamond k}$ is uniformly bounded in $N,M$ in $L^p(\Omega;L^q([0;T];\Cc^{-\eps}(\M)))$, and we have the tail estimates
\begin{align}\label{EstimateDeviationLoliNM}
        \mathbb{P}\left(\|(\P_N\chi_M\llp)^{\diamond k}\|_{L^q_T\CC_x^{-\varepsilon}}> R\right)\leq C e^{-cR^{\frac{2}{k}}T^{-\frac{2}{qk}}}
    \end{align}
    and
    \begin{align}\label{EstimateDeviationLoliNMdiff}
        \mathbb{P}\left(\|(\P_N\chi_M\llp)^{\diamond k}-(\P_N\llp)^{\diamond k}\|_{L^q_T\CC_x^{-\varepsilon}}> N^{c\kappa}M^{-\kappa}R\right)\leq C e^{-cR^{\frac{2}{k}}T^{-\frac{2}{qk}}},
    \end{align}
   uniformly in $R$, $p$, and $N,M$.
   
   Similar results hold for $(\P_N\chi_M\llp(0))^{\diamond k}$ in $\Cc^{-\eps}(\M)$.
\end{proposition}
\begin{proof}
Repeating the arguments above in the proof of Proposition~\ref{RenormWick}, we have for fixed $t\ge 0$:
\begin{align*}
 &\E\left[\|(\P_N\chi_{M_1}\llp(t))^{\diamond k}-(\P_N\chi_{M_2}\llp(t))^{\diamond k}\|_{\Cc^{-\eps}}^p\right]\\
 &\les_p\int_\M\left(\int_{\M^2} G_{\frac\eps4}(x,y)G_{\frac\eps4}(x,z)\right.\\
 &\times\left.\left[G_{N,M_1,M_1}^\H(y,z)^k-G_{N,M_1,M_2}^\H(y,z)^k-G_{N,M_2,M_1}^\H(y,z)^k+G_{N,M_2,M_2}^\H(y,z)^k\right]\,dy\,dz\right)^{\frac{p}{2}}\,d x
\end{align*}
for some $p\gg1$ depending on $\eps$, and where $$G_{N,M_j,M_\ell}^\H=(\P_N\chi_{M_j}\otimes\P_N\chi_{M_\ell})G^\H=(\P_N\otimes\P_N)(\chi_{M_j}\otimes\chi_{M_\ell})G^\H.$$ Then using the mean value theorem and \eqref{GHM1}-\eqref{GHM2}-\eqref{Gsigma} we can estimate the inner integrals by
\begin{align*}
\int_{\M^2}\dg(x,y)^{\frac{\eps}4-2}\dg(x,y)^{\frac\eps4-2}N^{c\delta}M_1^{-\delta}|\log(\dg(y,z))|^{k-1}dydz \les N^{c\delta}M_1^{-\delta}
\end{align*}
for some small $\delta>0$. This is enough for \eqref{EstimateDeviationLoliNMdiff} after finishing the proof as in that of Proposition~\ref{RenormWick}.

The estimate \eqref{EstimateDeviationLoliNM} follows similarly, using that we have a bound uniform in $N\in\N^*$ from using only \eqref{GHM1} and not \eqref{GHM2}.
\end{proof}

\subsection{Truncated Gibbs measure and convergence}{\label{SubConvergenceGibbsMeasure}}

We finally construct the Gibbs measure $\rho$ as the limit in total variation of $(\rho_N)_N$ given by \eqref{TruncatedGibbsMeasure}:
$$
d\rho_N(u)=\ZZ_N^{-1}e^{-\int_\M F_N^\diamond(\P_Nu)\,d x}\,d \mu^\H(u),
$$
where $\ZZ_N$ is the normalization constant 
$$
\ZZ_N:=\E_{\mu^\H}\left[e^{-\int_\M F_N^\diamond(\P_Nu)\,d x}\right].
$$
First, we prove the convergence of the density functions.

\begin{proposition}{\label{RenormNonLinear}}
    For any $p\geq1$, the sequence $\left(\int_\M F_N^\diamond(\P_Nu)\,d x\right)_N$ is a Cauchy sequence in $L^p(\mu^\H)$, converging to a limit denoted by $\int_\M F^\diamond(u)\,d x\in L^p(\mu^\H)$.
\end{proposition}

\begin{proof}
    Since $F$ is a polynomial of degree $k=2m$, it is sufficient to treat the case where $F(X)=X^k$ is a monomial. The proof follows the same lines as the one of Proposition \eqref{RenormWick}. First, we have the bound
    \begin{align*}
        \left\|\int_\M (\P_Nu)^{\diamond k}\,d x\right\|_{L^p(\mu^\H)}^2&\lesssim_p \left\|\int_\M (\P_Nu)^{\diamond k}\,d x\right\|_{L^2(\mu^\H)}^2\\
        &= \int_{\M^2} \E\left[(\P_Nu)^{\diamond k}(x)(\P_Nu)^{\diamond k}(y)\right]\,d x\,dy\\
        &= \int_{\M^2} \E\left[(\P_Nu)(x)(\P_Nu)(y)\right]^k\,d x\,dy\\
        &= \int_{\M^2} G_N^\H(x,y)^k\,d x\,dy\\
        &\lesssim \int_{\M^2}\big(1+|\log(\dg(x,y)+N^{-1})|)^{k}\,d x\,dy\\
        &\lesssim 1,
    \end{align*}
uniformly in $N$, where we used successively the Wiener chaos estimate \eqref{WienerChaos}, the algebraic rule \eqref{RuleHermite} and the estimate \eqref{GN1}. This proves the uniform boundedness in $L^p(\mu^\H)$.

Now to prove the Cauchy property, take $N_2\geq N_1\geq 1$ and following the lines of Proposition \ref{RenormWick}, we have
\begin{align*}
   & \left\|\int_\M \big((\P_{N_1}u)^{\diamond k}-(\P_{N_2}u)^{\diamond k}\big)\,d x\right\|_{L^p(\mu^\H)}^2\\
   &\lesssim_p \left\|\int_\M \big((\P_{N_1}u)^{\diamond k}-(\P_{N_2}u)^{\diamond k}\big)\,d x\right\|_{L^2(\mu^\H)}^2\\
    &= \int_{\M^2} \E\Big[\big((\P_{N_1}u)^{\diamond k}-(\P_{N_2}u)^{\diamond k}\big)(x)\big((\P_{N_1}u)^{\diamond k}-(\P_{N_2}u)^{\diamond k}\big)(y)\Big]\,d x\,dy\\
    &= \int_{\M^2} \Big(G_{N_1}^\H(x,y)^k-G_{N_1N_2}^\H(x,y)^k-G_{N_2N_1}^\H(x,y)^k+G_{N_2}^\H(x,y)^k\Big)\,d x\,dy.
\end{align*}
From there, the estimate \eqref{GN1N2k} from the proof of Proposition~\ref{RenormWick} ensures that the right-hand side vanishes as $N_1\to+\infty$, proving that $\left(\int_\M F_N^\diamond(\P_Nu)\,d x\right)_N$ is a Cauchy sequence in $L^p(\mu^\H)$.
\end{proof}

\noindent Write $R_N(u):=e^{-\int_\M F_N^\diamond(\P_Nu)\,d x}$ for the exponential density function. Note that as a direct consequence of Proposition~\ref{RenormNonLinear}, we get the convergence of $\left(R_N(u)\right)_N$ almost surely with respect to $\mu^\H$ as a continuous function of the almost sure convergent sequence $\left(-\int_\M F_N^\diamond(\P_Nu)\,d x\right)_N$. To conclude that the convergence also happens in $L^p(\mu^\H)$ for $p\geq1$, fix $N_1,N_2\geq1$ and $p\geq1$. Denote by $A_{N_1,N_2,\eta}$ the event $$
A_{N_1,N_2,\eta}:=\left\{\left|R_{N_1}(u)-R_{N_2}(u)\right|>\eta\right\}
$$
then convergence in measure of $R_N(u)$ ensures that for any $\eta>0$, 
$$
\mu^\H(A_{N_1,N_2,\eta})\to 0\quad\text{as}\quad N_1,N_2\to+\infty.
$$
Thus 
\begin{align}{\label{PartitionCauchy}}
    \left\|R_{N_1}(u)-R_{N_2}(u)\right\|_{L^p(\mu^\H)}&\leq \left\|\left(R_{N_1}(u)-R_{N_2}(u)\right)\mathbf{1}_{A_{N_1,N_2,\eta}}\right\|_{L^p(\mu^\H)}\notag\\
    &+\left\|\left(R_{N_1}(u)-R_{N_2}(u)\right)\mathbf{1}_{A_{N_1,N_2,\eta}^c}\right\|_{L^p(\mu^\H)}\\
    &\leq \mu^\H(A_{N_1,N_2,\eta})^{\frac{1}{2p}}\left\|R_{N_1}(u)-R_{N_2}(u)\right\|_{L^{2p}(\mu^\H)}+ \mu^\H(A^c_{N_1,N_2,\eta})^{\frac{1}{p}}\eta\notag.
\end{align}
From there, we see that a uniform bound on $R_N(u)$ in $L^{2p}(\mu^\H)$ is enough to get convergence of $R_N(u)$ in $L^p(\mu^\H)$. This will be provided by the Boué-Dupuis' formula in the form introduced in \cite{Ustunel} and \cite{BG2}. Roughly speaking, given a cylindrical Brownian motion $X=(X_t)_{t\in[0,1]}$ on $L^2(\M)$ and under some assumptions on the non-linear functional $J:C([0,1];C^{(-1)^-}(\M))\to\R$, then the following variational formula holds:
\begin{align}{\label{BoueDupuisFormula}}
    \log\,\E\left[e^{-J(X)}\right]=\sup_{v}\,\E\left[-J\left(X+\int_0^\cdot v(t)\,dt\right)-\frac{1}{2}\int_0^1\|v(t)\|_{L^2(\M)}^2\,dt\right],
\end{align}
where the supremum is taken over all the progressively measurable processes $v\in L^2(\Omega;L^2([0,1];L^2(\M)))$. This allows to bound uniformly the partition function $\ZZ_N$.

\begin{proposition}{\label{UnifPartitionFunction}}
    The normalization constant $\ZZ_N=\E_{\mu^\H}\left[R_N(u)\right]$ is uniformly bounded :
    $$
    \sup_N\,\E_{\mu^\H}\left[R_N(u)\right]<+\infty.
    $$
\end{proposition}

\begin{proof}
    The assumption on $J$ is to be \textit{tame} (see \cite[Definition 1]{BG2}) with respect to the law of $X$ for which an easy sufficient condition is to be bounded below and have some finite moment. Consider therefore the functional 
    $$
    J_N(X):=\int_\M F_N^{\diamond}(\P_N \H^{-\frac{1}{2}}X_{t=1})\,d x.
    $$
    First of all note that $J_N$ is well defined and measurable on $C([0,1];\Cc^{-1-\kappa}(\M))$ by composition, for any $\kappa>0$. Since by assumption $F$ is a polynomial of even degree and positive leading coefficient, so is $F_N^\diamond$ and therefore $J_N$ is bounded from below for each fixed $N$.\\
    Then, note that the law of $u=\H^{-\frac{1}{2}}X_{t=1}$ is exactly the one of the GFF $\mu^\H$ in \eqref{DefGFF}, for which we already proved that 
    $$
    \int_\M F_N^\diamond(\P_Nu)\,d x
    $$
    is bounded uniformly in $L^p(\mu^\H)$, see Proposition \ref{RenormNonLinear}. Therefore, the functional $J_N$ is tame in the sense of \cite{BG2}, that is
    $$
    \E_\mathbb{P}\left[e^{-qJ}\right]+\E_\mathbb{P}\left[|J|^p\right]<+\infty
    $$
    for some conjugate exponents $p,q>1$, and Boué-Dupuis' formula holds true in our context. As the supremum runs over the progressively measurable processes $v\in L^2(\mathbb{P};L^2([0,1];L^2(\M)))$, we define the process $\Theta$ by 
    $$
    \Theta(t):=\int_0^t\H^{-\frac{1}{2}}v(s)\,ds,
    $$
    so that $\dot{\Theta}=\H^{-\frac{1}{2}}v$ and identity \eqref{BoueDupuisFormula} writes
    \begin{align}{\label{BoueDupuisFormula2}}
        \log\,\E_{\mu^\H}\left[R_N(u)\right]=\sup_{\Theta}\,\E\left[-\int_\M F^\diamond_N\left(\P_Nu+\P_N\Theta_{t=1}\right)\,d x-\frac{1}{2}\int_0^1\|\dot{\Theta}\|_{D\left(\sqrt{\H}\right)}^2\,dt\right].
    \end{align}
    Even though $F$ has positive leading coefficient, $F_N^\diamond$ is not bounded from below uniformly in $N$. We thus expand the non-linearity using the binomial rule in Lemma~\ref{BinomialHermite} to get
    \begin{align}
        F^\diamond_N\left(\P_Nu+\P_N\Theta_{|t=1}\right)&=\sum_{p=0}^{2m}a_pH_p(\P_Nu+\P_N\Theta_{|t=1},\sigma_N^2)\notag\\
        &=\sum_{p=0}^{2m}\sum_{j=0}^pa_p\binom{p}{j}(\P_Nu)^{\diamond j}(\P_N\Theta_{|t=1})^{p-j}\label{binom}\\
        &=F(\P_N\Theta_{|t=1})+\sum_{p=1}^{2m}\sum_{j=1}^pa_p\binom{p}{j}(\P_Nu)^{\diamond j}(\P_N\Theta_{|t=1})^{p-j}.\notag
    \end{align}
    The assumption that $F$ is bounded below is crucial here as the contribution from $j=0$ together with the quadratic part in \eqref{BoueDupuisFormula2},
    \begin{align*}
        -\int_\M F(\P_N\Theta_{t=1})\,d x-\frac{1}{2}\int_0^1\|\dot{\Theta}\|_{D\left(\sqrt{\H}\right)}^2\,dt,
    \end{align*}
    will control all the cross product terms in between. The duality pairing between $\CC^{-\varepsilon}$ and $B^\eps_{1,1} \supset W^{\eps,1}$ yields
    \begin{align*}
        \left|\int_\M (\P_Nu)^{\diamond j}(\P_N\Theta_{t=1})^{p-j}\,d x\right|&\lesssim \|(\P_Nu)^{\diamond j}\|_{\CC^{-\varepsilon}(\M)}\|(\P_N\Theta_{t=1})^{p-j}\|_{W^{\varepsilon,1}(\M)}.
    \end{align*}
    At this point, the only remaining difficulty is to deal with the terms involving $\Theta$ in \eqref{binom}. For the sake of clarity in the following estimates, let $v$ be a generic element of $W^{\varepsilon,1}(\M)$. First use the interpolation inequality
    $$
    \|v\|_{W^{\varepsilon,1}}\lesssim \|v\|_{W^{1-\eta,1}}^{\frac{\varepsilon}{1-\eta}}\|v\|_{L^1}^{\frac{1-\eta-\varepsilon}{1-\eta}}
    $$
    for $\eta\in(0,1)$ small enough. Together with Young's inequality, this yields 
    $$
    \|v\|_{W^{\varepsilon,1}}\lesssim \|v\|_{W^{1-\eta,1}}^{\frac{\varepsilon}{1-\eta}r}+\|v\|_{L^1}^{\frac{1-\eta-\varepsilon}{1-\eta}r'}
    $$
    for some $\frac{1}{r}+\frac{1}{r'}=1$ to be chosen later. Recall the fractional Leibniz rule from Lemma~\ref{LEM:prodHugo}:
    $$
    \|uv\|_{H^t}\lesssim\|u\|_{H^s}\|v\|_{H^r}
    $$
    as long as $t=r+s-1$ satisfy $0<r+s$ and $r,s<1$. Then, taking $r=s=1-\frac{\eta}{2}$, we get
    \begin{align*}
        \|v^{p-j}\|_{H^{1-\eta}}&\lesssim \|v^{p-j-1}\|_{H^{1-\frac{\eta}{2}}}\|v\|_{H^{1-\frac{\eta}{2}}}\\
        &\lesssim_j \|v\|_{H^{1-\frac{\eta}{2^{p-j}}}}^{p-j}
    \end{align*}
    and combined with the embedding $H^{1-\eta}(\M)\hookrightarrow W^{1-\eta,1}(\M)$ and $L^{2m}(\M)\hookrightarrow L^{p-j}(\M)$ as $\M$ is compact together with the fact that the form domain $D(\sqrt{\H})$ embeds in any $H^{1-\kappa}(\M)$, we get
    \begin{align*}
        \|v^{p-j}\|_{W^{\varepsilon,1}}&\lesssim_j \|v\|_{H^{1-\frac{\eta}{2^{p-j}}}}^{\frac{\varepsilon}{1-\eta}(p-j)r}+\|v^{p-j}\|_{L^1}^{\frac{1-\eta-\varepsilon}{1-\eta}r'}\\
        &\lesssim_j \|v\|_{D(\sqrt{\H})}^{\frac{\varepsilon}{1-\eta}(p-j)r}+\|v\|_{L^{2m}}^{\frac{1-\eta-\varepsilon}{1-\eta}(p-j)r'}.
    \end{align*}
    Note that $H^{1-}(\M)$ itself is not an algebra, hence the loss of regularity when performing the product estimate. However, Leibniz rule allows to still have a nice estimate in spaces below $H^1(\M)$ and in the end, an estimate involving only the norm on $D(\sqrt{\H})$.\\
    \noindent Recall that the goal is to deal with the supremum over $\Theta$, hence the terms involving $u$ are harmless as they are bounded in $L^p(\Omega,\CC^{-\varepsilon}(\M))$ for any $p>1, \varepsilon>0$ and thus provide bounds uniformly in $N$. For $\delta>0$ and $s>1$, we use once again Young's inequality to have
    \begin{align*}
        \left|\int_\M (\P_Nu)^{\diamond j}(\P_N\Theta_{t=1})^{p-j}\,d x\right|&\lesssim \|(\P_Nu)^{\diamond j}\|_{\CC^{-\varepsilon}(\M)}\|(\P_N\Theta_{t=1})^{p-j}\|_{W^{\varepsilon,1}(\M)}\\
        &\lesssim \|(\P_Nu)^{\diamond j}\|_{\CC^{-\varepsilon}(\M)}\|\Theta_{t=1}\|_{D(\sqrt{\H})}^{\frac{\varepsilon}{1-\eta}(p-j)r}\\
        &\qquad+\|(\P_Nu)^{\diamond j}\|_{\CC^{-\varepsilon}(\M)}\|\P_N\Theta_{t=1}\|_{L^{2m}}^{\frac{1-\eta-\varepsilon}{1-\eta}(p-j)r'}\\
        &\le C(\delta)\|(\P_Nu)^{\diamond j}\|_{\CC^{-\varepsilon}(\M)}^{\frac{s}{s-1}}\\
        &\qquad+ \delta^s\|\Theta_{t=1}\|_{D(\sqrt{\H})}^{\frac{\varepsilon}{1-\eta}(p-j)rs}+\delta^s\|\P_N\Theta_{t=1}\|_{L^{2m}}^{\frac{1-\eta-\varepsilon}{1-\eta}(p-j)r's}.
    \end{align*}
    Note that we dropped the regularization $\P_N$ in the $\|\P_N\Theta_{t=1}\|_{H^{1-\frac{\eta}{2^{p-j}}}}$ term as it maps usual Sobolev spaces to themselves, uniformly in $N$. In view of \eqref{BoueDupuisFormula2}, to prove uniform boundedness we only need to tune the parameters so that the $D(\sqrt{\H})$ norm (resp. $L^{2m}$) is at most to the power $2$ (resp. $2m$), hence the following conditions
    \begin{align*}
        &\frac{\varepsilon}{1-\eta}(p-j)r < 2,\\
        &\frac{1-\eta-\varepsilon}{1-\eta}(p-j)r' < 2m.
    \end{align*}
    As $p-j<2m$ for $j\neq0$, the second condition is satisfied provided $r$ is large enough. For such $r$ being fixed, we just need to take $\varepsilon$ small enough so that the first condition is satisfied as well. This ensures that for some $s>1$ close enough to 1, the Young's estimate above, together with $\delta$ chosen small enough, make the cross product terms be absorbed by 
    \begin{align*}
        -\int_\M F(\P_N\Theta_{t=1})\,d x-\frac{1}{2}\int_0^1\|\dot{\Theta}\|_{D\left(\sqrt{\H}\right)}^2\,dt.
    \end{align*}
    Indeed, gathering everything together we get
    \begin{align*}
        &-\int_\M F^\diamond_N\left(\P_Nu+\P_N\Theta_{t=1}\right)\,d x-\frac{1}{2}\int_0^1\|\dot{\Theta}\|_{D\left(\sqrt{\H}\right)}^2\,dt\\
        &\hspace{3cm}= -\int_\M F(\P_N\Theta_{t=1})\,d x-\frac{1}{2}\int_0^1\|\dot{\Theta}\|_{D\left(\sqrt{\H}\right)}^2\,dt\\
        &\hspace{5cm} -\sum_{p=1}^{2m}\sum_{j=1}^pa_p\binom{p}{j}\int_\M (\P_Nu)^{\diamond j}(\P_N\Theta_{t=1})^{p-j}\,dx\\
        &\hspace{3cm}\le -\int_\M F(\P_N\Theta_{t=1})\,dx -\frac{1}{2}\int_0^1\|\dot{\Theta}\|_{D\left(\sqrt{\H}\right)}^2\,dt\\
        &\hspace{5cm} + \delta^s\|\Theta_{t=1}\|_{D(\sqrt{\H})}^{2}+\delta^s\|\P_N\Theta_{t=1}\|_{L^{2m}}^{2m}+C(\delta)\|(\P_Nu)^{\diamond j}\|_{\CC^{-\varepsilon}(\M)}^{\frac{s}{s-1}}\\
        &\hspace{3cm}\le -\int_\M \big(F(\P_N\Theta_{t=1})-\delta^s (\P_N\Theta_{t=1})^{2m}\big)\,dx\\ &\hspace{5cm}-\big(\frac{1}{2}-\delta^s\big)\int_0^1\|\dot{\Theta}\|_{D\left(\sqrt{\H}\right)}^2\,dt+C(\delta)\|(\P_Nu)^{\diamond j}\|_{\CC^{-\varepsilon}(\M)}^{\frac{s}{s-1}}.
        \end{align*}
        Taking $\delta$ small enough ensures that the first two terms are bounded above by a deterministic constant, uniformly in $N$ and $\Theta$. Taking the expectation, the last term is harmless as $u$ is distributed according to $\mu^\H$ for which we already proved uniform bounds for the renormalized polynomials, this concludes the proof.
\end{proof}

\noindent Note that from the definition of $R_N$, then $(R_N)^p$ only amounts for $p\geq1$ to scaling the non linear functional $F$ by a factor $p$ so that the above argument still holds for any order moment and we have
$$
\sup_N\left\|R_N(u)\right\|_{L^p(\mu^\H)}< +\infty.
$$
Combined with estimate \eqref{PartitionCauchy} we get the following proposition.

\begin{proposition}
    For any $p\geq1$, $(R_N)_N$ is a Cauchy sequence in $L^p(\mu^\H)$.
\end{proposition}

We can now give the proof of our first result.
\begin{proof}[Proof of Theorem~\ref{ConvergenceGibbsMeasure}]
\noindent Since for all $N\in\N^*$, the measures $\rho_N$ are all defined by their density $\ZZ_N^{-1}R_N$ with respect to the same measure $\mu^\H$, and those density converge in $L^1(\mu^\H)$, it ensures the convergence in total variation of $(\rho_N)_N$ to some measure $\rho$ that is absolutely continuous with respect to $\mu^\H$.
\end{proof}

For our almost sure globalization argument, we will also need a bound uniform with respect to both a truncation with respect to $\Dl$ and one with respect to $\H$.
\begin{proposition}\label{PROP:BGNM}
Let $\chi\in C^\infty_0(\R)$ be supported in $(-1;1)$, and for $M\in\N^*$ set $\chi_M=\chi(M^{-2}\H)$.    Then 
    $$
    \sup_{N,M}\,\E_{\mu^\H}\left[R_N(\chi_M u)\right]<+\infty.
    $$
\end{proposition}
\begin{proof}
The proof follows from repeating verbatim that of Proposition~\ref{UnifPartitionFunction} up to replacing $\Theta$ and $u$ by $\chi_M\Theta$ and $\chi_Mu$, and using that $\chi_M$ is uniformly bounded on $\D(\sqrt{\H})$, and the last part of Proposition~\ref{PROP:llpM} instead of Proposition~\ref{RenormNonLinear}.
\end{proof}

\section{Well-posedness theory}{\label{SEC:dynamics}}

Now that the measure $\rho$ has been constructed, we study a parabolic dynamics leaving it invariant. Similarly to $\rho$, the dynamics will rather be defined at a truncated level $N$ and then passed to the limit. We therefore investigate the truncated dynamics \eqref{TruncEquation} that we recall here
\begin{equation*}
    \partial_t u_N + \H u_N + \P_Nf_N^\diamond(\P_N u_N)=\sqrt{2}\zeta
\end{equation*}
with $u(0)=u_0$. Note that we decide to truncate the non-linearity rather than the noise and initial data. This is mainly because the spectral projector $\P_N$ does not commute with $\H$; see Remark~\ref{Rk:commute}.\\
In this section we prove convergence of the smoothened dynamics \eqref{TruncEquation} to a dynamics leaving $\rho$ invariant, as well as deterministic local well-posedness and probabilistic global well-posedness results.

\subsection{Construction of the dynamics}

We study the dynamics
\begin{align*}
   \begin{cases}\partial_t u_N + \H u_N + \P_Nf^{\diamond}_N(\P_Nu_N)=\sqrt{2}\zeta\\
    u_N(0)=u_0
    \end{cases}
\end{align*}
\noindent for some rough initial data $u_0\in\Cc^{-\eps}(\M)$. As we investigate the invariance of the measure, $u_0$ is intended to be distributed with law $\rho$. Since $\rho$ is absolutely continuous with respect to $\mu^\H$ which is supported in $\Cc^{-\eps}$, it is enough to treat the case where $u_0\in \Cc^{-\eps}$.\\
We follow the usual Da Prato-Debussche trick from \cite{DPD}: split the equation into a linear dynamics driven by $\zeta$ and a non-linear dynamics starting from zero initial data. To that end we write $u_N = \llp + w_N$ where $\llp$ is the stationary solution to the linear equation
\begin{align*}
    (\partial_t + \H) \llp = \sqrt{2}\zeta
\end{align*}

\noindent and $w_N$ solves the non-linear equation
\begin{align}{\label{NonLinearEqv_N}}
    \begin{cases}
    \partial_t w_N + \H w_N + \P_Nf^{\diamond}_N(\P_Nw_N+\P_N\llp)=0\\
    w_N(0)=u(0)-\llp(0).
    \end{cases}
\end{align}

\noindent Note that the equation on $w_N$ still contains stochastic forcing terms due to the appearance of $\llp$ in the non-linearity. Since $\llp$ contains the rough part of the solution $u_N$, $w_N=u_N-\llp$ is expected to be smoother in some appropriate sense. 
In particular Proposition \ref{RenormWick} ensures that for any $p\geq q\geq1$, $T>0$ and $\varepsilon>0$, $(\P_N\llp)_N$ forms a Cauchy sequence in $L^p(\Omega;L^q([0;T];\CC^{-\varepsilon}(\M)))$ that converges to $\llp$. Moreover, $\llp$ is almost surely in $C([0,T];\CC^{-\varepsilon}(\M))$ and satisfies the tail and deviation estimates \eqref{EstimateTailLoliN} and \eqref{EstimateDeviationLoliN}.\\
Recall that $f=F'$ where $F$ is some polynomial of even degree $2m$ and positive leading coefficient. In particular, we write 
$$
f=\sum_{p=0}^{2m-1}b_pX^p
$$
and can expand the non linear term in \eqref{NonLinearEqv_N} using the algebraic rule \eqref{RuleHermite}:
\begin{align*}
    f_N^{\diamond}(\P_Nw_N+\P_N\llp)&=\sum_{p=0}^{2m-1}\sum_{j=0}^pb_p\binom{p}{j}(\P_Nw_N)^j(\P_N\llp)^{\diamond (p-j)}.
\end{align*}
For the sake of clarity in the following sections, we define 
$$
\llpg_N=\left(\llpg_N^{(j)}\right)_{0\leq j\leq 2m-1}:=\left((\P_N\llp)^{\diamond j}\right)_{0\leq j\leq 2m-1}
$$
to be the vector containing the renormalized powers of $\llp_N$. We also define $\llpg$ the same way, dropping the smoothing operator $\P_N$. Note that it immediately follows that $(\llpg_N)_N$ converges to $\llpg$ in any $L^p(\Omega)$, as a $L^q([0;T];\CC^{-\varepsilon}(\M))^{2m}$ valued random variable, for any $p,q\geq1$.\\
We also define the corresponding non-linearity 
\begin{align}{\label{NonLinearityN}}
    \mathbf{f}_N(w,\llpg):=\sum_{p=0}^{2m-1}\sum_{j=0}^pb_p\binom{p}{j}(\P_Nw)^j\llpg^{(p-j)},
\end{align}
and $\mathbf{f}$ in the same way dropping the operator $\P_N$, so that we are left with the study of the equation
\begin{align}{\label{NonLinearEqv_NGeneral}}
    \begin{cases}
    \partial_t w_N + \H w_N + \P_N\mathbf{f}_N(w_N,\llpg_N)=0,\\
    w_N(0)=u_0-\llp(0).
    \end{cases}
\end{align}

\noindent Provided we can solve \eqref{NonLinearEqv_N} in an appropriate functional space, solutions to \eqref{TruncEquation} will therefore be constructed as $u_N=\llp+w_N$. Note that $\llp$ in the decomposition does not depend on $N$, this is because we study the dynamics where neither the noise nor the initial condition are smoothened, but only the non-linear term.

\subsection{Deterministic local well-posedness}{\label{SEC:local}}

We investigate the well-posedness of \eqref{TruncEquation} for deterministic initial data $u_0\in\CC^{-\epsilon}(\M)$. In that case, we can postulate 
$$
u_N=\llp+w_N,
$$
\noindent where $w_N$ solves the non-linear equation
\begin{align}{\label{NonLinearEquationw_N}}
    \begin{cases}
    \partial_tw_N+\H w_N + \P_Nf_N^\diamond(\P_Nw_N+\llp_N)=0,\\
    w_N(0)=u_0-\llp(0).
    \end{cases}
\end{align}

\noindent In view of the Schauder estimate \eqref{SchauderH} for $\H$, and the initial condition belonging to $\CC^{-\eps}(\M)$, one expects $w_N$ to belong to the following space:
\begin{align}\label{space}
X_T^{-\eps,\sigma}=C([0,T],\CC^{-\eps}(\M))\cap C((0,T],\CC^{\sigma}(\M)),
\end{align}
endowed with the norm
$$
\|w\|_{X_T^{-\eps,\sigma}}=\|w\|_{L^\infty_T\CC^{-\epsilon}}+\sup_{0<t\leq T}t^{\frac{\sigma+\epsilon}{2}}\|w(t)\|_{\CC^{\sigma}},
$$
for $\eps<\sigma<1$ that takes account of the blow-up at $t=0$. The unusual restriction $\sigma<1$ instead of $\sigma<2$ in the case of the standard Laplace operator is due to the roughness of the Anderson Hamiltonian. One can prove that the solution belongs to the $L^2$-based Sobolev space $\D^{2-}$ associated to $\H$ which, only embeds in $C^{1-}$. Note that even if \eqref{NonLinearEquationw_N} is a truncated equation, as $\P_N$ maps Sobolev spaces to themselves, it falls into the scope of the following more general problem 
\begin{align}\label{model}
    \begin{cases}
    \partial_tw+\H w + \mathbf{f}(w,\llpg)=0,\\
    w(0)=w_0,
    \end{cases}
\end{align}
for the non-linearity $\mathbf{f}$ described above, a vector $\llpg\in L^q([0;T];\CC^{-\eps}(\M))^{2m}$ and initial condition $w_0\in\CC^{-\eps}(\M)$. We then have the following local well-posedness result.

\begin{proposition}{\label{Localwpw}}
    Let $0<\eps<\sigma< 1$ and $ 1\le q<\infty$ be such that $\frac{\sigma+\eps}{2}q'<\frac1{2m-1}$. Then there exists $\theta,C>0$ such that for any $T_0,r,R>0$, there exists $T\sim \min\big(T_0;\big(\frac{r}{(1+r+R)^{2m-1}}\big)^{\theta}\big)\in (0;T_0\wedge 1]$ such that for any $w_0\in\CC^{-\eps}(\M)$ and $\llpg\in L^q([0;T_0],\CC^{-\eps}(\M))^{2m}$ satisfying $\|w_0\|_{\Cc^{-\eps}}\le r$ and $\|\llpg\|_{L^q_{T_0}\Cc^{-\eps}}\le R$, there exists a unique solution $w$ in $X_T^{-\eps,\sigma}$ to \eqref{model}.
    Moreover $w$ depends Lipschitz continuously on both $w_0$ and $\llpg$, and satisfies
    \begin{align}\label{localbound}
    \|w\|_{X_T^{-\eps,\sigma}}\le Cr.
    \end{align}
\end{proposition}

\begin{proof}
    For $T\in(0;1]$, we define the solution mapping $\Phi_{(w_0,\llpg)}$ by
    $$
    \Phi_{(w_0,\llpg)}(w)(t):=e^{-t\H}w_0-\int_0^te^{-(t-s)\H}\mathbf{f}(w,\llpg)(s)\,ds
    $$
    and seek for a solution $w$ as a fixed point of $\Phi_{(w_0,\llpg)}$ in the ball of radius $4Cr$ in $X_T^{-\eps,\sigma}$ for some $C>0$ and small enough $T>0$ depending on $r,R$ as in the statement of Proposition~\ref{Localwpw}.
    
   First, for $v$ in the ball of radius $4Cr$ in $X_T^{-\eps,\sigma}$, using the expression \eqref{NonLinearityN} of $\mathbf{f}$, the Schauder estimate \eqref{SchauderH} for $\H$, and a repeated use of the product estimates in Lemma~\ref{LEM:Besov}~(iii), we estimate
    \begin{align*}
    \big\|\Phi_{(w_0,\llpg)}(w)\big\|_{L^\infty_T\Cc^{-\eps}}&\les \|w_0\|_{\Cc^{-\eps}}+\int_0^T\big\|\mathbf{f}(w,\llpg)(s)\big\|_{\Cc^{-\eps}}ds\\
    &\les \|w_0\|_{\Cc^{-\eps}}+\sum_{p=0}^{2m-1}\sum_{j=0}^p|b_p|\int_0^T\big\|w^j\llpg^{(p-j)}\big\|_{\Cc^{-\eps}}ds\\
    &\les \|w_0\|_{\Cc^{-\eps}}+\sum_{p=0}^{2m-1}\Big\{\sum_{j=1}^p|b_p|\int_0^T\|w(s)\|_{\Cc^{\sigma}}^j\|\llpg^{(p-j)}(s)\|_{\Cc^{-\eps}}ds\\
    &\qquad\qquad\qquad\qquad+|b_0|\int_0^T\|\llpg^{(p)}(s)\|_{\Cc^{-\eps}}ds\Big\}\\
     &\les \|w_0\|_{\Cc^{-\eps}}+\sum_{p=0}^{2m-1}\Big\{\sum_{j=1}^p|b_p|\|w\|_{X_T^{-\eps,\sigma}}^j\|\llpg^{(p-j)}\|_{L^q_T\Cc^{-\eps}}\Big(\int_0^Ts^{-\big(\frac{\sigma+\eps}2\big)jq'}ds\Big)^{\frac1{q'}}\\
    &\qquad\qquad\qquad\qquad+|b_0|T^{\frac1{q'}}\|\llpg^{(p)}\|_{L^q_T\Cc^{-\eps}}\Big\}\\
    &\le Cr+C'(1+r+R)^{2m-1}T^{\frac1{q'}-\frac{\sigma+\eps}{2}(2m-1)},
    \end{align*}
    for some constants $C,C'>0$ independant of $r$, $R$, and $T$. Thus, taking $T\sim \big(\frac{r}{(1+r+R)^{2m-1}}\big)^{\theta}$ with $\theta=\big(\frac1{q'}-\frac{\sigma+\eps}{2}(2m-1)\big)^{-1}>0$, we get
    \begin{align*}
    \big\|\Phi_{(w_0,\llpg)}(w)\big\|_{L^\infty_T\Cc^{-\eps}}\le 2Cr.
    \end{align*}
   Similarly as above,
   \begin{align*}
    &\sup_{0<t\le T}t^{\frac{\sigma+\eps}{2}}\big\|\Phi_{(w_0,\llpg)}(w)(t)\big\|_{\Cc^{\sigma}}\\
    &\les \|w_0\|_{\Cc^{-\eps}}+\sum_{p=0}^{2m-1}\Big\{|b_0|\sup_{0<t\le T}t^{\frac{\sigma+\eps}{2}}\int_0^t(t-s)^{-\frac{\sigma+\eps}{2}}\|\llpg^{(p)}(s)\|_{\Cc^{-\eps}}ds\\
    &\qquad\qquad+\sum_{j=0}^p|b_p|\sup_{0<t\le T}t^{\frac{\sigma+\eps}{2}}\int_0^t (t-s)^{-\frac{\sigma+\eps}2}\|w^j(s)\llpg^{(p-j)}(s)\|_{\Cc^{-\eps}}ds\Big\}\\
     &\les \|w_0\|_{\Cc^{-\eps}}+\sum_{p=0}^{2m-1}\Big\{|b_0|\|\llpg^{(p)}\|_{L^q_T\Cc^{-\eps}}\sup_{0<t\le T}t^{\frac{\sigma+\eps}{2}}\Big(\int_0^t(t-s)^{-\frac{\sigma+\eps}{2}q'}ds\Big)^{\frac1{q'}}\\
    &\qquad\qquad+\sum_{j=1}^p|b_p|\|w\|_{X_T^{-\eps,\sigma}}^j\|\llpg^{(p-j)}\|_{L^q_T\Cc^{-\eps}}\sup_{0<t\le T}t^{\frac{\sigma+\eps}{2}}\Big(\int_0^t(t-s)^{-\frac{\sigma+\eps}{2}q'}s^{-\big(\frac{\sigma+\eps}2\big)jq'}ds\Big)^{\frac1{q'}}\Big\}\\
    &\le Cr+C'(1+r+R)^{2m-1}T^{\frac1{q'}-\frac{\sigma+\eps}{2}(2m-1)},
   \end{align*}
where we used the elementary integral inequality
\begin{align*}
\int_0^t(t-s)^{-a}s^{-b}ds\les t^{1-a-b}
\end{align*}   
   for any $0<a,b<1$ with $a+b>1$.
   Putting the two estimates together with the choice of $T$ yields
   \begin{align*}
   \big\|\Phi_{(w_0,\llpg)}(w)\big\|_{X_T^{-\eps,\sigma}}\le 4Cr.
   \end{align*}
    As for the time continuity, take $0<t\leq T$ and $h$ small enough, then split the integral 
    \begin{align*}
        \Phi_{(w_0,\llpg)}(w)(t+h)-\Phi_{(w_0,\llpg)}(w)(t)&=e^{-t\H}(e^{-h\H}-1)w_0+\int_0^te^{-(t-s)\H}(e^{-h\H}-1)\mathbf{f}(w,\llpg)(s)\,ds\\
        &\qquad+\int_t^{t+h}e^{-(t+h-s)\H}\mathbf{f}(w,\llpg)(s)\,ds\\
        &=:(1) + (2) + (3).
    \end{align*}
    We use both the Schauder estimate \eqref{SchauderH} and the continuity of the semigroup given by Lemma~\ref{ContinuitySemiGroupH} to investigate the convergence $h\to0$. This is straightforward for $(1)$, and for $(2)$ we get 
    \begin{align*}
        \left\|(2)\right\|_{\CC^{\sigma}}&\lesssim \int_0^t(t-s)^{-\frac{\sigma+\varepsilon}{2}}\|(e^{-h\H}-1)\mathbf{f}(w,\llpg)(s)\|_{\CC^{-\varepsilon}}\,ds\underset{h\to0}{\to}0
    \end{align*}
    in view of previous computations and the dominated convergence theorem. Same goes for $(3)$, changing variables and proceeding as above we get
    \begin{align*}
        \left\|(3)\right\|_{\CC^{\sigma}}&\lesssim \int_t^{t+h}(t+h-s)^{-\frac{\sigma+\varepsilon}{2}}\|\mathbf{f}(w,\llpg)(s)\|_{\CC^{-\varepsilon}}\,ds\\
        &\lesssim \int_0^{h}\tau^{-\frac{\sigma+\varepsilon}{2}}\|\mathbf{f}(w,\llpg)(t+h-\tau)\|_{\CC^{-\varepsilon}}\,ds\\
        &\lesssim \left(\int_0^{h}\tau^{-q'\frac{\sigma+\varepsilon}{2}}(t+h-s)^{-q'(2m-1)\frac{\sigma+\varepsilon}{2}}\,ds\right)^{\frac{1}{q'}}\|w\|_{X_T^{-\eps,\sigma}}^{2m-1}\|\llpg\|_{L^q_{T_0}\CC^{-\varepsilon}}\\
        &\lesssim t^{-(2m-1)\frac{\sigma+\varepsilon}{2}}h^{\frac{1}{q'}-\frac{\sigma+\varepsilon}{2}}\|w\|_{X_T^{-\eps,\sigma}}^{2m-1}\|\llpg\|_{L^q_{T_0}\CC^{-\varepsilon}}\\
        &\underset{h\to0}{\to}0.
    \end{align*}
    This proves continuity of $\Phi(w)$ as a $\CC^\sigma(\M)$-valued map on $(0,T]$. Continuity in $\CC^{-\varepsilon}(\M)$ on $[0,T]$ is obtained via the same arguments, without the blow-up at $t=0$. Altogether, this shows that $\Phi_{(w_0,\llpg)}$ maps the ball of radius $4Cr$ of $X_T^{-\eps,\sigma}$ to itself.
   
   Now for $w_1,w_2$ in this ball, proceeding as above we also have
 \begin{align*}
    &\big\|\Phi_{(w_0,\llpg)}(w_1)-\Phi_{(w_0,\llpg)}(w_2)\big\|_{L^\infty_T\Cc^{-\eps}}\\
    &\les \int_0^T\big\|\mathbf{f}(w_1,\llpg)(s)-\mathbf{f}(w_2,\llpg)(s)\big\|_{\Cc^{-\eps}}ds\\
    &\les \sum_{p=1}^{2m-1}\sum_{j=1}^p|b_p|\int_0^T\big\|(w_1^j-w_2^j)\llpg^{(p-j)}\big\|_{\Cc^{-\eps}}ds\\
    &\les \sum_{p=1}^{2m-1}\sum_{j=1}^p|b_p|\int_0^T\|w_1(s)-w_2(s)\|_{\Cc^{\sigma}}\big(\|w_1(s)\|_{\Cc^\sigma}^{j-1}+\|w_2(s)\|_{\Cc^\sigma}^{j-1}\big)\|\llpg^{(p-j)}(s)\|_{\Cc^{-\eps}}ds\\
     &\le C'(1+r+R)^{2m-2}T^{\frac1{q'}-\frac{\sigma+\eps}{2}(2m-1)}\|w_1-w_2\|_{X_T^{-\eps,\sigma}},
    \end{align*}
    and
 \begin{align*}
&\sup_{0<t\le T}t^{\frac{\sigma+\eps}{2}}\big\|\Phi_{(w_0,\llpg)}(w_1)(t)-\Phi_{(w_0,\llpg)}(w_2)(t)\big\|_{\Cc^{\sigma}}\\
&\qquad\qquad\le C'(1+r+R)^{2m-2}T^{\frac1{q'}-\frac{\sigma+\eps}{2}(2m-1)}\|w_1-w_2\|_{X_T^{-\eps,\sigma}}.
    \end{align*}
    With our choice of $T$ this shows that $\Phi_{(w_0,\llpg)}$ is also a contraction on this ball, thus admits a unique fixed point $w$ which is then a mild solution to \eqref{model} in $X_T^{-\eps,\sigma}$ which satisfies \eqref{localbound}. Similar estimates yield for solutions $w,\widetilde{w}$ with data $(w_0,\llpg),(\widetilde{w_0},\widetilde{\llpg})$:
    \begin{align*}
\|w-\widetilde{w}\|_{X_T^{-\eps,\sigma}}&=\big\|\Phi_{(w_0,\llpg)}(w)-\Phi_{(\widetilde{w_0},\widetilde{\llpg})}(\widetilde{w})\big\|_{X_T^{-\eps,\sigma}}\\
&\le C\|w_0-\widetilde{w_0}\|_{\Cc^{-\eps}}+C'(1+r+R)^{2m-2}T^{\frac1{q'}-\frac{\sigma+\eps}{2}(2m-1)}\|w-\widetilde{w}\|_{X_T^{-\eps,\sigma}}\\
&\qquad\qquad+C'(1+r)^{2m-2}T^{\frac1{q'}-\frac{\sigma+\eps}{2}(2m-1)}\|\llpg-\widetilde{\llpg}\|_{L^q_T\Cc^{-\eps}}.
    \end{align*}
    This shows the Lipschitz dependence on $(w_0,\llpg)$.
\end{proof}

With the local well-posedness result of Proposition~\ref{Localwpw}, we can prove our main local well-posedness result of Theorem~\ref{THM:LWP}.
\begin{proof}[Proof of Theorem~\ref{THM:LWP}]
    Let $\eps,\sigma,q$ be as in the statement of Theorem~\ref{THM:LWP}. Then for $T_0,R>0$, $0\leq k\leq 2m-1$ and $N\geq1$, define the events 
    \begin{align*}
        \Sigma_R^{N,k}&:=\Big\{\omega\in\Omega, \textup{such that}~\llp\in C([0;T_0];\Cc^{-\eps}(\M)),~~\llp^{\diamond k}\in L^q([0;T_0];\Cc^{-\eps}(\M)),\\
        &\qquad\qquad\|(\P_N\llp)^{\diamond k}-\llp^{\diamond k}\|_{L^q_{T_0}\CC^{-\varepsilon}}\leq N^{-\frac{\kappa}{2}}R,\ \|(\P_N\llp)^{\diamond k}\|_{L^q_{T_0}\CC^{-\varepsilon}}\leq R \Big\},
    \end{align*}
    where $\kappa$ is as in Proposition \ref{RenormWick}. Then set
    $$
    \Sigma_R:=\bigcap_{N\in\N^*}\bigcap_{0\leq k\leq 2m-1}\Sigma_R^{N,k}
    $$
    and
    $$
    \Sigma:=\liminf_{R\to+\infty}\Sigma_R.
    $$
    From the bounds \eqref{EstimateTailLoliN} and \eqref{EstimateDeviationLoliN} we get 
    $$
    \mathbb{P}\left(\Omega\setminus\Sigma_R\right) \leq C e^{-cR^\frac{1}{m}}
    $$
    for some constants $c,C>0$. Then Borel-Cantelli lemma gives that $\Sigma$ is of full probability.
    
    Given $R>0$, we get from Proposition~\ref{Localwpw} that there exists a time $T=T(\|u_0\|_{\Cc^{-\eps}},R)\in(0,T_0\wedge 1]$ such that for any $\omega\in\Sigma_R$ and $N\geq1$, \eqref{model} admits a unique solution $w_N\in X_T^{-\eps,\sigma}$ with data $w_N(0)=u_0-\llp(0)$. We then obtain a unique solution $u_N$ to \eqref{TruncEquation} by setting
    $$
    u_N=\llp + w_N\in \llp + X_T^{-\sigma,\eps}.
    $$
    Moreover, since $(\P_N\llp)^{\diamond k}\to \llp^{\diamond k}$ in $L^q([0;T_0];\CC^{-\varepsilon}(\M))$, the Lipschitz continuity property of the solution in Proposition~\ref{Localwpw} ensures that $w_N$ converges in $X_T^{-\eps,\sigma}$ to the solution $w$ to \eqref{model} with data $(u_0-\llp(0),\llpg)$. Since $u_N=\llp+w_N$, this shows that $u_N$ converges in $\llp+X_T^{-\eps,\sigma}$ to $u=\llp+w$.
\end{proof}

\subsection{Almost sure globalization and measure invariance}\label{SEC:global}

We now turn to the globalization problem of the truncated dynamics \eqref{TruncEquation} for random initial data $u_0$ with law given by $\rho_N$ defined as in \eqref{TruncatedGibbsMeasure}. Note that since Theorem~\ref{THM:LWP} holds $\Prob$ almost surely for any initial data in $\Cc^{-\eps}(\M)$, in view of Proposition~\ref{RenormWick}, it then also holds $\rho\otimes \Prob$ almost surely. Moreover, recall that Proposition~\ref{Localwpw} provides a blow-up alternative: since the local time of existence $T=T(\|u_0\|_{\Cc^{-\eps}},\|\llpg\|_{L^q_{T_0}\Cc^{-\eps}})$ only depends on $\|u_0\|_{\Cc^{-\eps}}$ and $\|\llpg\|_{L^q_{T_0}\Cc^{-\eps}}$, for any $T_0>0$, if $T^*=T^*(u_0,\omega)>0$ is the maximal time of existence of $w_N$ on $[0,T_0]$, then
\begin{align}\label{blowup}
\sup_{0\leq t< T^*}\|w_N(t)\|_{\CC^{-\varepsilon}}=+\infty\qquad\textup{or}\qquad T^*=T_0.
\end{align}
As we intend to prove global well-posedness for random initial data, we first need to investigate measure invariance of $\rho_N$ under the flow of the truncated equation \eqref{TruncEquation}. As $\P_N$ does not make the truncated equation \eqref{TruncEquation} a finite-dimensional SDE, we first study another truncated dynamics. Let $M>0$ and set $\Pi_M$ to be the orthogonal projection of the finite dimensional space $\textup{Span}\left\{\varphi_n\,,\ \lambda_n\leq M\right\}$, we denote by $d_M$ its dimension, with $d_M\simeq M$ according to Weyl's law \eqref{WeylH}. $\Pi_M$ being a sharp projection, we do not recover continuity in Hölder spaces, so we rather consider a smooth counterpart $\chi_M=\chi(M^{-2}\H)$ for some $\chi\in C^{\infty}_0(\R)$ compactly supported in $(-1;1)$. Note that $\P_N$ is based on $\Delta$ while $\chi_M$ is defined with $\H$, hence both operators do not commute with one another. Still, we consider $u_{N,M}$ to solve the following equation starting from some initial data $u_0$:
\begin{align}\label{TruncNM}
        \begin{cases}
        \partial_t u_{N,M} + \H u_{N,M} + \chi_M\P_N f_N^{\diamond}\left(\P_N\chi_M u_{N,M}\right)=\sqrt{2}\zeta,\\
        u_{N,M}(0)=u_0.
    \end{cases}
\end{align}

\noindent Expanding $f_N^\diamond$ to emphasis the dependency in $N$, we can rewrite \eqref{TruncNM} as
\begin{align}\label{Equationu_NM}
  \begin{cases}
  \partial_t u_{N,M} + \H u_{N,M} + \sum_{p=0}^{2m-1}b_p\chi_M\P_N H_p\left(\P_N\chi_M u_{N,M},\sigma_N\right)=\sqrt{2}\zeta,\\
        u_{N,M}(0)=u_0.
   \end{cases}
\end{align}

\noindent Note that $\chi_M$ acts on both the non-linearity and the $u_{N,M}$ term inside, and since $\chi_M=\Pi_M\chi_M=\chi_M\Pi_M$ by the support property of $\chi$, we see that $\Pi_Mu_{N,M}$ will solve a finite dimensional system of coupled SDEs while $(1-\Pi_M)u_{N,M}$ satisfies a linear evolution equation driven by a white noise. Namely, if we write $\Pi_Mu_{N,M}(t)=\sum_{n=0}^{d_M}a_{N,M}^{(n)}(t)\varphi_n$, then for $0\leq n\leq d_M$ we have 
$$
da_{N,M}^{(n)}+\left[\lambda_na_{N,M}^{(n)} + \sum_{p=0}^{2m-1} b_p\left\langle \chi_M\P_NH_p\left(\P_N \chi_M\sum_{n=0}^{d_M}a_{N,M}^{(n)}\varphi_n, \sigma_N\right),\varphi_n\right\rangle\right]dt=\sqrt{2}\,dB_n
$$
for independent Brownian motions $B_n$. This can be seen as a finite dimensional stochastic gradient flow
$$
da^{(n)}+\partial_{a^{(n)}}\CE_{N,M}(a^{(0)},\cdots,a^{(d_M)})\,dt=\sqrt{2}\,dB_n\quad\text{for}\ 0\leq n\leq d_M
$$
with truncated energy
$$
\CE_{N,M}(a^{(0)},\cdots,a^{(d_M)}):=\frac{1}{2}\sum_{n=0}^{d_M}\lambda_n(a^{(n)})^2+\sum_{p=0}^{2m-1} \frac{b_p}{p+1}\int_\M H_{p+1}\left(\P_N \chi_M\sum_{n=0}^{d_M}a_{N,M}^{(n)}\varphi_n, \sigma_N\right)\,dx.
$$
\noindent To this finite dimensional system correspond the generator $\mathcal{L}_{N,M}$ and Gibbs measure $\rho_{N,M}$ defined respectively by
$$
\mathcal{L}_{N,M}:= \sum_{n=0}^{d_M} \partial_{a^{(n)}}^2 - \partial_{a^{(n)}}\CE_{N,M}(a^{(0)},\cdots,a^{(d_M)})\partial_{a^{(n)}}
$$
and 
\begin{align*}
    d\rho_{N,M}(a^{(0)},\cdots,a^{(d_M)})&:= \mathcal{Z}_{N,M}^{-1}e^{-\CE_{N,M}(a^{(0)},\cdots,a^{(d_M)})}\,da^{(0)}\cdots da^{(d_M)}\\
    &=\mathcal{Z}_{N,M}^{-1}R_{N,M}(a^{(0)},\cdots,a^{(d_M)})d((\Pi_M)_*\mu^\H)(a^{(0)},\cdots,a^{(d_M)}),
\end{align*}
with normalization constant $\mathcal{Z}_{N,M}$. Integrating $\mathcal{L}_{N,M}f$ against $d\rho_{N,M}$ for some test function $f$ we get (denoting $a^{(0-d_M)}=(a^{(0)},\cdots,a^{(d_M)})$)
\begin{align*}
    \int_{\R^M}\mathcal{L}_{N,M}f(a^{(0-d_M)})\,d\rho_{N,M}(a^{(0-d_M)})&=\mathcal{Z}_{N,M}^{-1}\int_{\R^M}e^{-\CE_{N,M}(a^{(0-d_M)})}\sum_{n=0}^{d_M}\left(\partial_{a^{(n)}}^2f(a^{(0-d_M)})\right.\\
    &\qquad- \left.\partial_{a^{(n)}}\CE_{N,M}\partial_{a^{(n)}}f(a^{(0-d_M)})\right)\,da^{(0)}\cdots da^{(d_M)}\\
    &=0 
\end{align*}
after integration by parts. This proves that $\Pi_Mu_{N,M}$ leaves the measure $\rho_{N,M}$ invariant.\\
Now for remaining part $u_{N,M}^\bot:= (1-\Pi_M)u_{N,M}$, the non-linear term disappears since $(1-\Pi_M)\chi_M=0$ by the support property of $\chi$, so that $u_{N,M}^\bot$ satisfies the linear equation
$$
\partial_tu_{N,M}^\bot + \H u_{N,M}^\bot = \sqrt{2}(1-\Pi_M)\zeta.
$$
This equation has as unique invariant measure $(1-\Pi_M)_*\mu^\H$. Therefore, $u_{N,M}=\Pi_Mu_{N,M}+u_{N,M}^\bot$ leaves the measure $\rho_{N,M}\otimes (1-\Pi_M)_*\mu^\H$ invariant. The only part missing to make this whole discussion rigorous is the global existence of a solution $u_{N,M}$ in $\CC^{-\varepsilon}(\M)$, hence the following proposition.

\begin{proposition}\label{PROP:invarianceNM}
    For fixed $N,M\in\N^*$, and initial data $u_0\in\CC^{-\eps}$, then the flow of \eqref{Equationu_NM} extends globally. Moreover the truncated measure $\rho_{N,M}\otimes(1-\Pi_M)^*\mu^\H$ is invariant under the flow of \eqref{Equationu_NM}.
\end{proposition}

\begin{proof}
    First we prove global well-posedness. As per usual set $u_{N,M}=\llp+w_{N,M}$ where $\llp$ is the stochastic convolution as before and $w_{N,M}$ solves the non-linear equation
    \begin{align}\label{vNM}
        \begin{cases}
            \partial_t w_{N,M} + \H w_{N,M} + \sum_{p=0}^{2m-1}b_p\chi_M\P_N H_p\left(\P_N\chi_M     w_{N,M}+\P_N\chi_M\llp,\sigma_N\right)=0,\\
            w_{N,M}(0)=u_0-\llp(0).
            \end{cases}
    \end{align}
    Then, due to Proposition~\ref{PROP:llpM}, local existence of $w_{N,M}$ falls under the scope of Proposition~\ref{Localwpw}, and we only need to prove that solutions exist globally in time. According to the blow-up criterion above, we only have to make sure that the $\CC^{-\epsilon}(\M)$ norm of $w_{N,M}$ does not blow-up in finite time. Note that applying $\Pi_M$ to \eqref{vNM}, we see that $\Pi_Mw_{N,M}$ solves \eqref{vNM} with initial data $\Pi_M\big(u_0-\llp(0)\big)$ while $(1-\Pi_M)w_{N,M}$ solves a linear equation, thus exists globally. Note also that we have by Sobolev inequality
    $$
    \|\Pi_Mw_{N,M}(t)\|_{\CC^{-\epsilon}}\les \|\Pi_Mw_{N,M}(t)\|_{H^{1-\frac\eps2}}\sim \|\Pi_Mw_{N,M}(t)\|_{\D^{1-\frac\eps2}}\lesssim_M\|\Pi_Mw_{N,M}(t)\|_{L^2}.
    $$
    Thus we only need to bound the $L^2(\M)$ norm of $\Pi_Mw_{N,M}$ to get global existence for each fixed $N,M$. This is done by estimating the following energy
    $$
    \Psi_{N,M}(t):=\frac{1}{2}\int_\M |\Pi_Mw_{N,M}(s,x)|^2\,dx + b_{2m-1}\int_0^t\int_\M |\P_N \chi_M \Pi_Mw_{N,M}(s,x)|^{2m}\,ds\,dx.
    $$
    We prove that $\Psi_{N,M}$ is bounded on any interval $[0,T]$ by some finite positive constant, depending on $T$. Dropping the ${N,M}$ subscript in $\Pi_Mvw_{N,M}=:w$ for the sake of clarity and differentiating in $t$, we get
    \begin{align*}
        &\frac{d}{dt}\Psi_{N,M}\\
        &=\int_\M v\partial_tw + b_{2m-1}\int_\M |\P_N\chi_M w|^{2m}\\
        &=-\int_\M |\H^{\frac{1}{2}}w|^2 - \int_\M f_N^\diamond(\P_N\chi_Mw+\P_N\chi_M\llp)\P_N\chi_Mw + b_{2m-1}\int_\M |\P_N\chi_Mw|^{2m}\\
        &=-\int_\M |\H^{\frac{1}{2}}w|^2 - \sum_{p=0}^{2m-1}\sum_{j=0}^p\binom{p}{j}b_p\int_\M (\P_N\chi_Mw)^{p-j+1}(\P_N\chi_M\llp)^{\diamond j} + b_{2m-1}\int_\M |\P_N\chi_Mw|^{2m}\\
        &=-\int_\M |\H^{\frac{1}{2}}w|^2 - \sum_{p=0}^{2m-1}\sum_{j=0}^p\mathbf{1}_{(j,p)\neq (0,2m-1)}\binom{p}{j}b_p\int_\M (\P_N\chi_Mw)^{p-j+1}(\P_N\chi_M\llp)^{\diamond j}.
    \end{align*}
    
    \noindent Having only powers at most $2m-1$ of $\P_N\chi_Mw$, we can fix $\delta\in(0,1)$ and use Young's estimate to get
    \begin{align*}
        \left|\int_\M (\P_N\chi_Mw)^{p-j+1}(\P_N\chi_M\llp)^{\diamond j}\right|&\leq \|\P_N\chi_Mw\|_{L^{p-j+1}}^{p-j+1}\|(\P_N\chi_M\llp)^{\diamond j}\|_{L^\infty}\\
        &\leq C_{j,p}\delta^{\frac{2m}{p-j+1}}\|\P_N\chi_Mw\|_{L^{2m}}^{2m} + C_{j,p}(\delta)\|(\P_N \chi_M\llp)^{\diamond j}\|_{L^\infty}^{\frac{2m}{2m-p+j-1}}\\
        &\leq C_{j,p}\delta\|\P_N\chi_Mw\llp\|_{L^{2m}}^{2m} + C_{j,p}(\delta)\|(\P_N\chi_M\llp)^{\diamond j}\|_{L^\infty}^{\frac{2m}{2m-p+j-1}}.
    \end{align*}
    
    \noindent And then after a proper rescale of $\delta$ :
    
    \begin{align*}
        &\left|\sum_{p=0}^{2m-1}\sum_{j=0}^p\mathbf{1}_{(j,p)\neq (0,2m-1)}\binom{p}{j}b_p\int_\M (\P_N\chi_Mw)^{p-j+1}(\P_N\chi_M\llp)^{\diamond j}\right|\\
        &\qquad\qquad\leq \delta \|\P_N\chi_Mw\|_{L^{2m}}^{2m}+C(\delta,T,N,M),
    \end{align*}
    
    \noindent for some constant $C$. Then, as our non-linearity is defocusing ($b_{2m-1}>0$), integration in time yields
    \begin{align*}
        \Psi_{N,M}(t)&\leq \delta \Psi_{N,M}(t) + C(\delta,T,N,M),
    \end{align*}
    \noindent and then $\Psi_{N,M}(t)$ is bounded on $[0,T]$. Since $\Psi_{N,M}$ controls the $L^2(\M)$ norm of $w=\Pi_Mw_{N,M}$, we obtain global existence of $w_{N,M}$, and then global existence of $u_{N,M}$.\\
    Now that we have proved that \eqref{Equationu_NM} is globally well-posed, the previous discussion ensures that $\rho_{N,M}\otimes(1-\Pi_M)^*\mu^\H$ is invariant under the flow of \eqref{Equationu_NM}.
\end{proof}

\noindent We now turn to the limiting procedure by taking $M,N\to+\infty$. At first we let $M\to+\infty$, this will prove the following bound on the local flow $\Phi_N$ of \eqref{TruncEquation}.

\begin{lemma}
    For any $T>0$ and $N\geq1$, there exists a constant $C_T$ uniform in $N\in\N^*$ such that
    \begin{align}{\label{BoundProbaPhi_N}}
            \E_{\rho_N\otimes\Prob}\Big[\sup_{0\leq t\leq T\wedge T^*}\big\|\Phi_N(u_0,\omega)(t)\big\|_{\CC^{-\epsilon}}\Big]\leq C_T,
    \end{align}
    where $T^*=T^*_N(u_0,\omega)$ is the maximal time of existence of $u_N=\Phi_N(u_0,\omega)$.
\end{lemma}

\begin{proof}
    Let $T>0$ and $N,M\in\N^*$, we first prove a similar bound for $\Phi_{N,M}$ the flow of \eqref{Equationu_NM}. From Proposition \ref{PROP:invarianceNM}, $\Phi_{N,M}$ is globally well-defined by
    $$
    \Phi_{N,M}(u_0,\omega)(t)=\llp(\omega)(t)+w_{N,M}(t)
    $$
    where we see $w_{N,M}(t)$ as a continuous function of the stochastic data $(\llp^{\diamond p})_{0\leq p\leq 2m-1}$, as showed in Proposition~\ref{Localwpw}. Thus $\Phi_{N,M}$ satisfies the Duhamel formula for $t\geq0$ 
\begin{align*}
    \Phi_{N,M}(u_0,\omega)(t)&=e^{-t\H}\big(u_0-\llp(\omega)(0)\big)+\llp(\omega)(t)\\
    &\qquad-\int_0^te^{-(t-s)\H}\chi_M\P_Nf_N^{\diamond}(\P_N\chi_M\Phi_{N,M}(u_0,\omega)(s))\,ds.
\end{align*}
    Hence taking the $\CC^{-\epsilon}(\M)$ norm and then the supremum in $t$ yields 
    \begin{align*}
    \sup_{0\leq t\leq T}\|\Phi_{N,M}(u_0,\omega)(t)\|_{\CC^{-\epsilon}}&\les \sup_{0\leq t\leq T}\|e^{-t\H}\big(u_0-\llp(\omega)(0)\big)+\llp(\omega)(t)\|_{\CC^{-\epsilon}}\\
    &\qquad+\int_0^T\|\chi_M\P_Nf_N^{\diamond}(\P_N\chi_M\Phi_{N,M}(u_0,\omega)(s))\|_{\CC^{-\epsilon}}\,ds\\
 &\les\sup_{0\leq t\leq T}\|e^{-t\H}\big(u_0-\llp(\omega)(0)\big)+\llp(\omega)(t)\|_{\CC^{-\epsilon}}\\
    &\qquad+\int_0^T\|f_N^{\diamond}(\P_N\chi_M\Phi_{N,M}(u_0,\omega)(s))\|_{\CC^{-\frac\eps2}}\,ds,
     \end{align*}
    where we used successively Corollary~\ref{COR:PMH} and Lemma~\ref{LEM:PM} for the boundedness properties of $\chi_M$ and $\P_N$ on $\Cc^{-\eps}(\M)$. Note that $u_0\sim\rho_N$, so in particular we also have $u_0\in \Cc^{-\frac\eps2}$ almost surely, and thus also $\Phi_{N,M}(u_0)(t)\in \Cc^{-\frac\eps2}$, so that the last term above makes sense. Integrating with respect to $\rho_{N,M}\otimes(1-\Pi_M)^*\mu^\H\otimes\mathbb{P}$, Fubini's theorem and invariance of $\rho_{N,M}\otimes(1-\Pi_M)^*\mu^\H$ under $\Phi_{N,M}$ allow to deal with the integral in the right hand side:
    \begin{align*}
        &\E_{\rho_{N,M}\otimes(1-\Pi_M)^*\mu^\H\otimes\Prob}\Big[\int_0^T\|f_N^{\diamond}(\P_N\chi_M\Phi_{N,M}(u_0,\omega)(s))\|_{\CC^{-\frac\eps2}}\,ds\,\Big]\\
        &\hspace{2cm}=\int_0^T\E_\Prob\E_{\rho_{N,M}\otimes(1-\Pi_M)^*\mu^\H}\Big[\|f_N^{\diamond}(\P_N\chi_M\Phi_{N,M}(u_0,\omega)(s))\|_{\CC^{-\frac\eps2}}\Big]ds\\
        &\hspace{2cm}=T\E_{\rho_{N,M}\otimes(1-\Pi_M)^*\mu^\H}\Big[\|f_N^{\diamond}(\P_N\chi_Mu_0)\|_{\CC^{-\frac\eps2}}\Big]\\
        &\hspace{2cm}=T\E_{\mu^\H}\Big[\mathcal{Z}^{-1}_{N,M}R_{N}(\chi_Mu_0)\|f_N^{\diamond}(\P_N\chi_Mu_0)\|_{\CC^{-\frac\eps2}}\Big]\\
        &\hspace{2cm}\leq T\Big\|\|f_N^{\diamond}(\P_N\chi_Mu_0)\|_{\CC^{-\frac\eps2}}\Big\|_{L^2(\mu^\H)}\Big\|\mathcal{Z}^{-1}_{N,M}R_{N}(\chi_Mu_0)\Big\|_{L^2(\mu^\H)}.
    \end{align*}
    From there, we invoke Propositions~\ref{PROP:llpM} and~\ref{PROP:BGNM} to get a bound uniform in both $N$ and $M$. 
 
    For the other part, note that $e^{-t\H}\big(u_0-\llp(\omega)(0)\big)+\llp(\omega)(t)$ is explicitly given by 
    $$
    e^{-t\H}\big(u_0-\llp(\omega)(0)\big)+\llp(\omega)(t)=e^{-t\H}u_0+\int_0^te^{-(t-s)\H}\,\zeta(ds).
    $$
    We treat the propagator in a similar way we did for the Duhamel integral
    \begin{align*}
        \E_{\rho_{N,M}\otimes(1-\Pi_M)^*\mu^\H\otimes\Prob}\Big[\sup_{0\leq t\leq T}\|e^{-t\H}u_0\|_{\CC^{-\epsilon}}\Big]&\lesssim\E_{\rho_{N,M}\otimes(1-\Pi_M)^*\mu^\H}\Big[\|u_0\|_{\CC^{-\epsilon}}\Big]\\
        &\lesssim\E_{\mu^\H}\Big[\|u_0\|^2_{\CC^{-\epsilon}}\Big]^{\frac{1}{2}}\E_{\mu^\H}\Big[\mathcal{Z}_{N,M}^{-2}R_{N}(\chi_Mu_0)^2\Big]^{\frac{1}{2}},
    \end{align*}
    and again the right hand side is uniformly bounded in $N,M$. For the remaining part, we use the Burkholder-Davis-Gundy inequality and finally get 
    $$
    \E_{\rho_{N,M}\otimes(1-\Pi_M)^*\mu^\H\otimes\Prob}\Big[\sup_{0\leq t\leq T}\big\|e^{-t\H}u_0+\int_0^te^{-(t-s)\H}\,\zeta(ds)\big\|_{\CC^{-\epsilon}}\Big]\leq C_T
    $$
    for some constant $C_T$ uniform in $N$ and $M$. This proves that
    \begin{align*}
        \E_{\rho_{N,M}\otimes(1-\Pi_M)^*\mu^\H\otimes\Prob}\Big[\sup_{0\leq t\leq T}\|\Phi_{N,M}(u_0,\omega)(t)\|_{\CC^{-\epsilon}}\Big]\leq C_T.
    \end{align*}
    Now, from $C_T$ being uniform in $N,M$, the convergence almost surely and in $L^p(\Omega)$ of $(\P_N\chi_M\llp)^{\diamond k}$ to $(\P_N\llp)^{\diamond k}$ in $L^q([0;T];\Cc^{-\eps}(\M))$ (Proposition~\ref{PROP:llpM}) together with the continuity property of the flow given by Proposition~\ref{Localwpw}, and the convergence in total variation of $\rho_{N,M}\otimes(1-\Pi_M)^*\mu^\H$ to $\rho_N$ due to Propositions~\ref{PROP:llpM} and~\ref{PROP:BGNM}, we can use Fatou's lemma and pass $M$ to the limit so that all in all we have
    \begin{align*}
        &\E_{\rho_{N}\otimes\Prob}\Big[\sup_{0\leq t\leq T\wedge T^*}\|\Phi_{N}(u_0,\omega)(t)\|_{\CC^{-\epsilon}}\Big]\\
        &\qquad\leq \liminf_{M\to+\infty}\E_{\rho_{N,M}\otimes(1-\Pi_M)^*\mu^\H\otimes\Prob}\Big[\sup_{0\leq t\leq T\wedge T^*}\|\Phi_{N,M}(u_0,\omega)(t)\|_{\CC^{-\epsilon}}\Big]\\
        &\qquad\leq C_T,
    \end{align*}
    again uniformly in $N\in\N^*$.
\end{proof}

\noindent In view of the blow-up criterion \eqref{blowup} for the equation \eqref{NonLinearEqv_N}, thus for the equation \eqref{TruncEquation}, the estimate \eqref{BoundProbaPhi_N} is enough to prove probabilistic global well-posedness of the truncated renormalized dynamics \eqref{TruncEquation}.

\begin{proposition}{\label{GlobalWellPosednessu_N}}
    For any $N\geq1$, \eqref{TruncEquation} is $\rho_N\otimes\mathbb{P}$-almost surely globally well posed. Moreover the truncated Gibbs measure $\rho_N$ is invariant under the flow $\Phi_N$ of \eqref{TruncEquation}, that is, for any $t\geq0$ and $\mathfrak{F}\in C_b(\CC^{-\varepsilon}(\M))$,
    $$
    \E_{\rho_N\otimes\Prob}\Big[\mathfrak{F}\left[\Phi_N(u_0,\omega)(t)\right]\Big]= \E_{\rho_N}\Big[\mathfrak{F}\left[u_0\right]\Big].
    $$
\end{proposition}

\begin{proof}
    Fix $N\geq1$, then \eqref{BoundProbaPhi_N} reads as, uniformly in $N$, for each $T$ there exists an event $\Sigma_{N,T}$ of full $\rho_N^\H\otimes\mathbb{P}$-measure such that 
    $$
    \|\Phi_N(u_0,\omega)(t)\|_{\CC^{-\epsilon}}\le C_T
    $$
    for any $(u_0,\omega)\in\Sigma_{N,T}$. Specializing to times $k\in\N^*$, then the blow-up alternative ensures that $\Phi_N(u_0,\omega)$ is well defined at least up to time $k$. Taking the countable intersection, we see that 
    $$
    \Sigma_N:=\bigcap_{k\in\N^*}\Sigma_{N,T_k}
    $$
    is of full $\rho_N\otimes\mathbb{P}$-measure and that $\Phi_N$ is defined globally for any $(u_0,\omega)\in\Sigma_N$.\\
    Now for the invariance part, pick $t\geq0$ and $\mathfrak{F}\in C_b(\CC^{-\varepsilon}(\M))$.
    As in Proposition~\ref{UnifPartitionFunction}, we write $\rho_N$ as
    $$
    d\rho_N(u)=\mathcal{Z}_N^{-1} R_N(u)\,d\mu^\H(u).
    $$
    Then from the convergence in total variation of $\rho_{N,M}\otimes (1-\Pi_M)^*\mu^\H$ to $\rho_N$ and that of $\Phi_{N,M}$ to $\Phi_N$, we have
    \begin{align*}
        &\E_{\rho_N\otimes\Prob}\Big[ \mathfrak{F}\left[\Phi_N(u_0,\omega)(t)\right]\Big]=\lim_{M\to+\infty}\E_{\rho_{N,M}\otimes(1-\Pi_M)^*\mu^\H\otimes\Prob}\Big[\mathfrak{F}\left[\Phi_{N,M}(u_0,\omega)(t)\right]\Big],
    \end{align*}
    and from Proposition~\ref{PROP:invarianceNM} we can deduce that 
    \begin{align*}
        \E_{\rho_N\otimes\Prob}\Big[ \mathfrak{F}\left[\Phi_N(u_0,\omega)(t)\right]\Big]&=\E_{\rho_{N,M}\otimes(1-\Pi_M)^*\mu^\H\otimes\Prob}\Big[ \mathfrak{F}\left[\Phi_{N,M}(u_0,\omega)(t)\right]\Big]\\
        &=\lim_{M\to+\infty}\E_{\rho_{N,M}\otimes(1-\Pi_M)^*\mu^\H}\Big[\mathfrak{F}\left[u_0\right]\Big]\\
        &=\E_{\rho_{N}}\Big[\mathfrak{F}\left[u_0\right]\Big],
    \end{align*}
    proving that $\rho_N$ is invariant under the flow $\Phi_N$.
\end{proof}

\noindent Combining now the a priori estimate of Proposition~\eqref{BoundProbaPhi_N} together with the local convergence of $\Phi_N$ to $\Phi$ (Theorem~\ref{THM:LWP}) and that of $\rho_N$ to $\rho$ (Proposition~\ref{RenormWick}), we can finally get our almost sure global well-posedness result.

\begin{proof}[Proof of Theorem~\ref{THM:GWP}]
    Recall that the constant in \eqref{BoundProbaPhi_N} is uniform in $N$. Since Propositions~\ref{Localwpw} and~\ref{ConvergenceGibbsMeasure} ensure convergence of $\Phi_N$ and $\rho_N$ respectively almost surely and in total variation, Fatou's lemma yet again yields
    \begin{align*}
        \E_{\rho\otimes\Prob}\Big[\sup_{0\leq t\leq T\wedge T^*}\|\Phi(u_0,\omega)(t)\|_{\CC^{-\epsilon}}\Big]&\leq \liminf_{N\to+\infty}\E_{\rho_N\otimes\Prob}\Big[\sup_{0\leq t\leq T\wedge T^*}\|\Phi_{N}(u_0,\omega)(t)\|_{\CC^{-\epsilon}}\Big]\\
        &\leq C_T,
    \end{align*}
    where $T^*=T^*(u_0,\omega)$ is the maximal time of existence of $u=\Phi(u_0,\omega)$. As before, this yields the existence of a set\footnote{Recall that $\Sigma\subset \Omega$ is the set defined in Theorem~\ref{THM:LWP} where almost sure local well-posedness of \eqref{LimitEquation} holds.} $\tilde{\Sigma}\subset \Cc^{-\eps}(\M)\times\Sigma$ of full $\rho\otimes\mathbb{P}$-measure such that $\Phi(u_0,\omega)$ is defined globally for any $(u_0,\omega)\in\tilde{\Sigma}$. The invariance of $\rho$ then follows again from dominated convergence, the convergence of $\Phi_N$ to $\Phi$ (Theorem~\ref{THM:LWP}), the convergence of $\rho_N$ to $\rho$ (Theorem~\ref{ConvergenceGibbsMeasure}), and the invariance of $\rho_N$ under the flow $\Phi_N$ (Proposition~\ref{GlobalWellPosednessu_N}): for any $t\geq0$ and $\mathfrak{F}\in C_b(\CC^{-\epsilon}(\M))$ we get
    \begin{align*}
        \E_{\rho\otimes\Prob}\Big[ \mathfrak{F}\left[\Phi(u_0,\omega)(t)\right]\Big]&=\lim_{N\to+\infty}\E_{\rho_N\otimes\Prob}\Big[ \mathfrak{F}\left[\Phi_N(u_0,\omega)(t)\right]\Big]\\
        &=\lim_{N\to+\infty}\E_{\rho_N}\Big[\mathfrak{F}\left[u_0\right]\Big]\\
        &=\E_{\rho}\Big[\mathfrak{F}\left[u_0\right]\Big]
    \end{align*}
    This concludes the proof of Theorem~\ref{THM:GWP}.
\end{proof}

\section{A singular case: the Anderson Hamiltonian}{\label{SEC:Anderson}}

In this section, we explain the construction of the Anderson Hamiltonian and state its properties used throughout this work. It was first constructed by Allez and Chouk \cite{AllezChouk} on the two-dimensional torus and then extended to various framework, see \cite{GUZ,Labbe,Mouzard22} and references therein. It is the Schrödinger operator
\begin{equation*}
\H=-\Delta+\xi
\end{equation*}
with potential the spatial white noise on $\M$. As explained, the white noise is a random distribution of negative Hölder regularity and belongs almost surely to $\CC^{-1-\kappa}$ for any $\kappa>0$. In one dimension, the white noise is the derivative of the Brownian motion and was first constructed by Paley and Zygmund \cite{PaleyZygmund30,PaleyZygmund32}. In this case, $\xi\in\CC^{-\frac{1}{2}-\kappa}$ and the quadratic form $\langle(-\partial_x^2+\xi)u,u\rangle$ is well-defined with domain $H^1(\T)$, see Fukushima and Nakao \cite{FN}. Having a precise description of its domain and its large volume limit was the subject of recent work by Dumaz and Labbé, see \cite{DL} and references therein. In two dimensions, this operator is singular in the sense that a renormalisation procedure is needed to make sense of $\H$. This is hinted by the fact that the product $u\xi$ is singular for $u\in H^1(\M)$ hence the quadratic form is not well-defined. Following the recent progress in singular SPDEs, the operator can be defined as the limit of the renormalised operator
\begin{equation*}
\H=\lim_{\eps\to0}(-\Delta+\xi_\eps-c_\eps)
\end{equation*}
with $\xi_\eps$ a mollification of the noise and $c_\eps$ a logarithmic diverging quantity. Following the idea of paracontrolled calculus from Gubinelli, Imkeller and Perkowski \cite{GIP}, one can consider function $u$ paracontrolled by a functionnal of the noise as domain for the operator or its quadratic form. 

\subsection{Construction of the operator}{\label{SubAnderson}}

Consider the space
\begin{equation*}
\D^\sigma:=\{u\in L^2(\M)\ ;\ u-\PA_u X\in H^\sigma\}
\end{equation*}
with $X_1$ solution to $\Delta X_1=\xi$ and $\PA$ a paraproduct for $\sigma\ge1$. The notion of paraproduct goes back to Bony \cite{Bony} based on the Littlewood-Paley decomposition and can be extended to manifold, following for example \cite{Mouzard22}. The main idea is one can decompose a product as
\begin{equation*}
uv=\PA_uv+\PI(u,v)+\PA_vu
\end{equation*}
for any distribution $u\in\CC^\alpha$ and $v\in\CC^\beta$ where the resonant term $\PI(u,v)$ is well-defined only if $\alpha+\beta>0$, this is Young condition. Moreoever, this decomposition allows for a precise track of regularity with the continuity result
\begin{equation*}
\|\PA_uv\|_{B^{\alpha\wedge0+\beta}_{p,q}}\lesssim\|u\|_{B^{\alpha\wedge 0}_{p_1,\infty}}\|v\|_{B^\beta_{p_2,q}}
\end{equation*}
for any $\alpha,\beta\in\R$ and $1\le p_1,p_2,p,q\le\infty$ such that $\frac1p=\frac1{p_1}+\frac1{p_2}$, and
\begin{equation*}
\|\PI(u,b)\|_{B^{\alpha+\beta}_{p,q}}\lesssim\|u\|_{B^\alpha_{p_1,\infty}}\|v\|_{B^\beta_{p_2,q}}
\end{equation*}
for any $\alpha,\beta\in\R$ such that $\alpha+\beta>0$. The idea motivating the introduction of $\D^\sigma$ lies in the algebraic cancellation
\begin{align*}
-\Delta u+u\xi&=-\Delta\PA_uX_1-\Delta u^\sharp+\PA_u\xi+\PI(u,\xi)+\PA_\xi u\\
&=\PA_u(-\Delta X+\xi)-\Delta u^\sharp+[\Delta,\PA_u]X_1+\PI(u,\xi)+\PA_\xi u\\
&=-\Delta u^\sharp+[\Delta,\PA_u]X_1+\PI(u,\xi)+\PA_\xi u
\end{align*}
where $u=\PA_uX_1+u^\sharp\in\D^\sigma$. This shows that introducing roughness in $u$ via the paracontrolled expansion cancels exactly the most irregular term in $\H u$. However, this introduces a new singular product of distribution in $\PI(u,\xi)$. Indeed, $\xi\in\CC^{-1-\kappa}$ thus $X_1\in\CC^{1-\kappa}$ hence the sum of the regularity exponent $-1-\kappa+1-\kappa=-2\kappa<0$ barely fails Young condition. This is the singular nature of the operator and \cite{GIP} introduced the corrector $\DC$ to write
\begin{equation*}
\PI(u,\xi)=\PI(\PA_uX_1,\xi)+\PI(u^\sharp,\xi)=u\PI(X_1,\xi)+\DC(u,X_1,\xi)+\PI(u^\sharp,\xi)
\end{equation*}
for $u\in\D^\sigma$. For $\sigma>1$, the resonant product $\PI(u^\sharp,\xi)$ is well-defined and one can prove that $\DC$ is a continuous trilinear operator from $H^\alpha\times\CC^\beta\times\CC^\gamma$ to $H^{\alpha+\beta+\gamma}$ for $\alpha,\beta,\gamma\in\R$ such that
\begin{equation*}
\beta+\gamma<0<\alpha+\beta+\gamma<1.
\end{equation*}
The condition $\beta+\gamma<0$ corresponds exactly to the singular nature of the product $X_1\xi$ hence $\DC(u,X_1,\xi)$ is well-defined as a function as soon as $u$ is regular enough. This does not change that the resonant product $\PI(X_1,\xi)$ is singular hence the need for a renormalisation procedure, this can be done with the Wick product
\begin{equation*}
\PI(X_1,\xi)=\lim_{\eps\to0}\Big(\PI(X_{1,\eps},\xi_\eps)-\E\big[\PI(X_{1,\eps},\xi_\eps)\big]\Big)
\end{equation*}
with $\xi_\eps$ a mollification of the noise and $\Delta X_{1,\eps}=\xi_\eps$. After this procedure, one obtains that $\H$ is well-defined on $\D^\sigma$ for $\sigma>1$ and
\begin{equation*}
\|\H u+\Delta u^\sharp\|_{H^{-\kappa}}\lesssim\|u\|_{\D^{1+\eps}}
\end{equation*}
for any $\kappa,\eps>0$. In particular, the form quadratic $\langle\H u,u\rangle$ is well-defined and one can prove
\begin{equation*}
\langle \H u,u\rangle+c\|u\|_{L^2}^2\ge\frac{1}{2}\|u\|_{\D^1}^2\ge0
\end{equation*}
for a random constant $c>0$ using almost duality, see for example \cite[Proposition 2.9]{Mouzard22}. At this point, one can prove that the quadratic form $\langle\H u,u\rangle$ with domain $\D^1$ is closed, symmetric, continuous and bounded from below thus its associated operator is self-adjoint. One can then obtain a second order paracontrolled expansion of the domain as
\begin{equation*}
\D(\H)=\{u\in L^2(\M)\ ;\ u-\PA_uX\in H^2\}
\end{equation*}
with $X=X_1+X_2$ and $X_2$ solution to $\Delta X_2=\PI(X_1,\xi)+\PA_\xi X_1$, see \cite{Mouzard22} for this construction. The operator $\H$ has a compact resolvent with discrete spectrum $\lambda_0<\lambda_1\le\ldots$ with a basis of eigenfunctions $(\varphi_n)_{n\ge1}$, see \cite{OM} for a simple argument of the spectral gap. In the following, we suppose that the operator is positive considering for example that we include in the renormalisation the shift by the random constant $-\lambda_0+1$. Moreover, it was proved in \cite{Mouzard22} that the sequence of eigenvalues satisfies the same Weyl law \eqref{Weyl} as that of the Laplace-Beltrami operator, namely
\begin{align}\label{WeylH}
\ld_n \sim n.
\end{align}
The form domain $\D(\sqrt{H})=\D^1$ is the reciproque image of $H^1$ by the application
\begin{equation*}
\Phi(u)=u-\PA_uX
\end{equation*}
which is continuous from $H^\sigma$ to itself for any $\sigma<1$. Up to a random truncation of the noise which does not change the regularity, one can prove that $\Phi:H^\sigma\to H^\sigma$ is invertible as a perturbation of the identity, we denote by $\Gamma$ its inverse. It can be seen at the implicit solution to
\begin{equation}\label{Gamma}
\Gamma u^\sharp=\PA_{\Gamma u^\sharp}X+u^\sharp
\end{equation}
for $u^\sharp\in L^2(\M)$. The important idea is that the operator $\H\Gamma$ is a better behaved perturbation of the identity at the price that it is not self-adjoint anymore. Namely 
$$\H\Gamma u^\#=-\Delta u^\# + G_\xi(u^\#)$$
where $G_\xi$ is a bounded operator from\footnote{The cases $p=2$ and $p=+\infty$ are proved in \cite{MZ} or \cite{GIP}. The case $p=1$ follows from a straightforward modification of the computation in \cite{MZ}, and the case of general $p$ then follows from interpolation.} $W^{\sigma,p}$ to $W^{\sigma-1-\kappa,p}$ for any $\kappa>0$ provided the resonant product $\PI(u^\#,\xi)$ is well defined, i.e. $\sigma>1+\kappa$.\\
We introduce the conjugaison of $\H$ via the map $\Gamma$, that is
\begin{equation*}
\H^\sharp=\Gamma^{-1}\H\Gamma.
\end{equation*}
In particular, we have
\begin{equation*}
\|(\H^\sharp+\Delta)u\|_{\CC^{-\kappa}}\lesssim\|u\|_{\CC^{1+\eps}}
\end{equation*}
for any $\kappa,\eps>0$. This allows to compare the heat semigroup of associated to $\H^\sharp$ and $\Delta$ with
\begin{equation*}
e^{-t\H^\sharp}-e^{t\Delta}=\int_0^te^{(t-s)\Delta}(\H^\sharp+\Delta)e^{-s\H^\sharp} ds,
\end{equation*}
this is the main argument for the results in \cite{BDM}. The same argument for the Schrödinger with a second order expansion yields Strichartz inequalities, see \cite{MZ}. In particular, this perturbative argument gives Schauder estimates for the heat semigroup of $\H^\sharp$, that is
\begin{equation*}
\|e^{-t\H^\sharp}u\|_{\CC^\alpha}\lesssim t^{-\frac{\alpha-\beta}{2}}\|u\|_{\CC^\beta}
\end{equation*}
for $\alpha\in(1,2)$ and $0<\alpha-\beta<2$, see \cite[Theorem 22]{BDM}. This can be used to obtain Schauder estimates for the heat semigroup associated to $\H$. This will be important to prove local well-posedness for the Anderson stochastic quantization equation for deterministic initial data.

\begin{proposition}
For any $\alpha,\beta\in(-1,1)$ such that $\alpha>\beta$, we have
\begin{equation}\label{SchauderH}
\|e^{-t\H}u\|_{\CC^\alpha}\lesssim t^{-\frac{\alpha-\beta}{2}}\|u\|_{\CC^\beta}.
\end{equation}
\end{proposition}

\begin{proof}
First, one interpolates the bounds $\|e^{-t\H^\sharp}u\|_{\CC^\beta}\lesssim\|u\|_{\CC^\beta}$ and
\begin{equation*}
\|e^{-t\H^\sharp}u\|_{\CC^\alpha}\lesssim t^{-\frac{\alpha-\beta}{2}}\|u\|_{\CC^\beta}
\end{equation*}
for $\alpha\in(1,2)$ and $0<\alpha-\beta<2$. Let $\theta\in(0,1)$ and consider
\begin{equation*}
\gamma_\theta:=\theta\alpha+(1-\theta)\beta\in(\beta,\alpha).
\end{equation*}
Interpolation gives
\begin{align*}
\|e^{-t\H^\sharp}\|_{\CC^\beta\to\CC^{\gamma_\theta}}&\le \|e^{-t\H^\sharp}\|_{\CC^\beta\to\CC^\beta}^{1-\theta}\|e^{-t\H^\sharp}\|_{\CC^\beta\to\CC^\alpha}^\theta\\
&\lesssim t^{-\frac{\alpha-\beta}{2}\theta}\\
&\lesssim t^{-\frac{\gamma_\theta-\beta}{2}}
\end{align*}
since $\gamma_\theta-\beta=\theta\alpha+(1-\theta)\beta-\beta=(\alpha-\beta)\theta$ hence
\begin{equation*}
\|e^{-t\H^\sharp}u\|_{\CC^\alpha}\lesssim t^{-\frac{\alpha-\beta}{2}}\|u\|_{\CC^\beta}
\end{equation*}
for any $\alpha,\beta\in(-1,2)$ such that $\alpha>\beta$. For the result for $\H$, we have
\begin{align*}
\|e^{-t\H}u\|_{\CC^\alpha}&=\|\Gamma e^{-t\H^\sharp}\Gamma^{-1}u\|_{\CC^\alpha}\\
&\lesssim\|e^{-t\H^\sharp}\Gamma^{-1}u\|_{\CC^\alpha}\\
&\lesssim t^{-\frac{\alpha-\beta}{2}}\|\Gamma^{-1}u\|_{\CC^\beta}\\
&\lesssim t^{-\frac{\alpha-\beta}{2}}\|u\|_{\CC^\beta}
\end{align*}
using that $\Gamma$ and $\Gamma^{-1}$ are continuous from $\CC^\gamma$ to itself for $\gamma<1$ and the result for $\H^\sharp$.
\end{proof}

Natural spaces are the general Sobolev spaces associated to $\H$ defined as the closure of the vector space spanned by the eigenfunctions of $\H$ with respect to the norm
\begin{equation}{\label{SobolevH}}
\|u\|_{\D^\sigma}:=\Big(\sum_{n\ge1}\langle\lambda_n\rangle^\sigma|\langle u,\varphi_n\rangle|^2\Big)^{\frac{1}{2}}
\end{equation}
for $\sigma\in\R$. In particular, one immediatly has the Schauder estimates
\begin{equation*}
\|e^{-t\H}u\|_{\D^\alpha}\lesssim t^{-\frac{\alpha-\beta}{2}}\|u\|_{\D^\beta}
\end{equation*}
for $\alpha>\beta$. One can prove that $\D^\sigma=H^\sigma$ for $|\sigma|<1$, thus
\begin{equation*}
\|e^{-t\H}u\|_{H^\alpha}\lesssim t^{-\frac{\alpha-\beta}{2}}\|u\|_{H^\beta}
\end{equation*}
for $\alpha,\beta\in(-1,1)$ such that $\alpha>\beta$. With the previous result, this gives the general Schauder estimates
\begin{equation*}
\|e^{-t\H}u\|_{B_{p,p}^\alpha}\lesssim t^{-\frac{\alpha-\beta}{2}}\|u\|_{B_{p,p}^\beta}
\end{equation*}
for $p\in[2,\infty]$ and $\alpha,\beta\in(-1,1)$ such that $\alpha>\beta$ however we should only use the case $p=2$ or $p=\infty$ in this work.

\begin{lemma}{\label{ContinuitySemiGroupH}}
    For any $\sigma\in(-1,1)$ and $u\in\CC^\sigma(\M)$, then the map $t\mapsto e^{-t\H}u\in\CC^\sigma$ is continuous.
\end{lemma}

\begin{proof}
    First note that, by the semigroup property and \eqref{SchauderH}, continuity at 0 is enough. Let $u\in\Cc^{\sigma}(\M)$ and fix $\delta>0$. Then by density of $H^2(\M)$ in $\Cc^\sigma(\M)$ and continuity of $\Gamma$, we have that $\D(\H)$ is dense in $\Gamma\Cc^\sigma(\M)=\Cc^\sigma(\M)$. Thus there exists $v\in \D(\H)$ such that $\|u-v\|_{\Cc^\sigma}<\delta.$ Then we have for $h\ge 0$ small enough:
    \begin{align*}
        \big\|(e^{-h\H}-1)u\big\|_{\Cc^\sigma}&\le  \big\|(e^{-h\H}-1)(u-v)\big\|_{\Cc^\sigma}+ \big\|(e^{-h\H}-1)v\big\|_{\Cc^\sigma}.
    \end{align*}
    As for the first term, using the Schauder estimate \eqref{SchauderH}, we have
    \begin{align*}
        \big\|(e^{-h\H}-1)(u-v)\big\|_{\Cc^\sigma}\les \|u-v\|_{\Cc^\sigma} \les\delta,
    \end{align*}
    uniformly in $h$.
    For the second term, we use
    \begin{align*}
        \big\|(e^{-h\H}-1)v\big\|_{\Cc^\sigma}&\les \big\|(e^{-h\H}-1)v\big\|_{\D^{1+\sigma+\kappa}} \les h^{\frac{1-\sigma-\kappa}{2}}\|v\|_{\D^2}.
    \end{align*}
    Taking $h$ small enough depending on $\delta$, this finally shows that
    \begin{align*}
        \big\|(e^{-h\H}-1)u\big\|_{\Cc^\sigma}\le C\delta.
    \end{align*}
\end{proof}

\subsection{Anderson Green's function and GFF}\label{AndersonGreenGFF}

In order to get a precise bound on the truncated Green function of the Anderson operator, we first show that $G^\H$ and $G^{1-\Delta}$ have the same singularity. This is done by comparing both resolvents.

\begin{lemma}\label{LEM:GreenH}
Let $4\le p<\infty$ and $0<\delta<\frac12$. Then $\H^{-1}-(1-\Delta)^{-1}$ is bounded from $W^{-\frac2p-\delta,p'}$ to $W^{\frac2p+\delta,p}$. In particular $(G^\H-G^{1-\Delta})(x,y)\in L^\infty_y W^{\frac2p+\delta,p}_x$, and $G^\H(x,y)\in L^\infty_x W^{\frac2p-\delta,p}_y$.
\end{lemma}

\begin{proof}
$\H$ being self-adjoint and positive thus invertible, an application of the resolvent formula yields
    \begin{eqnarray*}
        \H^{-1}-(1-\Delta)^{-1}&=&(\H\Gamma)^{-1}-(1-\Delta)^{-1}+\H^{-1}-(\H\Gamma)^{-1}\\
        &=&(\H\Gamma)^{-1}((1-\Delta)-\H\Gamma)(1-\Delta)^{-1}+B(\H\Gamma)^{-1}
    \end{eqnarray*}
    following yet again the same paradigm as before: $\H$ is a perturbation of $-\Delta$ through the map $\Gamma$. By definition, $\Gamma$ is a perturbation of the identity, namely 
    \begin{equation}\label{decomp-Gamma}
        \Gamma = 1+\sum_{k\geq1}(\PA_\cdot  X)^k=:1+B.
    \end{equation}
    Note that, as $X\in\Cc^{1-\kappa}$ for any $\kappa>0$, we can write a nice continuity estimate for $B$. Indeed, for any $\nu\ge 0$ and $p>1$,
    $$
    \|\PA_uX\|_{W^{1-\nu-\kappa,p}}\leq c(\nu,\kappa,p)\|u\|_{W^{-\nu,p}}\|X\|_{\Cc^{1-\kappa}}
    $$
    where $\|X\|_{\Cc^{1-\kappa}}$ can be made arbitrarily small up to a random truncation of $X$ as in the definition of $\Gamma$. This proves that 
    $$
    B : W^{-\nu,p} \to W^{1-\nu-\kappa,p}
    $$
    is a bounded operator for any $\kappa>0$ and $\nu\ge 0$.\\
  As for $(\H\Gamma)^{-1}$, if $u\in\D^{\nu-2}$ and $v=(\H \Gamma)^{-1}u$ is the solution to $\H\Gamma v = -\Dl v + G_\xi v = u$, we have for any $\nu\in (1;2]$ by definition of $\D^\nu(\H) = \Gamma^{-1}H^\nu$:
  \begin{align*}
  \|v\|_{H^{\nu}}&=\big\|\Gamma \H^{-1} u\big\|_{H^{\nu}}= \big\|\H^{-1}u\big\|_{\D^\nu} = \|u\|_{\D^{\nu-2}}\sim \|u\|_{H^{\nu-2}}
  \end{align*}
    where in the last step we used that $\Gamma$ is invertible on $H^\sigma$ for $\sigma\in (-1;1)$.
    
    Combining all this with the resolvent formula above we obtain some $\kappa>0$ small enough: 
    \begin{align*}
        \left\|\H^{-1}u-(1-\Delta)^{-1}u\right\|_{W^{\frac2p+\delta,p}}&\les \big\|(\H\Gamma)^{-1}(1-G_\xi)(1-\Dl)^{-1}u\big\|_{H^{1+\delta}}+\big\|B(\H\Gamma)^{-1}u\big\|_{W^{\frac2p+\delta,p}}\\
        &\les \big\|(1-G_\xi)(1-\Dl)^{-1}u\big\|_{H^{-1+\delta}}+\big\|(\H\Gamma)^{-1}u\big\|_{W^{\frac2p+\delta+\kappa-1,p}}\\
        &\les \big\|(1-G_\xi)(1-\Dl)^{-1}u\big\|_{W^{-\frac2p+\delta,p'}}+\big\|(\H\Gamma)^{-1}u\big\|_{H^{\delta+\kappa}}\\
        &\les \big\|(1-\Dl)^{-1}u\big\|_{W^{\max(-\frac2p+\delta,0)+1+2\kappa,p'}}+\|u\|_{H^{\delta+\kappa-2}}\\
        &\les \|u\|_{W^{\max(-\frac2p+\delta,0)-1+2\kappa,p'}}+\|u\|_{W^{\delta+\kappa-\frac2p-1,p'}}\\
        &\les \|u\|_{W^{-\frac2p-\delta,p'}},
    \end{align*}
    where in the last step we used the restrictions on $p$ and $\delta$ in the statement of Lemma~\ref{LEM:GreenH}. This shows the boundedness properties of $\H^{-1}-(1-\Dl)^{-1}$, thus the regularity of $G^\H-G^{1-\Dl}$. The regularity of $G^\H$ then follows from that of $G^{1-\Dl}$ since $G^\H-G^{1-\Dl}$ is smoother.
\end{proof}

Next, we investigate the properties of the Green's function and its smoothened counterparts. Let us recall the following fundamental lemma that can be found in \cite[Lemma 2.8\,,\ 2.12\ and\ 2.13]{ORTW} :

\begin{lemma}\label{LEM:Green}
    \textup{(i)} There exists $C>0$ such that for any $N\in\N$ and $(x,y)\in\M\times\M\setminus\mathrm{diag}$ it holds
    \begin{align}\label{GN1}
        \Big|(\P_{N}\otimes \P_{N})G^{1-\Delta}(x,y) +\frac1{2\pi}\log\big(\dg(x,y)+N^{-1}\big)\Big| \le C.
    \end{align}
    \textup{(ii)} For all $0<\dl\ll1$ there exists $C>0$ such that for any $N_1\le N_2$ and $(x,y)\in\M\times\M\setminus\mathrm{diag}$ it holds for any $j\in\{1;2\}$:
    \begin{align}\label{GN2}
        &\Big|(\P_{N_j}\otimes \P_{N_j})G^{1-\Delta}(x,y)-(\P_{N_1}\otimes \P_{N_2})G^{1-\Delta}(x,y)\Big|\notag\\
        &\qquad\qquad\le C\min\Big\{-\log\big(\dg(x,y)+N_1^{-1}\big)\vee 1; N_1^{\dl-1}\dg(x,y)^{-1}\Big\}.
    \end{align}
\end{lemma}

Combining Lemmas~\ref{LEM:GreenH} and~\ref{LEM:Green}, we get the following result.

\begin{corollary}\label{COR:GHN}
 \textup{(i)} There exists $C>0$ such that for any $N\in\N$ and $(x,y)\in\M\times\M\setminus\mathrm{diag}$ it holds
    \begin{align}\label{GHN1}
        \Big|(\P_{N}\otimes \P_{N})G^{\H}(x,y) +\frac1{2\pi}\log\big(\dg(x,y)+N^{-1}\big)\Big| \le C.
    \end{align}
    \textup{(ii)} For all $0<\dl\ll1$ there exists $C>0$ such that for any $N_1\le N_2$ and $(x,y)\in\M\times\M\setminus\mathrm{diag}$ it holds for any $j\in\{1;2\}$:
    \begin{align}\label{GHN2}
        &\Big|(\P_{N_j}\otimes \P_{N_j})G^{\H}(x,y)-(\P_{N_1}\otimes \P_{N_2})G^{\H}(x,y)\Big|\notag\\
        &\qquad\qquad\le C\min\Big\{-\log\big(\dg(x,y)+N_1^{-1}\big)\vee 1; N_1^{\dl-1}\dg(x,y)^{-1}\Big\}.
    \end{align}

\end{corollary}

\begin{proof}
We start with the proof of \eqref{GHN1}. Due to \eqref{GN1}, we have
\begin{align*}
(\P_N\otimes\P_N)G^\H(x,y)=(\P_N\otimes \P_N)\big(G^\H-G^{1-\Dl})(x,y)-\frac1{2\pi}\log\big(\dg(x,y)+N^{-1}\big)+O_{L^\infty}(1).
\end{align*}
Since $\P_N$ is uniformly bounded on $L^\infty(\M)$, we also have
\begin{align*}
\big\|(\P_N\otimes \P_N)\big(G^\H-G^{1-\Dl})\big\|_{L^\infty}\les \|G^\H-G^{1-\Dl}\|_{L^\infty}\les 1
\end{align*}
uniformly in $N\ge 1$, due to Lemma~\ref{LEM:GreenH}. This proves \eqref{GHN1}.

As for \eqref{GHN2}, we first show that 
\begin{align}\label{PNdiff}
\big\|\P_{N_1}-\P_{N_2}\big\|_{W^{\sigma+\delta,p}\to W^{\sigma,p}}\les N_1^{-\delta}
\end{align}
for any $N_1\le N_2$, $\delta\in [0;2]$, $\sigma\in\R$, and $1\le p\le \infty$. Indeed, for any $\lambda_n$, we have from the mean value theorem
\begin{align}\label{psidiff}
\psi(N_1^{-2}\lambda_n)-\psi(N_2^{-2}\lambda_n)=(N_1^{-2}-N_2^{-2})\lambda_n\int_0^1\psi'\big((\theta N_1^{-2}+(1-\theta)N_2^{-2})\lambda_n\big)d\theta,
\end{align}
Since $\psi'$ is Schwartz, we get from \eqref{psidiff} and \eqref{K2b} that
\begin{align}\label{PNdiff2}
\big\|\P_{N_1}-\P_{N_2}\big\|_{W^{\sigma+2,p}\to W^{\sigma,p}}&\le (N_1^{-2}-N_2^{-2})\int_0^1\big\|\Delta\psi'\big((\theta N_1^{-2}+(1-\theta)N_2^{-2})\Delta\big)\big\|_{W^{\sigma+2,p}\to W^{\sigma,p}}d\theta\notag\\
&\les N_1^{-2}\int_0^1\big\|\psi'\big((\theta N_1^{-2}+(1-\theta)N_2^{-2})\Delta\big)\big\|_{W^{\sigma+2,p}\to W^{\sigma+2,p}}d\theta\\
&\les N_1^{-2},\notag
\end{align}
while using directly \eqref{K2b},
\begin{align}\label{PNdiff3}
\big\|\P_{N_1}-\P_{N_2}\big\|_{W^{\sigma,p}\to W^{\sigma,p}}\le \big\|\P_{N_1}\big\|_{W^{\sigma,p}\to W^{\sigma,p}}+\big\|\P_{N_2}\big\|_{W^{\sigma,p}\to W^{\sigma,p}} \les 1.
\end{align}
Thus \eqref{PNdiff} follows from interpolating \eqref{PNdiff2} and \eqref{PNdiff3}.

We can then write
\begin{align*}
&(\P_{N_1}\otimes \P_{N_1})G^{\H}(x,y)-(\P_{N_1}\otimes \P_{N_2})G^{\H}(x,y)\\
&=(\P_{N_1}\otimes \P_{N_1})G^{1-\Dl}(x,y)-(\P_{N_1}\otimes \P_{N_2})G^{1-\Dl}(x,y)\\
&\qquad+\big(\P_{N_1}\otimes (\P_{N_1}-\P_{N_2})\big)(G^{\H}-G^{1-\Dl})(x,y).
\end{align*}
As for the terms on the first line, we estimate them with the use of \eqref{GN2}, which is of course enough for \eqref{GHN2}. For the terms on the last line, since by Lemma~\ref{LEM:GreenH}, $G^\H-G^{1-\Dl}\in L^\infty W^{\sigma,p}$ for $4\le p <\infty$ and $\frac2p<\sigma<\frac2p+\frac12$, we have together with Sobolev embedding and \eqref{PNdiff}:
\begin{align*}
&\big\|\big(\P_{N_2}\otimes (\P_{N_1}-\P_{N_2})\big)(G^{\H}-G^{1-\Dl})\big\|_{L^\infty_{x,y}}\\
&\les \big\|\big(\P_{N_2}\otimes (\P_{N_1}-\P_{N_2})\big)(G^{\H}-G^{1-\Dl})\big\|_{L^\infty_x W^{\frac2p+\eps,p}_y}\\
&\les N_1^{-\delta}\big\|G^{\H}-G^{1-\Dl}\big\|_{L^\infty_x W^{\frac2p+\eps+\delta,p}_y}\les N_1^{-\delta}.
\end{align*}
The difference $(\P_{N_2}\otimes \P_{N_2})G^{\H}-(\P_{N_1}\otimes \P_{N_2})G^{\H}$ is estimated similarly. This finally proves \eqref{GHN2} provided that we take $\eps,\delta>0$ with $\eps+\delta<\frac12$.
\end{proof}

We have similar estimates if we use a ``regularization'' based on $\H$ instead of $\Dl$.
\begin{corollary}
Let $\psi\in\S(\R)$ be as in \eqref{PN}.\\
 \textup{(i)} There exists $C>0$ such that for any $N\in\N$ and $(x,y)\in\M\times\M\setminus\mathrm{diag}$ it holds
    \begin{align}\label{GHM1}
        \Big|(\psi(N^{-2}\H)\otimes \psi(N^{-2}\H))G^{\H}(x,y) +\frac1{2\pi}\log\big(\dg(x,y)+N^{-1}\big)\Big| \le C.
    \end{align}
    \textup{(ii)} For all $0<\dl\ll1$ there exists $c,C>0$ such that for any $1\le M_1\le M_2$ and $N\in\N^*$, it holds for any $j\in\{1;2\}$:
    \begin{align}\label{GHM2}
        &\Big\|(\P_N\psi(M_j^{-2}\H)\otimes \P_N\psi(M_j^{-2}\H))G^{\H}-(\P_N\psi(M_1^{-2}\H)\otimes \P_N\psi(M_2^{-2}\H))G^{\H}\Big\|_{L^\infty_{x,y}}\notag\\
        &\qquad\qquad\le CN^{c\delta}M_1^{-\delta}.
    \end{align}

\end{corollary}

\begin{proof}
For \eqref{GHM1}, we decompose
\begin{align*}
(\psi(N^{-2}\H)\otimes \psi(N^{-2}\H))G^{\H}(x,y)&=(\P_N\otimes\P_N)G^\H(x,y)\\
&\qquad+\big(\P_N\otimes(\psi(N^{-2}\H)-\P_N)\big)G^\H(x,y)\\
&\qquad\qquad+\big((\psi(N^{-2}\H)-\P_N)\otimes \psi(N^{-2}\H)\big)G^\H(x,y).
\end{align*}
The first term is estimated by Corollary~\eqref{COR:GHN}. As for the second term, we use the uniform boundedness of $\P_N$ on $L^\infty(\M)$ and Lemma~\ref{LEM:multipliers} below to get for $0<\eps,\kappa\ll1$
\begin{align*}
\big\|\big(\P_N\otimes(\psi(N^{-2}\H)-\P_N)\big)G^\H\big\|_{L^\infty_{x,y}}&\les \big\|\big(1\otimes(\psi(N^{-2}\H)-\P_N)\big)G^\H\big\|_{L^\infty_x W^{\frac2p+\eps,p}_y}\\
&\les \big\|G^\H\big\|_{L^\infty_x H^{\kappa+\eps}_y} \les 1.
\end{align*}
Similarly, we have by symmetry and Lemma~\ref{LEM:multipliers} again that 
\begin{align*}
&\big\|\big((\psi(N^{-2}\H)-\P_N)\otimes \psi(N^{-2}\H)\big)G^\H\big\|_{L^\infty_{x,y}}\\
&\qquad\les\big\|\big(1\otimes (\psi(N^{-2}\H)-\P_N)^2\big)G^\H\big\|_{L^\infty_x W^{\frac2p+\eps,p}_{y}}+\big\|\big(\P_N\otimes(\psi(N^{-2}\H)-\P_N)\big)G^\H\big\|_{L^\infty_{x,y}}\\
&\qquad\les \big\|\big(1\otimes (\psi(N^{-2}\H)-\P_N)\big)G^\H\big\|_{L^\infty_x H^{\kappa+\eps}_{y}}+1\\
&\qquad\les 1.
\end{align*}
This shows \eqref{GHM1}.

As for \eqref{GHM2}, we start by showing the analogue of \eqref{PNdiff}. From \eqref{psidiff}, we have
\begin{align*}
&\big\|\psi(M_1^{-2}\H)-\psi(M_2^{-2}\H)\big\|_{W^{\sigma+2+\eps,p}\to W^{\sigma,p}}\\
&\qquad\le (M_1^{-2}-M_2^{-2})\int_0^1\big\|\H\psi'\big((\theta M_1^{-1}+(1-\theta)M_2^{-2})\H\big)\big\|_{W^{\sigma+2+\eps,p}\to W^{\sigma,p}}d\theta\\
&\qquad\les (M_1^{-2}-M_2^{-2})\int_0^1\big\|\psi'\big((\theta M_1^{-1}+(1-\theta)M_2^{-2})\H\big)\big\|_{W^{\sigma+2+\eps,p}\to W^{\sigma+2,p}}d\theta\\
&\qquad\les M_1^{-2},
\end{align*} 
thanks to Lemma~\ref{LEM:multipliers}. Interpolating with
\begin{align*}
\big\|\psi(M_1^{-2}\H)-\psi(M_2^{-2}\H)\big\|_{W^{\sigma+\eps,p}\to W^{\sigma,p}} \les 1
\end{align*}
due to Lemma~\ref{LEM:multipliers} again, we get
\begin{align}\label{PMdiff}
\big\|\psi(M_1^{-2}\H)-\psi(M_2^{-2}\H)\big\|_{W^{\sigma+\delta+\eps,p}\to W^{\sigma,p}}\les M_1^{-\delta},
\end{align}
for any $\delta\in[0;2]$.

Then we estimate using Lemma~\ref{LEM:PM} and \eqref{PMdiff}:
\begin{align*}
&\Big\|(\P_N\psi(M_1^{-2}\H)\otimes \P_N\psi(M_1^{-2}\H))G^{\H}-(\P_N\psi(M_1^{-2}\H)\otimes \P_N\psi(M_2^{-2}\H))G^{\H}\Big\|_{L^\infty_{x,y}}\\
&\les \Big\|(\P_N\otimes\P_N)\big(\psi(M_1^{-2}\H)\otimes (\psi(M_1^{-2}\H)-\psi(M_2^{-2}\H)\big)G^{\H}\Big\|_{W^{\frac2p+\eps,p}_{x,y}}\\
&\les N^{-2(\frac2p+2\eps)}\Big\|\big(\psi(M_1^{-2}\H)\otimes (\psi(M_1^{-2}\H)-\psi(M_2^{-2}\H)\big)G^{\H}\Big\|_{W^{-\eps,p}_{x,y}}\\
&\les N^{2(\frac2p+2\eps)}M_1^{-\frac\eps4}\big\|G^{\H}\big\|_{W^{-\frac\eps2,p}_{x}W^{-\frac\eps4,p}_y}\\
&\les N^{2(\frac2p+2\eps)}M_1^{-\frac\eps4}.
\end{align*}
This shows \eqref{GHM2} provided that we take $p$ large enough and $\eps$ small enough.
\end{proof}

\subsection{Comparison of semi-classical multipliers}
In this part, we establish a comparison principle between Schwartz multipliers for the Anderson operator and Schwartz multipliers for the Laplace-Beltrami operator. This is similar to (A.15) in \cite{BGT}, but there the authors deal with abstract symmetric perturbations of $-\Delta$ of order 1 in $\R^d$, whereas here we deal with a perturbation of order $1+\kappa$, $0<\kappa\ll1$, with rough coefficients.

\begin{lemma}\label{LEM:multipliers}
Let $\psi\in \S(\R)$, $2\le p <\infty$ and $\frac2p-1<\sigma< \frac2p+1$, and $0<\kappa<\delta<1+\kappa$ and $\kappa-\delta \le \beta\le 2+\kappa-\delta$ such that $0\le \sigma-\frac2p+\delta\le 2$. Then there exists $C>0$ such that for any $N\ge 1$ and $f\in H^{\beta}(\M)$, it holds
\begin{align}\label{multipliers}
    \big\|\psi(N^{-2}\H)f-\psi(N^{-2}\Delta)f\big\|_{W^{\sigma,p}}\le CN^{\max(\sigma+\kappa-\frac2p-\beta,0)}\|f\|_{H^{\beta}}.
\end{align}
Moreover, if $\psi$ is compactly supported away from 0, it holds
\begin{align}\label{multipliersannulus}
    \big\|\psi(N^{-2}\H)f-\psi(N^{-2}\Delta)f\big\|_{W^{\sigma,p}}\le CN^{\sigma+\kappa-\frac2p-\beta}\|f\|_{H^{\beta}}.
\end{align}
\end{lemma}
\begin{proof}
The proof closely follows that of \cite[Proposition 2.1 and Theorem 6]{BGT}: by Helffer-Sj\"ostrand's formula, we have
\begin{align}\label{Helffer}
    \psi(N^{-2}\H)-\psi(-N^{-2}\Delta)=-\frac1\pi\int_\C \overline{\partial}\widetilde{\psi}(z)\Big((z-N^{-2}\H)^{-1}-(z+N^{-2}\Delta)^{-1}\Big)dz,
\end{align}
where $\widetilde{\psi}$ is an almost analytic extension of $\psi$. Here we take
\begin{align*}
    \widetilde{\psi}(z)=\Big(\sum_{k=0}^2\frac{\psi^{(k)}(\Re z)}{k!}(i\Im z)^k\Big)\chi(\Im z)
\end{align*}
with $\chi\in C^{\infty}_0(\R)$ satisfying $\chi\equiv 1$ near 0, so that
\begin{align*}
    \overline{\partial}\widetilde{\psi}(z)&=\frac{\widetilde{\psi}^{(2)}(\Re z)}{2}(i\Im z)^2\chi(\Im z)+\Big(\sum_{k=0}^2\frac{\widetilde{\psi}^{(k)}(\Re z)}{k!}(i\Im z)^k\Big)\chi'(\Im z),
\end{align*}
providing
\begin{equation}\label{dbarpsitilde}
    |\overline{\partial}\widetilde{\psi}(z)|\les \langle \Re z\rangle^{-10} |\Im z|^2 \mathbf{1}_{|\Im z|\les 1}.
\end{equation}
Then we decompose the right-hand side of \eqref{Helffer} as
\begin{align}
&-\frac1\pi\int_\C \overline{\partial}\widetilde{\psi}(z)\Big((z-N^{-2}\H^\sharp)^{-1}-(z+N^{-2}\Delta)^{-1}\Big)dz \Gamma^{-1}\notag\\
   &\qquad-\frac1\pi(\Gamma-1)\int_\C \overline{\partial}\widetilde{\psi}(z)(z-N^{-2}\H^\sharp)^{-1}dz\Gamma^{-1}\label{diffResolvent}\\
   &\qquad\qquad-\frac1\pi\int_\C \overline{\partial}\widetilde{\psi}(z)(z+N^{-2}\Delta)^{-1}dz(\Gamma^{-1}-1).\notag
\end{align}

We start by estimating the norm of the semi-classical resolvents in Sobolev spaces. First, for any $\nu\in (-1;2]$, since $\D^\nu=\Gamma H^\nu$, we have
\begin{align}\label{resolventH1}
    \big\|(z-N^{-2}\H^\sharp)^{-1}u\big\|_{H^\nu}&\sim \big\|(z-N^{-2}\H)^{-1}\Gamma u\big\|_{\D^\nu} \sim \Big(\sum_{n\geq 0} \lambda_n^{\nu}\frac{|\langle \Gamma u,\varphi_n\rangle|^2}{|z+N^{-2}\lambda_n|^2}\Big)^2\notag\\
    &\les |\Im z|^{-1}\|\Gamma u\|_{D^\nu},
\end{align}
and similarly, for any $\nu\in\R$,
\begin{align}\label{resolventD1}
    \big\|(z+N^{-2}\Delta)^{-1}u\big\|_{H^\nu}\les |\Im z|^{-1}\|u\|_{H^\nu},
\end{align}
uniformly in $N$. Moreover, if $u\in H^{\nu-2}(\M)$, letting $v=(z+N^{-2}\Delta)^{-1}u$, we have $(1-\Delta)v = v +N^2(zv-u)$, from which we infer
\begin{align}\label{resolventD2}
    \big\|(z+N^{-2}\Delta)^{-1}u\big\|_{H^\nu}&\les (1+N^2|z|)\big\|(z+N^{-2}\Delta)^{-1}u\big\|_{H^{\nu-2}}+N^2\|u\|_{H^{\nu-2}}\notag\\
    &\les |\Im z|^{-1}\langle z\rangle N^2\|u\|_{H^{\nu-2}}.
\end{align}
Proceeding similarly, since $v=(z-N^{-2}\H)^{-1}\Gamma u$ satisfies $\H v = zN^2 v - N^2 \Gamma u$, we have for any $\nu \in (-1;2]$:
\begin{align}\label{resolventH2}
\big\|(z-N^{-2}\H^\sharp)^{-1} u\big\|_{H^\nu}&\sim \big\|(z-N^{-2}\H)^{-1}\Gamma u\big\|_{\D^\nu}\notag\\
&\les |z|N^2\big\|(z-N^{-2}\H)^{-1}\Gamma u\big\|_{\D^{\nu-2}}+N^2\|\Gamma u\|_{\D^{\nu-2}}\\
&\les |\Im z|^{-1}\langle z\rangle N^2\|\Gamma u\|_{\D^{\nu-2}} ,\notag
\end{align}
where we used that $\Gamma$ is invertible on $H^s$, $s\in (-1;1)$ and $\Gamma H^s=\D^s$ for $s\in (-1;2]$.

Interpolation between \eqref{resolventH1} and \eqref{resolventH2} on the one hand, and between \eqref{resolventD1} and \eqref{resolventD2} on the other hand, leads to both
\begin{align}\label{resolventH3}
    \big\|(z-N^{-2}\H^\sharp)^{-1}u\big\|_{H^\nu}\les |\Im z|^{-1}\langle z\rangle^{\frac\eta2} N^\eta \|\Gamma u\|_{\D^{\nu-\eta}}
\end{align}
for any $\nu \in (-1;2]$ and $\eta\in[0;2]$, and
\begin{align}\label{resolventD3}
     \big\|(z+N^{-2}\Delta)^{-1}u\big\|_{H^\nu}\les |\Im z|^{-1}\langle z\rangle^{\frac\eta2} N^\eta \|u\|_{H^{\nu-\eta}},
\end{align}
for any $\nu\in\R$ and $\eta\in [0;2]$.

We can then estimate the terms in the right-hand side of \eqref{diffResolvent}. For the first term, we use the resolvent identity, \eqref{dbarpsitilde}, \eqref{resolventH3}-\eqref{resolventD3}, and the mapping property of $G_\xi$ and $\H\Gamma$, to get
\begin{align*}
    &\Big\|\int_\C \bar{\partial}\widetilde{\psi}(z)\Big((z-N^{-2}\H^\sharp)^{-1}-(z+N^{-2}\Delta)^{-1}\Big)dz\Gamma^{-1}f\Big\|_{W^{\sigma,p}}\\
    &\les \int_{|\Im z|\les 1} \langle z\rangle^{-10}|\Im z|^2\\
    &\qquad\times \big\|(z-N^{-2}\H^\sharp)^{-1}N^{-2}(G_\xi + (\H\Gamma \cdot)\pl X)(z+N^{-2}\Delta)^{-1}\Gamma^{-1}f\big\|_{H^{\sigma+1-\frac2p}}dz\\
   &\les N^{\sigma-\frac2p+\delta-2}\int_{|\Im z|\les 1} \langle z\rangle^{\frac\sigma2-\frac1p+\frac\delta2-10}|\Im z|\Big\{ \big\|G_\xi(z+N^{-2}\Delta)^{-1}\Gamma^{-1}f\big\|_{H^{1-\delta}}dz\\
   &\qquad\qquad+ \Big\|\big(\H\Gamma (z+N^{-2}\Delta)^{-1}\Gamma^{-1}f\big)\pl X\Big\|_{H^{1-\delta}}\Big\}dz \\
  &\les N^{\sigma-\frac2p+\delta-2}\int_{|\Im z|\les 1} \langle z\rangle^{\frac\sigma2-\frac1p+\frac\delta2-10}|\Im z|\Big\{ \big\|(z+N^{-2}\Delta)^{-1}\Gamma^{-1}f\|_{H^{2-\delta+\kappa}}\\
   &\qquad\qquad+ \big\|\H\Gamma (z+N^{-2}\Delta)^{-1}\Gamma^{-1}f\big\|_{H^{\kappa-\delta}}\Big\}dz \\
   &\les  N^{\sigma-\frac2p+\delta-2}\int_{|\Im z|\les 1}\langle z\rangle^{\frac\sigma2-\frac1p+\frac\delta2-10}|\Im z| \big\|(z+N^{-2}\Delta)^{-1}\Gamma^{-1}f\big\|_{H^{2-\delta+\kappa}}dz\\
   &\les N^{\sigma-\frac2p+\kappa-\beta}\int_{|\Im z|\les 1} \langle z\rangle^{\frac{\sigma+\kappa-\beta}{2}-\frac1p-10} \big\|\Gamma^{-1}f\big\|_{H^{\beta}}dz\\
   &\les N^{\sigma-\frac2p+\kappa-\beta}\|f\|_{H^{\beta}},
\end{align*}
where in the last step we used that $\Gamma^{-1}$ is bounded on $H^{\beta}(\M)$ by choice of $\beta$.

As for the second term in \eqref{diffResolvent}, we write again $\Gamma = 1+B$ as in \eqref{decomp-Gamma}, with $B$ bounded from $W^{\sigma-1+\kappa,p}(\M)$ to $W^{\sigma,p}(\M)$. Using that $\psi\in\S(\R)$, this term contributes as
\begin{align*}
    \Big\|B\psi(N^{-2}\H^\sharp)^{-1}\Gamma^{-1} f\Big\|_{W^{\sigma,p}} &\les \big\|\Gamma^{-1}\psi(N^{-2}\H) f\big\|_{W^{\sigma-1+\kappa,p}}\\
    &\les \big\|\Gamma^{-1}\psi(N^{-2}\H) f\big\|_{H^{\sigma+\kappa-\frac2p}}\sim\big\|\psi(N^{-2}\H) f\big\|_{\D^{\sigma+\kappa-\frac2p}}\\
    &= \Big(\sum_{n\ge 0}\psi(N^{-2}\lambda_n)\lambda_n^{\sigma+\kappa-\frac2p}|\langle f,\varphi_n\rangle|^2\Big)^\frac12\\
    &\les_L\Big(\sum_{n\ge 0}\langle N^{-2}\lambda_n\rangle^{-L}\lambda_n^{\sigma+\kappa-\frac2p}|\langle f,\varphi_n\rangle|^2\Big)^\frac12\\
 &\les\Big(\sum_{\lambda_n\les N^2}\lambda_n^{\sigma+\kappa-\frac2p}|\langle f,\varphi_n\rangle|^2\Big)^\frac12\\
 &\qquad\qquad+\Big(\sum_{\lambda_n\gg N^2}N^{2L}\lambda_n^{\sigma+\kappa-\frac2p-L}|\langle f,\varphi_n\rangle|^2\Big)^\frac12\\
  &\les\Big(\sum_{\lambda_n\les N^2}N^{2\max(\sigma+\kappa-\frac2p-\beta,0)}\lambda_n^{\beta}|\langle f,\varphi_n\rangle|^2\Big)^\frac12\\
  &\qquad\qquad+\Big(\sum_{\lambda_n\gg N^2}N^{2\max(\sigma+\kappa-\frac2p-\beta,0)}\lambda_n^{\beta}|\langle f,\varphi_n\rangle|^2\Big)^\frac12\\
  &\les N^{\max(\sigma+\kappa-\frac2p-\beta,0)}\|f\|_{\D^\beta}\sim  N^{\max(\sigma+\kappa-\frac2p-\beta,0)}\|f\|_{H^\beta}
\end{align*}
since $L\ge 0$ is arbitrary, and since $\D^\beta=H^\beta$ for the range of $\beta$ in Lemma~\ref{LEM:multipliers}.

At last, using \eqref{K2t}, the third term in \eqref{diffResolvent} is estimated similarly by
\begin{align*}
    \big\|\psi(N^{-2}\Delta)(f\pl X)\big\|_{W^{\sigma,p}}&\les N^{\max(\sigma+\kappa-\frac2p-\beta,0)} \big\|f\pl X\big\|_{B^{\beta+\frac2p-\kappa}_{p,2}}\\
    & \les N^{\max(\sigma+\kappa-\frac2p-\beta,0)}\|f\|_{B^{\beta+\frac2p-1}_{p,2}}\les N^{\max(\sigma+\kappa-\frac2p-\beta,0)}\|f\|_{H^{\beta}}.
\end{align*}
Putting everything together finally proves \eqref{multipliers}.

As for \eqref{multipliersannulus}, the only differences are that in treating the second term in the right-hand side of \eqref{diffResolvent}, the sums on $\lambda_n\les N^2$ and $\lambda_n \gg N^2$ are now reduced to a single sum on $\lambda_n\sim N^2$, for which the same bound holds with $N^{\max(\sigma+\kappa-\frac2p-\beta,0)}$ replaced by $N^{\sigma+\kappa-\frac2p-\beta}$. Similarly for the third term, using the definition of Besov spaces, that $\psi$ is compactly supported away from zero, together with the definition \eqref{funcalc} of the functional calculus, and \eqref{K2b}:
\begin{align*}
 \big\|\psi(N^{-2}\Delta)(f\pl X)\big\|_{W^{\sigma,p}}& \les  \big\|\psi(N^{-2}\Delta)(f\pl X)\big\|_{B^{\sigma}_{p,2}}\\
 &= \big\|\psi(N^{-2}\Delta)M^{\sigma}\mathbf{1}_{M\sim N}\Q_M(f\pl X)\big\|_{L^p_x\ell^2_M}\\
  &\les N^{\sigma+\kappa-\frac2p-\beta}\big\|M^{\beta+\frac2p-\kappa}\mathbf{1}_{M\sim N}\Q_M(f\pl X)\big\|_{L^p_x\ell^2_M}\\
  &\les N^{\sigma+\kappa-\frac2p-\beta}\big\|(f\pl X)\big\|_{B^{\beta+\frac2p-\kappa}_{p,2}}\\
  &\les N^{\sigma+\kappa-\frac2p-\beta}\big\|f\big\|_{B^{\beta+\frac2p-1}_{p,2}}\les N^{\sigma+\kappa-\frac2p-\beta}\big\|f\big\|_{H^\beta}.
\end{align*}
This proves \eqref{multipliersannulus}.
\end{proof}
%

As a Corollary, we will use the following uniform boundedness property of smooth projectors for $\H$ in H\"older spaces.
\begin{corollary}\label{COR:PMH}
Let $\chi\in C^\infty_0(\R)$. Then for any $\sigma\in (-1;1)$ and $0<\kappa\ll 1$ it holds
\begin{align*}
\big\|\chi(M^{-2}\H)\big\|_{\Cc^{\sigma+\kappa}\to\Cc^\sigma}\les 1,
\end{align*}
uniformly in $M\in\N^*$.
\end{corollary}
\begin{proof}
Applying Lemmas~\ref{LEM:multipliers} and \eqref{LEM:PM}, we get for $p\gg 1$:
\begin{align*}
\big\|\chi(M^{-2}\H)\big\|_{\Cc^{\sigma+\kappa}\to\Cc^\sigma}&\le \big\|\chi(M^{-2}\H)-\chi(M^{-2}\Dl)\big\|_{\Cc^{\sigma+\kappa}\to\Cc^\sigma}+\big\|\chi(M^{-2}\Dl)\big\|_{\Cc^{\sigma+\kappa}\to\Cc^\sigma}\\
&\les \big\|\chi(M^{-2}\H)-\chi(M^{-2}\Dl)\big\|_{H^{\sigma+\kappa}\to W^{\sigma+\frac2p+\frac\kappa2,p}}+M^{-\kappa}\\
&\les 1+M^{-\kappa}.
\end{align*}
\end{proof}

\appendix

\section{Analytic toolbox}\label{SEC:Appendix}

\subsection{Analysis on manifolds and function spaces}
Recall that we work on a two-dimensional closed (compact, boundaryless), connected, orientable smooth Riemannian manifold $(\M,\gm)$, where we fix the metric $\gm$ once and for all and drop it in the remaining of the paper. We set $\{\varphi_n^{\Dl}\}_{n\ge 0}\subset C^{\infty}(\M)$ to be a basis of $L^2(\M)$ consisting of eigenfunctions of $-\Dl$ associated with the eigenvalue $\ld_n^\Dl\ge 0$, assumed to be arranged in increasing order: $0\le \ld_0^\Dl <\ld_1^\Dl\le\ld_2^\Dl\le...$ As for $\ld_n^\Dl$, $n\to \infty$, we have the following asymptotic behaviour given by Weyl's law:
\begin{align}\label{Weyl}
\frac{\ld_n^\Dl}{n} \too \frac{|\M|}{4\pi} \text{ as }n\to\infty.
\end{align}

Using the above basis of $L^2(\M)$, we can expand any $u\in\D'(\M)$ as
\begin{align*}
u = \sum_{n\ge 0}\langle u,\varphi_n^\Dl\rangle_\gm \varphi_n^\Dl,
\end{align*}
where $\langle\cdot,\cdot\rangle_\gm$ denotes the distributional pairing $\D'(\M)\times\D(\M)\to \R$ which coincides with the usual inner product in $L^2(\M)$ for distributions which are regular enough. For any $s\in\R$ and $1\le p\le \infty$, we thus define the Sobolev spaces
\begin{align*}
W^{s,p}(\M,\gm) \deff\Big\{u\in\D'(\M),~\|u\|_{W^{s,p}(\M)} <\infty\Big\}
\end{align*}
where
\begin{align*}
\|u\|_{W^{s,p}(\M)}  \deff \big\|(1-\Dl)^{\frac{s}2}u\big\|_{L^p(\M)} = \Big\|\sum_{n\ge 0}(1+\ld_n^\Dl)^s\langle u,\varphi_n^\Dl\rangle_\gm \varphi_n^\Dl\Big\|_{L^p(\M)}.
\end{align*}
When $p=2$ we write $H^s(\M)\deff W^{s,2}(\M)$.

Next, recall that for any self-adjoint elliptic operator $\AA$ on $L^2(\M)$ with discrete spectrum $\{\ld_n\}$ and orthonormal basis of eigenfunctions $\{\varphi_n\}$, the functional calculus of $\AA$ is defined for any $\psi\in L^{\infty}(\R)$ by
\begin{align}\label{funcalc}
\psi\big(\AA\big)u = \sum_{n\ge 0}\psi(\ld_n^\AA)\langle u,\varphi_n^\AA\rangle \varphi_n^\AA,
\end{align}
for all $u\in C^{\infty}(\M)$. This in particular allows us to define the more general class of Besov spaces on $\M$ associated with $\Dl$. First, using the functional calculus, we can define the Littlewood-Paley projectors $\Q_M$ for a dyadic integer $M\in 2^{\Z_{\ge-1}}\deff \{0,1,2,4,...\}$ as 
\begin{align}\label{QM}
\Q_M = \begin{cases}
\psi_0\big(-(2M)^{-2}\Dl\big)-\psi_0\big(-M^{-2}\Dl\big),~~M\ge 1,\\
\psi_0\big(-\Dl\big),~~M=0,
\end{cases}
\end{align}
where $\psi_0\in C^{\infty}_0(\R)$ is non-negative and such that $\supp\psi_0\subset [-1,1]$ and $\psi_0\equiv 1$ on $[-\frac12,\frac12]$. With the inhomogeneous dyadic partition of unity $\{\Q_M\}_{M\in 2^{\Z_{-1}}}$, we can then define the Besov norms for $p,q \in [1,\infty]$ and $s\in \R$,
\begin{align}\label{Besov norm}
\| f \|_{\Bb^s_{p,q}(\M)} \deff \Big\|  \jb{M}^{qs} \| \Q_M  f \|_{L^p(\M)} \Big\|_{\ell^q_M} = \Big(\sum_{M\in 2^{\Z_{-1}}}\jb{M}^{qs}\big\|\Q_M f \big\|_{L^p(\M)}^q\Big)^{\frac1q}.
\end{align}
for any smooth function $f$. The last sum is of course understood to be the supremum for $q=\infty$.
\begin{definition}\label{def}\rm Fix $p,q \in [1,\infty]$ and $s \in \R$. We define the Besov space $\Bb^s_{p,q}(\M)$ to be the completion of $C^{\infty}(\M)$ for the norm $ \| \cdot \|_{\Bb^s_{p,q}(\M)} $. In particular, this defines the H\"older spaces $\Cc^s(\M) = \Bb^s_{\infty, \infty}(\M)$ for any $s\in \R$. 
\label{Besov}
\end{definition}
These function spaces are the natural generalization of the usual Besov spaces on $\R^d$ to the context of closed manifolds. Let us recall the following characterization of these spaces from \cite[Proposition 2.5]{ORT}
\begin{lemma}\label{LEM:Beloc}
Let $(U,V,\kk)$ be a coordinate patch and $\chi\in C^{\infty}_0(V)$. For any $s\in\R$ and $1\leq p,r\leq \infty$, there exist $c,C>0$ such that for any $u\in C^{\infty}(\M)$,
\begin{align*}
c\|\chi u\|_{B^s_{p,r}(\M)}\leq  \|\kk^{\star}(\chi u)\|_{B^{s}_{p,r}(\R^2)}\leq C\|u\|_{B^s_{p,r}(\M)}.
\end{align*}
\end{lemma}
This allows to transfer the usual linear or nonlinear estimates on Besov spaces on $\R^d$ to the case of Besov spaces on $\M$; see for example Lemma~\ref{LEM:Besov} below, which is taken from \cite[Corollary 2.7]{ORT}.
\begin{lemma}\label{LEM:Besov}
Let $B^s_{p,q}(\M)$ be the Besov spaces defined above. Then the following properties hold.\\
\textup{(i)} For any $s\in\R$ we have $B^s_{2,2}(\M)=H^s(\M)$, and more generally for any $2\le p <\infty$ and $\eps>0$ we have
\begin{align*}
\|u\|_{B^s_{p,\infty}(\M)}\les \|u\|_{W^{s,p}(\M)}\les \|u\|_{B^s_{p,2}(\M)} \les \|u\|_{B^{s+\eps}_{p,\infty}(\M)}.
\end{align*}
\textup{(ii)} Let $s\in\R$ and $1\leq p_1\leq p_2\leq \infty$ and $q_1,q_2\in [1,\infty]$. Then for any $f\in B^s_{p_1,q}(\M)$ we have
\begin{align*}
\|f\|_{B^{s-2\big(\frac{1}{p_1}-\frac{1}{p_2}\big)}_{p_2,q}(\M)}\les \|f\|_{B^s_{p_1,q}(\M)}.
\end{align*}
\textup{(iii)} Let $\alpha,\beta\in \R$ with $\alpha+\beta>0$, and $p_1,p_2,q_1,q_2\in [1,\infty]$ with
\begin{align*}
\frac1p = \frac{1}{p_1}+\frac{1}{p_2}\qquad\text{ and }\qquad\frac1q=\frac{1}{q_1}+\frac{1}{q_2}.
\end{align*}
 Then for any $f\in B^{\alpha}_{p_1,q_1}(\M)$ and $g\in B^{\beta}_{p_2,q_2}(\M)$, we have $fg\in B^{\alpha\wedge \beta}_{p,q}(\M)$, and it holds:
 \begin{itemize}
 \item If $\alpha\wedge \beta<0$, then
 \begin{align*}
\|fg\|_{B^{\alpha\wedge\beta}_{p,q}(\M)}\les \|f\|_{B^{\alpha}_{p_1,q_1}(\M)}\|g\|_{B^{\beta}_{p_2,q_2}(\M)}.
\end{align*}
\item If $\alpha\wedge \beta >0$, then
 \begin{align*}
\|fg\|_{B^{\alpha\wedge\beta}_{p,q}(\M)}\les \|f\|_{B^{\alpha\wedge\beta}_{p_1,q_1}(\M)}\|g\|_{B^{\alpha\wedge\beta}_{p_2,q_2}(\M)}.
\end{align*}
 \end{itemize}
\end{lemma}

We will also make use of the following fractional Leibniz rule.
\begin{lemma}\label{LEM:prodHugo}
Let $r,s<1$ with $r+s>0$, and $t=r+s-1$. Then there exists $C>0$ such that for any $u\in H^s(\M)$ and $v\in H^r(\M)$, we have the fractional Leibniz rule
    $$
    \|uv\|_{H^t}\le C\|u\|_{H^s}\|v\|_{H^r}.
    $$
 \end{lemma}
    \begin{proof}
    The corresponding estimate in $\R^d$ is proved in \cite[Corollary 2.1]{Tambaca}. Then the one on $\M$ follows from the one in $\R^d$ through the use of a finite partition of unity and Lemma~\ref{LEM:Beloc}; see the proof of Lemma~\ref{LEM:Besov}~(iii) in \cite{ORT} for details.
    \end{proof}

\subsection{Schwartz multipliers of the Laplace-Beltrami operator}
We will need to estimate the action of semi-classical multipliers on $\M$ with symbol in $\S(\R)$. We start by recalling the following universal bound on their kernel from \cite[Lemma 2.5]{ORTW}.
\begin{lemma}\label{LEM:PM}
Let $\psi\in\S(\R)$, and for any $h\in (0,1]$ define the kernel
\begin{align}\label{K1}
\K_h(x,y)\deff \sum_{n\ge 0}\psi\big(h^{2}\ld_n^\Dl\big)\varphi_n(x)\varphi_n(y)
\end{align}  
Then for any $L\ge 0$, there exists $C>0$ such that for any $h\in (0,1]$ and $x,y\in\M$ it holds
\begin{align}\label{K2}
\big|\K_h(x,y)\big|\le C h^{-2}\jb{h^{-1}\dg(x,y)}^{-L},
\end{align}
where $\dg$ is the geodesic distance on $\M$. In particular, it holds
\begin{align}\label{K2b}
\Big\|\psi(-h^{2}\Dl)\Big\|_{L^p(\M)\to L^q(\M)}\les h^{-2(\frac1p-\frac1q)}
\end{align}
and 
\begin{align}\label{K2t}
\Big\|\psi(-h^{2}\Dl)\Big\|_{B^{\sigma_1+\sigma_2}_{p,q}(\M)\to B^{\sigma_1}_{p,q}(\M)}\les h^{-\sigma_2}
\end{align}
for any $h\in (0,1]$ and $1\le p,q \le\infty$ and $\sigma_1\in\R$, $\sigma_2\ge 0$.
\end{lemma}
\begin{proof}
Estimates \eqref{K2} and \eqref{K2b} are proved in \cite[Lemma 2.5]{ORTW}. As for \eqref{K2t}, it is a straightforward adaptation of the proof of \eqref{Schauder} below given in \cite[Lemma 2.6]{ORTW}, since there only the fact that $e^{t\Dl}$ is a Schwartz multiplier was used, but not the precise form of its symbol; see (2.13) in \cite{ORTW}.
\end{proof}

As a corollary of the estimates above, we can then investigate the behaviour of the solutions to the linear heat equation on $\M$. Indeed, we can use the eigenfunctions expansion to represent the solution of the heat equation
\begin{align*}
\begin{cases}
\dt u -\Dl u = 0,\\
u(0) = u_0\in \D'(\M)
\end{cases}
~~(t,x)\in \R_+\times\M,
\end{align*}
as the distribution
\begin{align*}
u(t) = e^{t\Dl}u_0 = \sum_{n\ge 0}e^{-t\ld_n^\Dl} \langle u_0,\varphi_n\rangle\varphi_n.
\end{align*}
The well-known heat kernel is then the kernel of the above propagator, defined as
\begin{align}\label{heatker}
p_t^\Dl(x,y)\deff \sum_{n\ge 0}e^{-t\ld_n^\Dl}\varphi_n(x)\varphi_n(y).
\end{align}
As a corollary to Lemma~\ref{LEM:PM}, we have the following bounds; see \cite[Lemma 2.6]{ORTW} for a proof.
\begin{lemma}\label{LEM:heatker}
\textup{(i)} For any $L\ge 0$, there exists $C>0$ such that for any $0<t\le1$ and $x,y\in\M$ it holds
\begin{align}\label{heatkerbound}
|p_t^\Dl(x,y)|\le C t^{-1}\langle t^{-\frac12}\dg(x,y)\rangle^{-L}.
\end{align} 
\textup{(ii)} For any $1\le p\le q\le \infty$, any $1\le r\le\infty$ and $\al_1,\al_2\in\R$ with $\al_1\le \al_2$, there exists $C>0$ such that for any $0<t\le 1$ and any $u\in B^{\al_1}_{p,r}(\M)$, we have \textup{Schauder's estimate}
\begin{align}\label{Schauder}
\big\|e^{t\Dl}u\big\|_{B^{\al_2}_{q,r}(\M)}\le C t^{-\frac{\al_2-\al_1}2-(\frac1p-\frac1q)}\big\|u\big\|_{B^{\al_1}_{p,r}(\M)}.
\end{align}
\end{lemma}

We will also use of the following bound on the kernel of fractional antiderivatives on $\M$.
\begin{lemma}\label{LEM:Gsigma}
For $\sigma\in (0;2)$, let $G_\sigma$ be the distributional kernel of $(1-\Dl)^{-\frac{\sigma}2}$. Then $G_\sigma$ is non-negative distribution, smooth away from the diagonal in $\M^2$, and we have the estimate
\begin{align}\label{Gsigma}
G_\sigma(x,y)\les \dg(x,y)^{\sigma-2}.
\end{align} 
\end{lemma}
\begin{proof}
Using the eigenfunction expansion, we can write
\begin{align*}
G_\sigma(x,y) &= \sum_{n\ge 0}\frac{\varphi_n(x)\varphi_n(y)}{(1+\ld_n^\Dl)^\sigma} = \sum_{n\ge 0}\varphi_n(x)\varphi_n(y)\Gamma\Big(\frac{\sigma}2\Big)^{-1}\int_0^{\infty} t^{\frac{\sigma}2-1}e^{-t(1+\ld_n^\Dl)}dt\\
&= \Gamma\Big(\frac{\sigma}2\Big)^{-1}\int_0^{\infty} t^{\frac{\sigma}2-1}e^{-t}p_t^\Dl(x,y)dt,
\end{align*}
where $\Gamma$ is the Gamma function, and the equality holds in the sense of distributions on $\M\times\M$. Note that the last integral converges since we assumed $\sigma>0$. This shows that $G_\sigma$ is a non-negative distribution, smooth away from the diagonal, and from the estimate \eqref{heatkerbound} on the heat kernel, we get for $0<\dg(x,y)\le 1$:
\begin{align*}
G_\sigma(x,y)&\les \int_0^{1} t^{\frac{\sigma}{2}-2}\jb{t^{-\frac12}\dg(x,y)}^{-10}dt + 1\\
 &\les \int_0^{\dg(x,y)^2}t^{\frac{\sigma}{2}+3}\dg(x,y)^{-10}+\int_{\dg(x,y)^2}^1t^{\frac{\sigma}{2}-2} +1\\
 &\les \dg(x,y)^{\sigma-2}.
\end{align*}
This shows \eqref{Gsigma}.
\end{proof}

\vspace{2cm}

\noindent \textcolor{gray}{$\bullet$} H. Eulry -- Univ Rennes, CNRS, IRMAR - UMR 6625, F- 35000 Rennes, France.\\
{\it E-mail}: hugo.eulry@ens-rennes.fr

\smallskip

\noindent \textcolor{gray}{$\bullet$} A. Mouzard -- CNRS \& Department of Mathematics and Applications, ENS Paris, 45 rue d'Ulm, 75005 Paris, France.\\
{\it E-mail}: antoine.mouzard@math.cnrs.fr

\smallskip

\noindent \textcolor{gray}{$\bullet$} T. Robert -- Université de Lorraine, CNRS, IECL, F-54000 Nancy, France\\
{\it E-mail}: tristan.robert@univ-lorraine.fr

\end{document}